\documentclass[12pt,reqno]{amsart}

\usepackage{lipsum}

\usepackage{graphicx}
\usepackage{array}
\usepackage{amssymb}
\usepackage{hyperref}
\usepackage{amsthm}
\usepackage{amsfonts}
\usepackage{amsmath}
\usepackage{bm}
\usepackage{mathrsfs}
\usepackage[all]{xy}
\usepackage{color}
\usepackage{subfigure}
\usepackage{tikz, tikz-cd}
\usepackage{enumerate}
\usepackage{anysize}
\usepackage{amscd}
\usepackage[letterpaper]{geometry}
\usepackage{geometry}
\usepackage{multirow}
\usepackage{array}
\usepackage{booktabs}

\newtheorem{theorem}{Theorem}[section]
\newtheorem{proposition}[theorem]{Proposition}

\newtheorem{lemma}[theorem]{Lemma}

\newtheorem{corollary}[theorem]{Corollary}

\theoremstyle{remark}

\theoremstyle{definition}
\newtheorem{definition}[theorem]{Definition}
\newtheorem{remark}[theorem]{Remark}

\numberwithin{equation}{section}
\numberwithin{figure}{section}
\numberwithin{table}{section}

\newcommand{\bx}{\bm{x}}
\newcommand{\dC}{\mathbb{C}}
\newcommand{\dD}{\mathbb{D}}
\newcommand{\dN}{\mathbb{N}}
\newcommand{\dR}{\mathbb{R}}
\newcommand{\ee}{\mathfrak{e}}

\renewcommand{\i}{\sqrt{-1}}
\newcommand{\dT}{\mathbb{T}}
\newcommand{\dZ}{\mathbb{Z}}

\newcommand{\bo}{\bm{0}}
\newcommand{\by}{\bm{y}}
\newcommand{\M}{\mathcal{M}}

\newcommand{\fd}{\mathfrak{d}}
\newcommand{\fr}{\mathfrak{r}}
\newcommand{\fs}{\mathfrak{s}}
\newcommand{\W}{\Omega}
\newcommand{\RR}{\mathbb{R}}
\newcommand{\ZZ}{\mathbb{Z}}
\newcommand{\PP}{\mathbb{P}}
\newcommand{\p}{\partial}
\newcommand{\bp}{\bar{\partial}}
\newcommand{\fF}{\mathfrak{F}}
\newcommand{\fJ}{\mathfrak{J}}
\newcommand{\cF}{\mathcal{F}}
\newcommand{\cG}{\mathcal{G}}
\newcommand{\cJ}{\mathcal{J}}
\newcommand{\cK}{\mathcal{K}}
\newcommand{\cR}{\mathcal{R}}
\newcommand{\cS}{\mathcal{S}}
\newcommand{\sq}{\sqrt{-1}}

\newcommand{\cC}{\mathcal{C}}
\newcommand{\cO}{\mathcal{O}}
\newcommand{\CC}{\mathbb{C}}
\newcommand{\can}{\rm{can}}

\newcommand{\punsec}{\mathring{\Sec}}

\DeclareMathOperator{\ALG}{ALG}

\DeclareMathOperator{\Area}{Area}

\DeclareMathOperator{\diam}{Diam}
\DeclareMathOperator{\dvol}{dvol}
\DeclareMathOperator{\A}{A}
\DeclareMathOperator{\D}{D}
\DeclareMathOperator{\E}{E}

\DeclareMathOperator{\FF}{flat}

\DeclareMathOperator{\Id}{Id}
\DeclareMathOperator{\I}{I}
\DeclareMathOperator{\II}{II}
\DeclareMathOperator{\III}{III}
\DeclareMathOperator{\IV}{IV}

\DeclareMathOperator{\Image}{Image}

\DeclareMathOperator{\Isom}{Isom}

\DeclareMathOperator{\Ker}{Ker}

\DeclareMathOperator{\pr}{pr}

\DeclareMathOperator{\rank}{rank}
\DeclareMathOperator{\Ric}{Ric}
\DeclareMathOperator{\Res}{Res}
\DeclareMathOperator{\Sec}{Sec}
\DeclareMathOperator{\SF}{sf}

\DeclareMathOperator{\Ima}{Im}

\DeclareMathOperator{\Rea}{Re}

\DeclareMathOperator{\Rm}{Rm}

\DeclareMathOperator{\SL}{SL}

\DeclareMathOperator{\EH}{EH}
\DeclareMathOperator{\U}{U}

\DeclareMathOperator{\Vol}{Vol}

\allowdisplaybreaks[3]

\begin{document}

 \title{Collapsing Ricci-flat metrics on elliptic K3 surfaces}

 \author{Gao Chen} \thanks{The first author is supported by NSF Grant DMS-1638352 and a grant from S. S. Chern Foundation for Mathematical Research. The second author is supported by NSF Grant DMS-1811096. The third author is supported by  NSF Grant DMS-1906265. }
  \address{Department of Mathematics, University of Wisconsin-Madison, Madison, WI, 53706}
 \email{gchen233@wisc.edu}
 \author{Jeff Viaclovsky}
 \address{Department of Mathematics, University of California, Irvine, CA 92697}
 \email{jviaclov@uci.edu}
 \author{Ruobing Zhang}
\address{Department of Mathematics, Stony Brook University, Stony Brook, NY, 11794}
\email{ruobing.zhang@stonybrook.edu}

\begin{abstract}
For any elliptic K3 surface $\fF: \cK \rightarrow \PP^1$, we construct a family of collapsing Ricci-flat K\"ahler metrics such that curvatures are uniformly bounded away from singular fibers,
and which Gromov-Hausdorff limit to $\PP^1$ equipped with the McLean metric. There are well-known examples of this type of collapsing, but the key point of our construction is that we can additionally give a precise description of the metric degeneration near each type of singular fiber, without any restriction on the types of singular fibers.
\end{abstract}

\date{}

\maketitle

\setcounter{tocdepth}{1}
\tableofcontents

\section{Introduction}
\subsection{Background and main results}

A K3 surface, by definition, is a compact complex surface with trivial fundamental group and zero first Chern class.
Since all complex K3 surfaces are K\"ahler \cite{Siu}, it is a consequence of Yau's resolution to the Calabi conjecture \cite{Yau} that every K3 surface admits Ricci-flat K\"ahler metrics called {\textit{Calabi-Yau metrics}}.
 Our particular interest in this paper is to study an {\it elliptic} K3 surface $\fF : \cK \rightarrow \PP^1$.
It is well-known that there are no exceptional curves on $\cK$, and no fiber over the base is multiple \cite[Section~11.1]{Huybrechts}. The generic fiber is an elliptic curve,
and there are singular fibers which can be of type $\I_0^*$, $\II$, $\III$, $\IV$, $\II^*$, $\III^*$, $\IV^*$, $\I_{\nu}$ and $\I_{\nu}^*$ for $\nu\in\dZ_+$, over a finite set $\mathcal{S} \subset \mathbb{P}^1$  \cite{Kodaira1963}.

For an elliptic K3 surface which has $24$ singular fibers of Type $\I_1$ in Kodaira's list,
Gross and Wilson in \cite{GW} constructed a family of Calabi-Yau metrics with bounded diameters which are collapsing to the base $\PP^1$
in the Gromov-Hausdorff sense, that is,
\begin{equation}
(\cK, g_\delta) \xrightarrow{GH} (\PP^1,d_{ML}), \quad \text{as}\ \delta\to0,
\end{equation}
where $d_{ML}$ is the McLean metric on $\PP^1$,
see Definition~\ref{d:McLean}.
Moreover, away from 24 singular
fibers, $g_\delta$ are collapsing with uniformly bounded curvatures.
This is consistent with the general theory of the degeneration of Einstein
metrics in dimension four in that the sequence collapses
with bounded curvature away from finitely many singular points in the Gromov-Hausdorff limit \cite{CheegerTian},
and the limit is a Riemannian orbifold away from the singular points \cite{NT1}.
However, the general theory does not provide a description of the degeneration near the singular fibers.

The case of $24$ fibers of type $\I_1$ is the generic case, but there are many interesting K3 surfaces with non-generic configurations of singular fibers. In \cite{GTZ}, the construction of Gross-Wilson was extended to  general elliptic K3 surfaces, see also the recent work~\cite{OdakaOshima}.
Namely, families of Calabi-Yau metrics were constructed which Gromov-Hausdorff limit $(\PP^1,d_{ML})$ with uniform curvature estimates away from finite singular points.
Our primary goal in this paper is to describe the precise nature of the degeneration near the singular fibers, by describing all possible {\textit{canonical bubble limits}}, see Definition~\ref{d:canonical-bubble}.

We next recall some details regarding complete non-compact hyperk\"ahler four-manifolds.
Under the curvature decay condition $|\Rm|\leq Cr^{-2-\epsilon}$ for some $\epsilon>0$, complete non-compact hyperk\"ahler 4-manifolds were classified in a series of works \cite{Kronheimer, Minerbe, CCI, CCII, CCIII}. Under this curvature decay condition, the volume growth rates must be $O(r^4)$, $O(r^3)$, $O(r^2)$ or $O(r)$. They are called ALE (asymptotically locally Euclidean), ALF (asymptotically locally flat), ALG, or ALH (``G" and ``H" are the letters after ``E" and ``F") respectively. However, there are also complete non-compact hyperk\"ahler 4-manifolds with curvature decay rate exactly $O(r^{-2})$ and volume growth rate $O(r^{4/3})$ as well as complete non-compact hyperk\"ahler 4-manifolds with other curvature decay rates \cite{AKL, TianYau, Hein}.
The hyperk\"ahler 4-manifolds with volume growth rates $O(r^4)$, $O(r^3)$, $O(r^{4/3})$ and $O(r)$ have been realized as the bubbles of degenerating Calabi-Yau metrics on the K3 surfaces \cite{LeBrun-Singer, Donaldson-kummer, Foscolo, CCII, HSVZ}. The deepest bubbles in our construction are ALE Eguchi-Hanson metrics (see \cite{EguchiHanson}), ALF Taub-NUT metrics (see \cite{Hawking}), as well as a certain class of
ALG hyperk\"ahler metrics, which are {\textit{isotrivial}} and are ALG of order at least~$2$, see Subsection~\ref{ss:ALG}.
The main result in this paper is the following.
\begin{theorem}\label{t:main-theorem} For any elliptic K3 surface
 $\fF:\cK\to\PP^1$, with singular fibers over the finite set $\mathcal{S} \subset \mathbb{P}^1$, there are a
family of Ricci-flat K\"ahler metrics $g_\delta$ on $\cK$ with $\diam_{g_\delta}(\cK)=1$ and $\Vol_{g_\delta}(\cK)\sim\delta^2$
 such that
\begin{equation}
(\cK,g_\delta) \xrightarrow{GH} (\PP^1,d_{ML}), \quad \text{as}\ \delta\to0,
\end{equation}
 where $d_{ML}$ is the McLean metric.
Moreover, the following properties hold:
\begin{enumerate}
\item For any $p\in\PP^1\setminus\cS$, the fiber
$\fF^{-1}(p)$ is regular and homeomorphic to $\dT^2$
with $\Area(\fF^{-1}(p))\sim\delta^2$.

\item Curvatures are uniformly bounded away from singular fibers, while curvature is unbounded in a
neighborhood of any singular fiber.

\item Near singular fibers with finite monodromy, rescalings of the metrics converge to ALG hyperk\"ahler
 metrics.

\item
Near singular fibers of type $\I_\nu, \nu \geq 1$,
there are $\nu$ copies of Taub-NUT metrics which
occur as rescaling limits.

\item Near singular fibers of type $\I_{\nu}^*, \nu \geq 1$,
 $\nu$ copies of Taub-NUT metrics plus $4$ Eguchi-Hanson metrics occur as rescaling limits.
\end{enumerate}
\end{theorem}

\begin{remark} Theorem \ref{t:main-theorem} partially answers \cite[Problem~1.11]{Hein}. We note also that Theorem \ref{t:main-theorem} only states the non-collapsing bubble limits; in this case the pointed Gromov-Hausdorff convergence is equivalent to pointed $C^{\infty}$-convergence.  However,
there are also numerous collapsing bubble limits, and we refer the reader to Section~\ref{s:metric-geometry} for the complete description of all possibilities.
\end{remark}

\subsection{Outline of the proof}
We next give an outline of the main steps involved in the proof of Theorem~\ref{t:main-theorem}.
For a general elliptic K3 surface $\fF:\cK\to \PP^1$, there does not exist a global holomorphic section. However, Kodaira \cite{Kodaira1963} proved that $\fF:\cK\to \PP^1$ determines an elliptic surface $\fJ:\cJ\to \PP^1$ with a global holomorphic section, which is called the Jacobian of the original elliptic K3 surface. The Jacobian $\cJ$ (called the basic member by Kodaira) has the same functional and homological invariants as the original surface $\cK$.
In Section~\ref{s:semi-flat-metrics}, we will briefly review a construction of Greene-Shapiro-Vafa-Yau (\cite{GSVY}) which gives explicit hyperk\"ahler metrics $g_\delta^{\SF}$ on the regular region of $\cJ$ such that the area of each $\dT^2$-fiber equals $\delta^2$. These metrics are also known as semi-flat metrics, and their construction relies on the existence of a holomorphic section.
In Section~\ref{s:elliptic-fibration}, we will use a global diffeomorphism between the original elliptic K3 surface $\cK$ and its Jacobian $\cJ$ so that the semi-flat metrics $g_\delta^{\SF}$ can be naturally translated to a hyperk\"ahler metric $g_\delta^A$ on the regular region of the original K3 surface $\cK$.

The above procedure yields collapsing hyperk\"ahler metrics
on the regular region of the original elliptic K3 surface
$\fF:\cK\to \PP^1$. These semi-flat metrics are singular at the singular fibers; the next goal is to replace $g_\delta^A$ in the singular regions with smooth metrics by a gluing procedure.
In Section~\ref{s:infinite-monodromy} we will do this in the $I_{\nu}$ and
$I_{\nu}^*$ cases. In the $I_{\nu}$ case, this is done by gluing in a generalization of the Ooguri-Vafa metric, which we call a multi-Ooguri-Vafa metric. In the
$I_{\nu}^*$ case, we glue in a resolution of a $\ZZ_2$ quotient of a multi-Ooguri-Vafa metric, which is inspired by Kodaria's work \cite[Section~8]{Kodaira1963}. We note that the resolution of the infinite monodromy fibers can be done with respect to a fixed complex structure;
the resulting K\"ahler form will be denoted by $\omega_\delta^B$, with associated metric $g_{\delta}^B$.

 The metric  $g_\delta^B$ will still have singularities near the fibers with finite monodromy.  In Section~\ref{s:finite-monodromy}, near these singular fibers, we will glue in a certain class of
ALG hyperk\"ahler 4-manifolds. Namely, we require that the ALG manifold is {\textit{isotrivial}}, that is, the functional invariant is constant, and furthermore, that the metric is ALG of order at least~$2$. After this procedure,  we will obtain a family of ``approximately'' Calabi-Yau metrics on $\cK$ which are collapsing to the McLean metric on $\PP^1$, and denoted by $g_\delta^C$.
An important note is that this step cannot be done preserving the original complex structure
(because $\cK$ is not isotrivial), so in this step, the $2$-form $\omega_\delta^C$ is only symplectic, but not necessarily K\"ahler. Therefore, we have to use the {\it definite triples} originated in \cite{Donaldson}, to find out the approximately hyperk\"ahler metric $g_\delta^C$. This technique was also used, for instance, in  \cite{CCIII, FineLotaySinger, Foscolo, HSVZ}.
To explain this, let $(M^4,\dvol_0)$ be an oriented $4$-manifold with a fixed volume form $\dvol_0$. A triple
\begin{align}
\bm{\omega}=(\omega_1,\omega_2,\omega_3)\in\Omega^2(M^4)\otimes\dR^3
\end{align}
is called a {\textit{definite triple}} if the matrix $Q=(Q_{ij})$ defined by
\begin{equation}
\frac{1}{2}\omega_i\wedge\omega_j=Q_{ij}\dvol_0
\end{equation}
is positive definite. The triple is called {\textit{closed}} if $d \omega_i = 0$ for $i = 1, 2,3$.
For a definite triple $\bm{\omega}$ on  $(M^4,\dvol_0)$, let us define the associated renormalized volume form and coefficient matrix by
\begin{align}
\dvol_{\bm{\omega}} & \equiv  (\det(Q))^{\frac{1}{3}} \dvol_0, \\
Q_{\bm{\omega}} & \equiv (\det(Q))^{-\frac{1}{3}}Q. \label{e:renormalized-coefficients}
\end{align}
Immediately, $\det(Q_{\bm{\omega}})=1$ and $Q_{\bm{\omega}}$
is independent of the choice of $\dvol_0$.
\begin{definition} Given an oriented $4$-manifold $M^4$, a closed definite triple $\bm{\omega}=(\omega_1,\omega_2,\omega_3)\in\Omega^2(M^4)\otimes\dR^3$ is called {\it hyperk\"ahler} if
	$Q_{\bm{\omega}}=\Id$.
	Equivalently,
	\begin{equation}
	\frac{1}{2}\omega_i\wedge\omega_j =\frac{1}{6}\delta_{ij}(\omega_1^2+\omega_2^2+\omega_3^2),
	\end{equation}
	for every $1\leq i\leq j\leq 3$.
\end{definition}

\begin{remark}
	\label{r:triple}
    A definite triple $\bm{\omega}$ induces a Riemannian metric $g_{\bm{\omega}}$ such that each $\omega_j$ is self-dual with respect to $g_{\bm{\omega}}$ and the volume form of $g_{\bm{\omega}}$ is $\dvol_{\bm{\omega}}$. By \cite[Lemma~6.8]{Hitchin}, the metric $g_{\bm{\omega}}$ is hyperk\"ahler if and only if
	$\bm{\omega}=(\omega_1,\omega_2,\omega_3)$ is a hyperk\"ahler triple.
	In this case, $\omega_2 + \sq\omega_3$  is a holomorphic $2$-form with respect to the complex structure determined by $\omega_1$.
\end{remark}

In our construction we will take $\omega_2 + \i \omega_3 = \delta \Omega_{\cK}$,
where $\Omega_{\cK}$ is a fixed holomorphic $2$-form on $\cK$. This together with $\omega_\delta^C$ is a definite triple and determines a metric $g_\delta^C$. Section~\ref{s:metric-geometry} will then focus on the bubbling analysis of $g_\delta^C$.
We will quantitatively analyze the {\it regularity scales}
around each type of singular fibers, which will be achieved by explicitly classifying all possible bubble limits on $\cK$.
This leads us to canonically define a class of weighted H\"older spaces, which will be the key to our perturbative analysis.

Section~\ref{s:Liouville} will be devoted to the proof of vanishing theorems on various bubble limits, which we will refer to as {\textit{Liouville theorems}}.
In Section \ref{s:proof-of-main-theorem}, we will carry out the perturbation
of our family of approximately hyperk\"ahler metrics $g_\delta^C$ to a family of Calabi-Yau metrics $g_\delta^D$,
using the implicit function theorem (Lemma~\ref{l:implicit-function}).
To carry out the perturbation, in Proposition~\ref{p:injectivity-for-L} we will establish uniform estimates for the linearized operator in the geometrically canonical weighted spaces defined in Section~\ref{s:metric-geometry}. The proof will be based on contradiction arguments and bubbling analysis, which will reduce the proof to the various Liouville theorems on each type of bubble limits.

 We will end with some remarks in Section~\ref{s:moduli-space}. First, we will count the parameters involved in our construction, and show that they add up to the expected dimension.
Then we will also describe some other possible bubbles which can arise by other choices of the parameters involved in our construction.

\subsection{Acknowledgements}
The authors would like to thank Hans-Joachim Hein and Song Sun for many helpful suggestions and stimulating discussions.

\section{Semi-flat metric on a Jacobian K3}

\label{s:semi-flat-metrics}

In this section, we describe the semi-flat metric on an elliptic K3 surface with holomorphic section. It was originally constructed by Greene-Shapere-Vafa-Yau in \cite{GSVY}. We begin with a brief review of elliptic K3 surfaces.

\subsection{Elliptic K3 surfaces}
\label{ss:elliptick3}
To begin with, we review Kodaira's work on elliptic surfaces in \cite{Kodaira1963}. Recall that for each elliptic curve $\cC=\mathbb{C}/(\mathbb{Z}\tau_1\oplus\mathbb{Z}\tau_2)$, we can view the number $\varrho=\tau_2/\tau_1\in\mathbb{H}/\SL(2,\mathbb{Z})$, where $\mathbb{H}\equiv\{\tau\in\dC|\Ima\tau>0\}$ is the upper half plane. Let  $\fF: \cK \rightarrow \PP^1$ be an elliptic K3 surface with a finite singular set $\mathcal{S}\subset\PP^1$. Then $\varrho=\tau_2/\tau_1$ is a multi-valued holomorphic function on $\PP^1\setminus\mathcal{S}$. Recall that the $j$-invariant maps $\varrho\in\mathbb{H}/\SL(2,\mathbb{Z})$ to $j(\varrho)\in\mathbb{C}$, so the $j$-invariant is a holomorphic function on $\PP^1\setminus\mathcal{S}$ which extends to a meromorphic function on $\PP^1$, and is called the {\textit{functional invariant}}, denoted by $\mathscr{J}: \PP^1 \rightarrow \PP^1$.

The sheaf $R^1 \fF_* \ZZ$ is the first direct image sheaf of the constant sheaf $\ZZ$ on $\cK$, which is the sheaf on $\PP^1$
associated to the presheaf with sections over $U \subset \PP^1$ being $H^1( \fF^{-1}(U), \ZZ)$.
The sheaf $R^1 \fF_* \ZZ$ is called the homological invariant of $\cK$, and is locally constant over  $\PP^1\setminus\mathcal{S}$.
The homological invariant is equivalent to a representation $\rho: \pi_1( \PP^1 \setminus \mathcal{S}) \rightarrow SL(2, \ZZ)$,
which is called the {\textit{monodromy representation}}.
There is a compatibility relation between these invariants, see \cite[Section~8]{Kodaira1963}; the sheaf $R^1 \fF_* \ZZ$ {\textit{belongs to}} the meromorphic function $\mathscr{J}$.

By \cite{Kodaira1963}, given an elliptic surface $\fF: \cK\rightarrow \PP^1$, there exists a unique elliptic surface $\fJ:\cJ\to\PP^1$ with a holomorphic section $\sigma_0$ which has same functional and homological invariants as $\cK$. It was called the ``basic member" by Kodaira \cite{Kodaira1963} but was called ``Jacobian" by other authors \cite{GW}.

The space $\cK^\#$ is obtained from $\cK$ by replacing the singular fibers with only
the irreducible components with multiplicity 1 (minus the singular points and intersection points with other components). A section of $\fJ : \cJ \rightarrow \PP^1$ is a holomorphic mapping
$\sigma : \PP^1 \rightarrow \cJ$ such that $\fJ \circ \sigma = \Id_{\PP^1}$, which is
equivalent to a holomorphically embedded $\PP^1 \subset \cJ^\#$ which has intersection
number $1$ with fibers of $\cJ^\#$.
Therefore a section $\sigma_0$ distinguishes a point in every fiber of $\cJ^\#$. For $p \in \PP^1$, the fiber of $\cJ^\#$ over $p$, $\cJ^\#_p =
\sigma_0^{-1}(p) \cap \cJ^\#$, is an abelian group with identity $\sigma_0(p)$, see \cite[Section~9]{Kodaira1963}.
The subset $\cJ^\#_0\subset \cJ^\#$ is defined as the subset of $\cJ^\#$ consisting of the identity component of each fiber group.

Let $\pi:\mathcal{O}_{\PP^1}(-2) \rightarrow \PP^1$ be the
line bundle of degree $-2$ over $\PP^1$, which can be identified
with the cotangent bundle of $\PP^1$.
\begin{proposition}
\label{p:fprop}
Given the global section $\sigma_0:\PP^1\to \cJ$, there is an associated holomorphic mapping
\begin{align}
f_0: \mathcal{O}_{\PP^1}(-2)\rightarrow \cJ^\#_0,
\end{align}
satisfying $\fJ \circ f_0 = \pi$ such that restricted to a fiber, this mapping induces a group
homomorphism, with kernel equal to  $\ZZ \oplus \ZZ$, $\ZZ$, or $\{0\}$.
Furthermore,
\begin{align}
f_0^* \Omega = \Omega_{\can} \label{e:canonical-volume-form}
\end{align}
where $\Omega$ is a nonzero holomorphic $(2,0)$-form on $\cJ$,
and $\Omega_{\can}$ is the canonical holomorphic $(2,0)$-form on $ \mathcal{O}_{\PP^1}(-2)
\cong T^* \PP^1$.
\end{proposition}

\begin{proof}
The image of the section is a submanifold $\Sigma = \sigma_0 (\PP^1) \subset \cJ$.
To identify the normal bundle of $\Sigma$, use the adjunction formula
\begin{align}
K_{\Sigma} = ( K_{\cJ})|_{\Sigma} \otimes N_\Sigma
\end{align}
but the canonical bundle of $\cJ_0^\#$ is trivial, so
$N_{\Sigma} \cong K_{\Sigma} \cong \mathcal{O}_{\PP^1}(-2)$.
Then we get $f_0: \mathcal{O}_{\PP^1}(-2)_{\Sigma} \cong N_{\Sigma} \rightarrow \cJ_0^\#$ by
the fiberwise Lie group exponential map.

The second statement is proved in \cite[Proposition~7.2]{Gross1999}.
\end{proof}

\subsection{Construction of semi-flat metrics}
\label{ss:construction-semi-flat}
By definition, $f_0: \mathcal{O}_{\PP^1}(-2)\rightarrow \cJ^\#_0,
$ maps the zero section of $\mathcal{O}_{\PP^1}(-2)$
to the global holomorphic section $\sigma_0:\PP^1\to \cJ$.
By \eqref{e:canonical-volume-form}, we can write
\begin{equation}
f_0^*\Omega = \Omega_{\can}=-dx \wedge dy,
\end{equation}
where $\{x\}$ is the canonical holomorphic fiber coordinate of the cotangent bundle $\mathcal{O}_{\PP^1}(-2)$ and
$\{y\}$ is the holomorphic coordinate of the regular base region $\PP^1\setminus\mathcal{S}\subset\dC$. Note that our sign convention for $\Omega$ is opposite to that of \cite{GW}.

For each small constant $\delta>0$, Greene-Shapere-Vafa-Yau's work in \cite{GSVY} gives hyperk\"ahler metrics $g_\delta^{\SF}$ on the union of the regular fibers $\mathcal{R}\equiv \fJ^{-1}(\PP^1\setminus\mathcal{S})\subset \cJ$, which we can describe as follows.

The holomorphic periods are defined as $\tau_i(y) = \int_{\gamma_i} dx$,
where $\gamma_i$ are a basis of the first homology of the torus fiber, $i = 1,2$. Note that this depends on the choice of basis, so we will consider these as multi-valued functions. After exchanging $\gamma_1$ with $\gamma_2$ if necessary, we assume that $\Ima(\frac{\tau_2}{\tau_1})>0$.
In the above coordinates,
the left action of $(m,n)\in\Gamma  \equiv\mathbb{Z}\oplus\dZ $ on the regular part is given by
\begin{equation}(m,n)\cdot(y,x)\equiv(y,x+m\tau_1(y)+n\tau_2(y)).\end{equation}
 For each $\delta>0$,
let \begin{align}W & \equiv\frac\delta{\mathrm{Im}(\bar\tau_1\tau_2)},
 \\
 b & \equiv -\frac{W}\delta\Big(\mathrm{Im}(\tau_2\bar x)\partial_y\tau_1+\mathrm{Im}(\bar\tau_1x)\partial_y\tau_2\Big),\label{e:term-b}
\end{align}
then {\it Greene-Shapere-Vafa-Yau's semi-flat metric} with {\it bounded diameter}
is the metric $g_{\delta}^{\SF}$ associated to the K\"ahler form and holomorphic $2$-form
\begin{align}\omega_\delta^{\SF} & =\frac{\sq}{2}\cdot \delta \cdot\Big(W(dx+bdy)\wedge\overline{dx+bdy}+W^{-1}dy\wedge d\bar y\Big),
\\
\Omega_\delta & = -\delta \cdot dx\wedge dy = \delta \Omega. \end{align}

\begin{remark}The semi-flat metric is independent of the choice of $\gamma_1, \gamma_2$, and the
choice of local holomorphic coordinate $y$.
\end{remark}

Let $x=x_1\tau_1(y)+x_2\tau_2(y)$ for $x_1, x_2\in \RR/\ZZ$, then the K\"ahler form
can be rewritten as follows. First, \eqref{e:term-b} can be simplified as
\begin{align}
b=   -\frac{1}{\Ima(\bar{\tau}_1\tau_2)}\Big(x_1\Ima(\tau_2\bar{\tau}_1)\partial_y\tau_1+x_2\Ima(\bar{\tau}_1\tau_2)\partial_y\tau_2\Big)
=  - (x_1\cdot \partial_y\tau_1 + x_2\cdot \partial_y\tau_2),
\end{align}
which implies
$dx + b dy = \tau_1(y) dx_1 + \tau_2 (y)dx_2$.
So $\omega_\delta^{\SF}$ can  be simplified as
\begin{align}\omega_\delta^{\SF} & =\frac{\sq}{2}\cdot \delta \cdot\Big(W(\tau_1(y)dx_1+\tau_2(y)dx_2)\wedge\overline{\tau_1(y)dx_1+\tau_2(y)dx_2}+W^{-1}dy\wedge d\bar y\Big)
\nonumber\\
& =\delta^2 \cdot dx_1 \wedge dx_2+\frac{\sq}{2}\cdot \mathrm{Im}(\bar\tau_1\tau_2)dy\wedge d\bar y.
\label{e:SF1}
\end{align}
The holomorphic volume form $\Omega_\delta$ is expressed as
\begin{align}
\Omega_\delta  = -\delta \cdot \Big( \tau_1(y)dx_1+\tau_2(y)dx_2 \Big)\wedge dy.
\label{e:SF2}
\end{align}
It is easy to verify that both $\omega_\delta^{\SF}$ and $\Omega_\delta $ are $\Gamma$-invariant and hence they descend to the region $\mathcal{R}$. We will use the same notation for these descended forms.

Notice that there are
constants $C_0>0$ independent of $\delta>0$ such that  \begin{align}
\frac{1}{C_0}  \leq \diam_{g_\delta^{\SF}}(\mathcal{R}) & \leq C_0
\\
 \Area(\dC/\Gamma)  &= \delta^2,
\end{align}
where $\dC/\Gamma \cong \dT^2$ is the regular torus fiber.

\begin{definition}
\label{d:McLean}
The McLean
metric is the Riemannian metric $g_{ML}$ on $\mathbb{P}^1 \setminus \mathcal{S}$ associated to the K\"ahler form
\begin{align}
\label{omegaML}
\omega_{ML} = \frac{\sq}{2}\cdot \mathrm{Im}(\bar\tau_1\tau_2)dy\wedge d\bar y.
\end{align}
The induced distance function on $\mathbb{P}^1$ is denoted by  $d_{ML}$.
\end{definition}
We refer the reader to \cite{Gross1999, Hitchin97, McLean} for more details about the McLean metric.
Note that as $\delta \rightarrow 0$, $(\cJ \setminus \fJ^{-1}(\mathcal{S}), g^{\SF}_{\delta})$
converges to $(\mathbb{P}^1 \setminus \mathcal{S}, g_{ML})$ in the Gromov-Hausdorff distance.

\subsection{Rescaling and equivariant convergence}

\label{ss:rescaling-semi-flat}

The semi-flat metrics $g_\delta^{\SF}$ constructed in Section \ref{ss:construction-semi-flat} are hyperk\"ahler and collapsing with bounded curvatures away from the singular fibers.
For our purpose, we need to take a closer look at the convergence of the metrics and the lattices by unwrapping the collapsing torus fibers.
To describe this, we will use a standard notion, called {\textit{equivariant Gromov-Hausdorff convergence}}.
We refer the readers to \cite[Section~3]{FY} for other definitions and more details.

For $j \in \mathbb{Z}_+$, let $(M_j^n,g_j,p_j)$ be a sequence of Riemannian manifolds with $|\Rm_{g_j}|\leq 1$ such that
\begin{equation}
(M_j^n,g_j,p_j) \xrightarrow{GH} (X_{\infty}^k, d_{\infty},p_{\infty}),
\end{equation}
where $X_{\infty}^k$ is a $k$-dimensional Alexandrov space,
 then there exists some uniform constant $s_0=s_0(n,p_{\infty})>0$ such that we have the following diagram
\begin{equation}
\xymatrix{
\Big(\widehat{B_{s_0}(p_j)}, \hat{g}_j, \Gamma_j,\hat{p}_{j}\Big)\ar[rr]^{eqGH}\ar[d]_{\pr_j} &   & \Big(Y_{\infty}^n, \hat{g}_{\infty}, \Gamma_{\infty}, \hat{p}_{\infty}\Big)\ar [d]^{\pr_{\infty}}
\\
 \Big(B_{s_0}(p_j), g_j\Big)\ar[rr]^{GH} && \Big(B_{s_0}(p_{\infty}), d_{\infty}\Big)
}\label{e:equivariant-convergence-diagram}
\end{equation}
and the local Riemannian universal covering map $\pr_j:(\widehat{B_{s_0}(p_j)}, \hat{g}_j)\to B_{s_0}(p_j) $ converges to the submetry $\pr_{\infty}:Y_{\infty}^n\to \Gamma_{\infty}\backslash Y_{\infty}^n\equiv B_{s_0}(p_{\infty})$.
Here $\Gamma_j\equiv \pi_1(B_{s_0}(p_j))$, and the limiting group $\Gamma_{\infty}$ is a closed subgroup in $\Isom(Y_{\infty}^n)$.
The equivariant convergence means that the isometry actions of $\Gamma_j$ on $\widehat{B_{s_0}(p_j)}$ converge to the isometry action of $\Gamma_{\infty}$ on $Y_{\infty}^n$
with respect to their Gromov-Hausdorff convergence.
 Moreover, the local universal covers
$(\widehat{B_{s_0}(p_j)}, \hat{g}_j,\hat{p}_{j})$
are non-collapsing and hence $C^{1,\alpha}$-converge to the manifold $Y_{\infty}^n$ for any $\alpha\in(0,1)$.

Next, we will realize  the above picture
for the semi-flat metrics $g_\delta^{\SF}$
under the rescaled lattice $\Gamma_\delta\equiv (\delta\dZ)\oplus(\delta\dZ)$.
First, let $x_{\delta,1}\equiv \delta x_1$ and $x_{\delta,2}\equiv \delta x_2$. Then
\begin{equation}
x_\delta \equiv \delta x = \tau_1(y) x_{\delta,1} + \tau_2(y) x_{\delta,2}.
\end{equation}
The lattice $\Gamma_\delta$  gives a left action on $ \dC\times\dC$ as follows,
\begin{equation}
(\delta m,\delta n)\cdot(y,x_\delta)\equiv \Big(y,x_\delta+\delta m\cdot \tau_1(y)+\delta n\cdot \tau_2(y)\Big).\end{equation}
Let us recall that the semi-flat K\"ahler form and the corresponding holomorphic volume form are given by
\begin{align}\omega_\delta^{\SF}&=dx_{\delta,1} \wedge dx_{\delta,2}+\frac{\sq}{2}\cdot \mathrm{Im}(\bar\tau_1\tau_2)dy\wedge d\bar y, \\
 \Omega_\delta & = -\Big( \tau_1(y)dx_{\delta,1}+\tau_2(y)dx_{\delta,2} \Big)\wedge dy. \label{e:explicit-expression-uniform-coefficient}
 \end{align}
 In the above notations, one can check that both $\omega_\delta^{\SF}$
 and $\Omega_\delta$ are $\Gamma_\delta$-invariant and they descend to the regular region of the K3 surface $\mathcal{R}\equiv \fF^{-1}(\PP^1\setminus \mathcal{S})\subset \cJ$, still denoted by $\omega_\delta^{\SF}$ and $\Omega_\delta $ for convenience.

Notice that the K\"ahler form $\omega_\delta^{\SF}$
and the holomorphic $2$-form $\Omega_\delta$  are written in coordinates $(x_\delta,y)$ on the cotangent bundle
$\pi:\mathcal{O}_{\PP^1}(-2)\to \PP^1$.
In the following,
we will view $\omega_\delta^{\SF}$ and $\Omega_\delta$ as $2$-forms on the local universal covering space of $\mathcal{R}\subset \cJ$. To see this,
 take a ball  $B_s(p)\subset \PP^1$ which is diffeomorphic to a $2$-disc $\dD\subset \dC$. Using the holomorphic map
$f_0: \mathcal{O}_{\PP^1}(-2)\rightarrow \cJ^\#_0$, there is an open subset $\mathcal{V}_s\equiv\pi^{-1}(B_s(p))\subset\mathcal{O}_{\PP^1}(-2)$ which is biholomorphic to
$B_s(p)\times \dC$
and naturally gives a universal covering map
\begin{equation}
\pr_\delta: \mathcal{V}_s \longrightarrow \fJ^{-1}(B_s(p)).
\end{equation}
Now the equivariant Gromov-Hausdorff convergence diagram in our context reads as follows,
\begin{equation}
\xymatrix{
\Big(\mathcal{V}_s, g_\delta^{\SF}, \Gamma_\delta\Big)\ar[rr]^{eqGH}\ar[d]_{\pr_\delta} &   & \Big(\mathcal{Y}, g_{\mathcal{Y}}, \Gamma_0\Big)\ar [d]^{\pr_{0}}
\\
 \Big( \fJ^{-1}(B_s(p)), g_{\delta,\cJ}^{\SF}\Big)\ar[rr]_{\delta\to 0}^{GH} && \Big(B_{s}(p), d_{ML}\Big),
}\label{e:semi-flat-equivariant-convergence}
\end{equation}
where the limiting group $\Gamma_0=\dR\oplus \dR$ acts isometrically and freely on the $4$-manifold $\mathcal{Y}$.
In the above diagram, due to the elliptic regularity for Einstein equations in the non-collapsed context, the metrics $g_\delta^{\SF}$ on the local universal covers converge to the limiting metric $g_{\mathcal{Y}}$ in the  $C^k$-topology for any $k\in\dZ_+$.

To finish this section, we carry out some local computations for the semi-flat metrics, which will be used for the weighted analysis in our later arguments.
Let $\eta\in \Omega^1(\cJ)$
be a real differential $1$-form such that
 $\eta$ can be written as
\begin{align}\eta=f^{(y)} dy+f^{(\bar{y})}d\bar y+f^{(3)}\cdot e^3 + f^{(4)}\cdot e^4\label{e:1-form-representation}\end{align}
in the regular region of $\cJ$,
where $y=y_1+\sq y_2$,
\begin{align}
e^1 & = \sqrt{\Ima(\bar\tau_1\tau_2)}dy_1,
\\
e^2 & = \sqrt{\Ima(\bar\tau_1\tau_2)}dy_2,
\\
e^3 & = \frac{1}{\sqrt{\Ima(\bar\tau_1\tau_2)}}(\Rea(\tau_1)dx_{\delta,1} + \Rea(\tau_2)dx_{\delta,2}),
\\
e^4 & = \frac{1}{\sqrt{\Ima(\bar\tau_1\tau_2)}}(\Ima(\tau_1)dx_{\delta,1} + \Ima(\tau_2)dx_{\delta,2}).
\end{align}
Since $\eta$ is a real $1$-form, so it holds that $f^{(\bar{y})}=\overline{f^{(y)}}$.
Let us define
\begin{align}
e^{(x)} & \equiv e^3 + \sq e^4 = \frac{1}{\sqrt{\Ima(\bar\tau_1\tau_2)}}(\tau_1 \cdot dx_{\delta,1} + \tau_2 \cdot dx_{\delta,2}) , \label{e:ex}
\\
 \overline{e^{(x)}} & \equiv e^3 - \sq e^4 = \frac{1}{\sqrt{\Ima(\bar\tau_1\tau_2)}}(\bar{\tau}_1 \cdot dx_{\delta,1} + \bar{\tau}_2 \cdot dx_{\delta,2})) , \label{e:ex-bar}
\end{align}
so that $\eta$ can be represented as
\begin{equation}
\eta = f^{(y)} dy+f^{(\bar{y})}d\bar y +\Rea(F^{(x)}\cdot e^{(x)}),
\end{equation}
where $F^{(x)} \equiv f^{(3)} - \sq f^{(4)}$.
Also it is straightforward that the coframe $\{e^1, e^2, e^3, e^4\}$ is a standard hyperk\"ahler basis such that
\begin{align}
\omega_\delta^{\SF} & = e^1\wedge e^2+e^3\wedge e^4,
\\
\Omega_\delta & = (e^1+\i e^2)\wedge (e^3+\i e^4).
\end{align}
\begin{lemma}\label{l:D-xi-semi-flat}
Let the real differential $1$-form $\eta\in \Omega^1(\cJ)$ satisfy the representation \eqref{e:1-form-representation} in the regular region. If $f^{(y)}$, $f^{(\bar{y})}$, $f^{(3)}$, $f^{(4)}$ are functions depending only on $y$,  then
the following local formulas of $d^+\eta$ and $d^*\eta$ hold:
\begin{align}
d^+ \eta = &
 (\p_y f^{(\bar{y})} - \bar{\p}_y f^{(y)}) (dy \wedge d\bar{y})^+
+ \Rea \Big(\frac{\p_y(\sqrt{\Ima(\bar{\tau}_1\tau_2)}\cdot  F^{(x)})}{\sqrt{\Ima(\bar{\tau}_1\tau_2)}}dy\wedge  e^{(x)}\Big),
\\
d^*\eta  = & - \frac{2}{\Ima(\bar{\tau}_1\tau_2)}(\bar\partial_y f^{(y)} + \partial_y f^{(\bar{y})}),
\end{align}
where $\xi^+\equiv (\xi + * \xi)/2$ for any $2$-form $\xi$,  $d^+ \eta \equiv (d\eta)^+$,
 $d^*$ is the $L^2$-adjoint of $d$, all computed with respect to the metric $g_{\delta}^{\SF}$.
\end{lemma}

\begin{proof}
First, we compute $d^+\eta$. Notice that the coefficients of $\eta$ depends only on $y$ so that we have
\begin{align}
d\eta =&  (\p_y f^{(\bar{y})} - \bar{\p}_y f^{(y)}) (dy \wedge d\bar{y})
+ \Rea (d(F^{(x)} \cdot e^{(x)}))
\nonumber\\
= &  (\p_y f^{(\bar{y})} - \bar{\p}_y f^{(y)}) (dy \wedge d\bar{y})
+ \Rea \Big(d(\frac{F^{(x)}}{\sqrt{\Ima(\bar{\tau}_1\tau_2)}}\cdot \sqrt{\Ima(\bar{\tau}_1\tau_2)} \cdot e^{(x)})\Big)
\nonumber\\
=&  (\p_y f^{(\bar{y})} - \bar{\p}_y f^{(y)}) (dy \wedge d\bar{y})
+ \Rea \Big(\frac{F^{(x)}}{\sqrt{\Ima(\bar{\tau}_1\tau_2)}}\cdot d(\sqrt{\Ima(\bar{\tau}_1\tau_2)} \cdot e^{(x)})
\nonumber\\
& + \sqrt{\Ima(\bar{\tau}_1\tau_2)}\p_y (\frac{F^{(x)}}{\sqrt{\Ima(\bar{\tau}_1\tau_2)}}) dy\wedge   e^{(x)} + \sqrt{\Ima(\bar{\tau}_1\tau_2)} \bp_y (\frac{F^{(x)}}{\sqrt{\Ima(\bar{\tau}_1\tau_2)}}) d\bar{y} \wedge   e^{(x)}\Big).
\end{align}
Now we are ready to compute the self-dual part of $d\eta$.  One can verifies that
$dy\wedge e^{(x)}$ is self-dual and $d\bar{y}\wedge e^{(x)}$ is anti-self-dual, so it holds that
\begin{align}
d^+ \eta = &  (\p_y f^{(\bar{y})} - \bar{\p}_y f^{(y)}) (dy \wedge d\bar{y})
+ \Rea \Big(\sqrt{\Ima(\bar{\tau}_1\tau_2)}\p_y (\frac{F^{(x)}}{\sqrt{\Ima(\bar{\tau}_1\tau_2)}}) dy\wedge   e^{(x)}\Big)
\nonumber\\
&  + \Rea\Big(\frac{F^{(x)}}{\sqrt{\Ima(\bar{\tau}_1\tau_2)}}\cdot d^+(\sqrt{\Ima(\bar{\tau}_1\tau_2)} \cdot e^{(x)})\Big).\label{e:to-be-simplified}
\end{align}

 The next is to simplify the third term in \eqref{e:to-be-simplified}.
  By \eqref{e:ex}, it follows that
 \begin{align}
 d (\sqrt{\Ima(\bar{\tau}_1\tau_2)}\cdot e^{(x)}) = \p_y\tau_1 dy\wedge dx_{\delta,1}+\p_y\tau_2 dy\wedge dx_{\delta,2}.
 \end{align}
By \eqref{e:ex} and \eqref{e:ex-bar}
\begin{align}
dx_{\delta,1} = & -\frac{1}{2\i \cdot \sqrt{\Ima(\bar\tau_1\tau_2)}}(\bar\tau_2 e^{(x)}-\tau_2\overline{ e^{(x)} }), \\
dx_{\delta,2} = & \frac{1}{2\i \cdot  \sqrt{\Ima(\bar\tau_1\tau_2)}}(\bar\tau_1 e^{(x)}-\tau_1\overline{ e^{(x)}}).
\end{align}
Therefore,
\begin{align}
& d (\sqrt{\Ima(\bar{\tau}_1\tau_2)}\cdot e^{(x)})
 \nonumber\\
   = & \frac{1}{2\sq \cdot  \sqrt{\Ima(\bar\tau_1\tau_2)}}
\Big((\bar{\tau}_1\p_y \tau_2-\bar{\tau}_2\p_y\bar{\tau}_1) dy\wedge e^{(x)} + (\tau_2\p_y\tau_1 - \tau_1\p_y\tau_2) dy\wedge \overline{ e^{(x)}} \Big)
\nonumber\\
= &\frac{\p_y\Ima(\bar{\tau}_1 \tau_2)}{\sqrt{\Ima(\bar\tau_1\tau_2)}}
 dy\wedge e^{(x)} + \frac{1}{2\sq\cdot \sqrt{\Ima(\bar\tau_1\tau_2)}}(\tau_2\p_y\tau_1 - \tau_1\p_y\tau_2) dy\wedge \overline{ e^{(x)}},
\end{align}
where we used the property that $\tau_1$ and $\tau_2$ are holomorphic in $y$. Since $dy\wedge e^{(x)}$ is self-dual and $dy\wedge \overline{e^{(x)}}$ is anti-self-dual, so we have
\begin{align}
d^+ (\sqrt{\Ima(\bar{\tau}_1\tau_2)}\cdot e^{(x)})
= \frac{\p_y\Ima(\bar{\tau}_1 \tau_2)}{\sqrt{\Ima(\bar\tau_1\tau_2)}}
 dy\wedge e^{(x)} .\end{align}

Finally, we have
\begin{align}
d^+\eta = & (\p_y f^{(\bar{y})} - \bar{\p}_y f^{(y)}) (dy \wedge d\bar{y})^+
+ \Rea \Big((\sqrt{\Ima(\bar{\tau}_1\tau_2)}\p_y (\frac{F^{(x)}}{\sqrt{\Ima(\bar{\tau}_1\tau_2)}})\nonumber\\
+ & \frac{F^{(x)}\cdot \p_y\Ima(\bar{\tau}_1 \tau_2)}{\Ima(\bar{\tau}_1\tau_2)})dy\wedge  e^{(x)}\Big)\\
= &(\p_y f^{(\bar{y})} - \bar{\p}_y f^{(y)}) (dy \wedge d\bar{y})^+
+ \Rea \Big(\frac{\p_y(\sqrt{\Ima(\bar{\tau}_1\tau_2)}\cdot  F^{(x)})}{\sqrt{\Ima(\bar{\tau}_1\tau_2)}}dy\wedge  e^{(x)}\Big).
\end{align}

Next, we compute the term $d^*\eta $. To begin with, it is straightforward that
\begin{align}
&\dvol= e^1 \wedge e^2\wedge  e^3 \wedge e^4 = \Ima(\bar{\tau}_1\tau_2) dy_1\wedge dy_2\wedge dx_{\delta,1}\wedge dx_{\delta,2},
\\
& *(dy)  =-\sq   dy\wedge e^3\wedge e^4=-\sq   dy\wedge dx_{\delta,1}\wedge dx_{\delta,2},\\
& *(d\bar{y})= \sq  d\bar{y}\wedge e^3\wedge e^4=\sq  d\bar{y}\wedge e^3\wedge dx_{\delta,1}\wedge dx_{\delta,2},\\
& *(e^3)   = e^1\wedge e^2\wedge e^4 = \frac{\sq}{2}\Ima(\bar{\tau}_1\tau_2)\cdot dy\wedge d\bar{y}\wedge e^4,\\
& *(e^4)   = -e^1\wedge e^2\wedge e^3 = -\frac{\sq}{2}\Ima(\bar{\tau}_1\tau_2)\cdot dy\wedge d\bar{y}\wedge e^3
.\quad
\end{align}
So it follows that \begin{align}
*\eta =&\sq  \cdot( f^{(\bar{y})} d\bar{y}- f^{(y)} dy) \wedge dx_{\delta,1}\wedge dx_{\delta,2}
\nonumber\\
&+ \frac{\sq}{2} f^{(3)}\cdot \Ima(\bar{\tau}_1\tau_2)\cdot dy\wedge d\bar{y}\wedge e^4
- \frac{\sq}{2} f^{(4)}\cdot \Ima(\bar{\tau}_1\tau_2)\cdot dy\wedge d\bar{y}\wedge e^3,
\end{align}
and hence
\begin{align}
d * \eta = &\sq  \cdot( \p_y f^{(\bar{y})} +\bar{\p}_y f^{(y)}) \cdot dy\wedge d\bar{y} \wedge dx_{\delta,1}\wedge dx_{\delta,2}
\nonumber\\
= & 2\cdot (\p_y f^{(\bar{y})}+\bar{\p}_y f^{(y)}) \cdot dy_1\wedge dy_2\wedge dx_{\delta,1}\wedge dx_{\delta,2}.
\end{align}
Therefore,
\begin{align}
d^* \eta = -*  d  * \eta = - \frac{2}{\Ima(\bar{\tau}_1\tau_2)}(\bar\partial_y f^{(y)} + \partial_y f^{(\bar{y})}),
\end{align}
and the proof is complete.

\end{proof}

\section{Semi-flat metric on a general elliptic K3}
\label{s:elliptic-fibration}

In the previous section, we have described the Greene-Shapere-Vafa-Yau's semi-flat metric $g_\delta^{\SF}$ on $\cJ$. In this section, we translate $\omega_\delta^{\SF}$ by local sections to get a semi-flat K\"ahler form $\omega_\delta^A$ on $\cK$. In general, $\cK$ does not have any global holomorphic section. However, it always admits a smooth section.
\begin{theorem}
\label{existsecthm} Consider an elliptic K3 surface $\fF: \cK \rightarrow \PP^1$,
not necessarily with a holomorphic section.
Then there exists a $C^{\infty}$ section $\sigma_{\infty} : \PP^1 \rightarrow \cK^\#$.
\end{theorem}
\begin{proof}
This theorem is essentially proved in \cite{FriedmanMorgan}. By \cite[Lemma~I.5.11]{FriedmanMorgan}, \cite[Theorem~I.5.1]{FriedmanMorgan} and the paragraph before \cite[Theorem~I.5.13]{FriedmanMorgan}, we know that the elliptic K3 surfaces with the same Jacobian $\cJ$ form a single deformation equivalence class, with base parametrized by $H^2(\mathcal{J}, \mathcal{O}_{\mathcal{J}}) = \CC$. Then by \cite[Remark~II.1.4 ]{FriedmanMorgan}, we know that there exists a $C^{\infty}$ section $\sigma_{\infty} : \PP^1 \rightarrow \cK^\#$.
\end{proof}

In the following, we will let $\fJ: \cJ \rightarrow \PP^1$ be the associated Jacobian K3 surface. Then there exists a unique diffeomorphism map $\Psi:\cK\rightarrow\cJ$ which maps $\sigma_\infty$ to $\sigma_0$, satisfies $\fJ \circ \Psi = \fF$ and maps each fiber analytically to another fiber.
Define $\cK^\#_0$ by replacing the singular fiber with just the component of $\cK^\#$ intersecting the section $\sigma_\infty$.
We have the following analogue to Proposition~\ref{p:fprop}.
\begin{corollary}
\label{fprop2} Let $\fF: \cK \rightarrow \PP^1$ be any elliptic K3 surface.
Consider the $C^{\infty}$ mapping
$f: \mathcal{O}_{\PP^1}(-2)\rightarrow \cK^\#_0$,
defined by
$f \equiv \Psi^{-1} \circ f_0$,
where $f_0$ is as in Proposition \ref{p:fprop}. Then
\begin{align}
f^* \Omega = \Omega_{\can} + \pi^* \alpha,
\end{align}
where $\Omega$ is a non-zero holomorphic $(2,0)$-form on $\cK$,
$\alpha$ is a smooth $2$-form on $\PP^1$, and
$\Omega_{\can}$ is the canonical holomorphic $(2,0)$-form on $T^*\PP^1 \cong  \mathcal{O}_{\PP^1}(-2)$.
\label{c:holomorphi-two-form}
\end{corollary}
\begin{proof} This is proved in \cite[Proposition~7.2]{Gross1999}.
\end{proof}
Recall that there is a group structure on $\cJ^{\#}$ such that $\sigma_0$ is the identity element of this group action \cite{Kodaira1963}. We choose a good open cover $U_i$ of $\PP^1$ such that each $U_i$ contains at most one $p \in \cS$. Notice that $U_i \cap U_j$ is contractible because $U_i$ is a good cover, so we can find two periods $\tau_1$ and $\tau_2$ on each $U_i \cap U_j$.
 The next goal in this section is, given $\mathbb{B} \in H^1(\PP^1, R^1 \fJ_* \RR)$,
to find $\sigma_i\in C^\infty(U_i, \cJ_0^{\#})$ such that $\sigma_i-\sigma_j\in H^1( \fJ^{-1}(U_i \cap U_j), \RR)$ represents $\mathbb{B}$.
To make sense of this goal, we  lift $\sigma_i$ to sections $s_i$
of $\mathcal{O}_{\PP^1}(-2)$ over $U_i$, and we would like to write the difference on any $U_i \cap U_j$ as
\begin{equation}
\label{sisj}
s_i-s_j=a_{ij} \tau_1+ b_{ij} \tau_2,
\end{equation}
where $a_{ij}, b_{ij} \in \RR$ are constants.
To view this difference as a \v{C}ech cocycle
with coefficients in the sheaf $R^1 \fJ_* \RR$, for constants $a,b \in \RR$,
and $x = x_1 \tau_1 + x_2 \tau_2$,  we choose the identification
\begin{align}
\label{sident}
a \tau_1 + b \tau_2 \mapsto [ -b dx_1 + a dx_2],
\end{align}
where the right hand side is an element of $H^1 ( \fJ^{-1} (U_i \cap U_j), \RR)$.

If we can find such local smooth sections $s_i$, then $L(s_i-s_j)^*\omega_\delta^{\SF}=\omega_\delta^{\SF}$ on $\fJ^{-1} (U_i \cap U_j)$ by (\ref{e:SF1}), where $L(s_i)(s)=s_i+s$ is the group action. It implies that $L(s_i)^*\omega_\delta^{\SF}$ is a well defined form on $\cJ$ because $L(s_i)^*\omega_\delta^{\SF}=L(s_j)^*\omega_\delta^{\SF}$ on $\fJ^{-1} (U_i \cap U_j)$. We will define $\omega_\delta^A$ using this form.
\begin{theorem}[Leray-Serre] There is a spectral sequence whose second page is
\begin{align}
 E_2^{pq} = H^p(\PP^1, R^q \fJ_* \RR)
\end{align}
and which converges to
\begin{align}
E^{pq}_{\infty} = H^{p+q}(\cJ, \RR).
\end{align}
Furthermore, we have an exact sequence
\begin{equation}
\label{sss1}
\begin{tikzcd}
0  \arrow[r]  & E_2^{10} \arrow[r]  \arrow[d, phantom, ""{coordinate, name=Z}] &
H^1(\cJ, \RR) \arrow[r]   &
E_2^{01}  \arrow[dll,swap,"d_2",
rounded corners,
to path={ -- ([xshift=2ex]\tikztostart.east)
|- (Z) [near end]\tikztonodes
-| ([xshift=-2ex]\tikztotarget.west)
-- (\tikztotarget)}] \\
\ &    E_2^{20}
\arrow[r] & \Ker \{ H^2(\cJ,\RR) \mapsto E_2^{02} \}
 \arrow[r] & E_2^{11} \arrow[r] & E_2^{30}   .
\end{tikzcd}
\end{equation}
\end{theorem}
\begin{proof} This is standard, see for example \cite{McCleary}. The exact sequence \eqref{sss1} is also
known as the ``seven-term exact sequence'' associated to a converging spectral sequence.
\end{proof}
We next analyze this exact sequence. First,  $H^1(\cJ, \RR) = 0$ since $\cJ$ is simply connected.
Also \begin{align}
E_2^{30} = H^3(\PP^1, \fJ_*\RR) = 0
\end{align}
because $\PP^1$ has a good cover with all $4$-fold intersections vanishing.
Next,
\begin{align}
E_2^{01} = H^0( \PP^1, R^1 \fJ_* \RR) = 0,
\end{align}
because this is equal to $(\RR \oplus \RR)^{\rho}$ (the group
of invariants) where we view $\rho$ as a $2$-dimensional
\textit{real} representation of $\pi_1$, and this is easily seen to vanish
\cite[Proposition~2.1]{Shioda1972}.
Note that $\fJ_* \RR = \RR$ (this is $\RR$ with the discrete topology),
so we have
\begin{align}
E_2^{20}= H^2( \PP^1, \RR) = \RR.
\end{align}
The first mapping $E_2^{20}$ to $H^2(\cJ, \RR)$  is just the pull-back
$\fJ^*: H^2(\PP^1, \RR) \rightarrow H^2(\cJ, \RR)$.
Therefore, the seven-term sequence yields a short exact sequence
\begin{equation}
\label{sss2}
\begin{tikzcd}
0  \arrow[r] & \RR \arrow[r, "\fJ^*"]
& \Ker \{ H^2(\cJ, \RR) \mapsto E_2^{02} \} \arrow[r] & H^1(\PP^1, R^1 \fJ_* \RR) \arrow[r] & 0.
\end{tikzcd}
\end{equation}
Next, we have
\begin{align}
E^{02}_2 = H^0 ( \PP^1, R^2 \fJ_* \RR)
\end{align}
and the mapping from  $ H^2(\cJ, \RR) \to  H^0 ( \PP^1, R^2 \fJ_* \RR)$
is described as follows. Fix a finite good cover $\{U_i\}$ of $\PP^1$.
An element of $h \in  H^0 ( \PP^1, R^2 \fJ_* \RR)$ is
a $0$-cycle $h_i \in H^2( \fJ^{-1}(U_i), \RR)$ such that
$h_i = h_j$ in $H^2(\fJ^{-1}(U_i \cap U_j), \RR)$.
The mapping from $H^2(\cJ, \RR)$ to  $H^0 ( \PP^1, R^2 \fJ_* \RR)$
is
just $\omega \mapsto \omega|_{\fJ^{-1}(U_i)}$.
So the middle term is
\begin{align}
K \equiv \{ [\omega] \in H^2(\cJ, \RR) : \omega|_{\fJ^{-1}(U_i)} = 0 \in H^2( \fJ^{-1}(U_i), \RR) \}.
\end{align}
So we have arrived at the following.
\begin{corollary}There is a short exact sequence
\begin{equation}
\label{sss3}
\begin{tikzcd}
0  \arrow[r] & E \arrow[r, "\fJ^*"]
& K \arrow[r] & H^1(\PP^1, R^1 \fJ_* \RR) \arrow[r] & 0,
\end{tikzcd}
\end{equation}
where $E \equiv  H^2(\PP^1, \RR)$. Consequently,
\begin{align}
 H^1(\PP^1, R^1 \fJ_* \RR) \cong K/E.
\end{align}
\end{corollary}
\begin{remark}
\label{Kmap}
Note that the natural mapping from $K$ to $H^1(\PP^1, R^1 \fJ_* \RR)$ is to
write $\omega_i = d\alpha_{i}$ on $\fJ^{-1}(U_i)$, and then map to the
$1$-cycle given by $\alpha_i - \alpha_j$ on $U_i \cap U_j$, which is clearly an element of
$H^1( \fJ^{-1} (U_i \cap U_j), \RR)$.
\end{remark}
Now we are ready to describe the construction of $\omega_\delta^B$. First, recall that by Corollary \ref{c:holomorphi-two-form}, there exists a diffeomorphism $\Psi : \cK \rightarrow \cJ$ and a 2-form $\alpha\in H^2(\PP^1)$ such that $\fF = \fJ \circ \Psi$ and
\begin{equation}
\Omega_{\cK} = \Psi^*\Omega_{\cJ} + \fF^*\alpha.
\end{equation}
Define
\begin{align}
\label{Bdef}
\mathscr{B} = \Big\{ \mathbb{B} \in K/E  \ \Big| \ \int_{\cJ} \mathbb{B} \wedge \Omega_{\cJ} = - \int_{\PP^1} \alpha \Big\}.
\end{align}
Note that $\int_{\cJ} \mathbb{B} \wedge \Omega_{\cJ}$ is well defined because for any $\beta \in E=H^2(\PP^1, \RR)$, we necessarily have $\int_{\cJ} \fJ^*(\beta) \wedge \Omega_{\cJ}=0$.

Let $k_1$ denote the number of fibers with finite monodromy, $k_2$ denote the number of
$I_{\nu}$ fibers, and $k_3$ denote the number of $I_{\nu'}^*$ fibers.
\begin{proposition}  The space $\mathscr{B}$ has dimension
\begin{align}
\label{dimB}
\dim(\mathscr{B}) =   2 k_1 + k_2 + 2 k_3 - 5  >0 .
\end{align}
\end{proposition}
\begin{proof}

The dimension of $H^2(\cJ,\RR)$ is 22. The quotient by $E$ reduces the dimension by $1$. Then the restriction
\begin{equation}
\int_{\cJ} \mathbb{B} \wedge \Omega_{\cJ} = - \int_{\PP^1} \alpha
\end{equation} reduces the dimension by 1 again. The integral of $\omega$ vanishes on each regular fiber,
which reduces the dimension by another $1$.  For each $U_i$, we have the restriction
\begin{equation}
\omega|_{\fJ^{-1}(U_i)} = 0 \in H^2( \fJ^{-1}(U_i), \RR),
\end{equation}
but the vanishing of the integral of $\omega$ vanishes on each regular fiber is already counted, so there are $\dim H^2 (\fJ^{-1}(U_i), \RR) -1$ restrictions left on each $U_i$,
and we have
\begin{equation}
\dim(\mathscr{B}) = 22 - \sum_{i}\big( \dim H^2 (\fJ^{-1}(U_i), \RR) -1 \big) - 3.
\end{equation}
Since the sum of the Euler characteristics of the singular fibers must be $24$,
it follows that
\begin{equation}
22 - \sum_{i}\big( \dim H^2 (\fJ^{-1}(U_i), \RR) -1 \big) - 3 = 2 k_1 + k_2 + 2 k_3 - 5.
\end{equation}
To see the strict inequality in \eqref{dimB}, the Shioda-Tate formula implies that
\begin{align}
\rho(\cJ) = 26 - 2 k_1 - k_2 - 2 k_3 + \rank(\mbox{MW}(\cJ)),
\end{align}
where $\mbox{MW}(\cJ)$ is the Mordell-Weil group of $\cJ$ and $\rho(\cJ)$ is the Picard number, see \cite[Chapter~11]{Huybrechts}.
Since $\rho(\cJ) \leq 20$, and $\rank(\mbox{MW}(\cJ)) \geq 0$, this implies the inequality
\begin{align}
2 k_1 + k_2 + 2 k_3 \geq 6,
\end{align}
so $\dim(\mathscr{B}) > 0$.
\end{proof}
Consider the exact sequence of sheaves
\begin{equation}
\label{sss4}
\begin{tikzcd}
0  \arrow[r] &  R^1 \fJ_* \RR  \arrow[r]
&   C^\infty (\mathcal{O}_{\PP^1}(-2))  \arrow[r] & F \arrow[r] & 0,
\end{tikzcd}
\end{equation}
where $R^1 \fJ_* \RR$ is the first direct image sheaf of the constant sheaf $\RR$ on $\cJ$, which is identified
as a subsheaf of  $ C^\infty (\mathcal{O}_{\PP^1}(-2))$ using \eqref{sident}, and $F$ is the quotient sheaf
$ C^\infty (\mathcal{O}_{\PP^1}(-2)) / R^1 \fJ_* \RR$.
Since $ H^1(\PP^1, C^\infty (\mathcal{O}_{\PP^1}(-2)))=0$, the mapping from $H^0(\PP^1, F)$ to $H^1 (\PP^1, R^1 \fJ_* \RR)$ is surjective.
Therefore any element $\mathbb{B} \in \mathscr{B}$ has a preimage $\{s_i\}$ in $H^0(\mathbb{P}^1,F)$.
So as mentioned above, we can define $\omega_\delta^{A}\equiv \Psi^*L(s_i)^*\omega_\delta^{\SF}$, which is well-defined on $\cK$.

We next want to compute $L(s_i)^*\omega_\delta^{\SF}-\omega_\delta^{\SF}$. Write a section $s_i$ as
\begin{equation}
s_i=s_{i,1}\tau_1+s_{i,2}\tau_2
\end{equation}
locally.
Then the 1-form $\eta_i$ defined as
\begin{equation}
\eta_i = \delta^2(s_{i,1} dx_2 - s_{i,2} dx_1 + \frac{1}{2} (s_{i,1} d s_{i,2} - s_{i,2} d s_{i,1}))
\end{equation}
is $\SL(2,\ZZ)$ invariant and is therefore well defined on $\fJ^{-1}(U_i\setminus \cS)$.
By (\ref{e:SF1}),
\begin{equation}
L(s_i)^*\omega_\delta^{\SF} - \omega_\delta^{\SF} = \delta^2 d (x_1+s_{i,1}) \wedge d (x_2+s_{i,2})- \delta^2 d x_1 \wedge d x_2 = d \eta_i.
\end{equation}
In general, $L(s_i)^*\omega_\delta^{\SF} - \omega_\delta^{\SF}$ is singular near the singular fibers. However, we can choose a smooth cut-off function $\chi_i$ supported in $U_i$ such that $\chi_i=1$ near points in $\cS$, if there is any such point in $U_i$, and $\chi_i=0$ on $\fJ^{-1}(U_j)$ for all $j\not=i$. Then the form
\begin{equation}
\omega_\delta^{\mathbb{B}} \equiv L(s_i)^*\omega_\delta^{\SF} - \omega_\delta^{\SF} - d (\chi_i \eta_i)
\end{equation}
is  well-defined and  smooth on $\cJ$. Note also that, on $\fJ^{-1}(U_i)$, we have
\begin{align}
\omega_\delta^{\mathbb{B}}= d ( (1-\chi_i) \eta_i ),
\end{align}
so $\omega_\delta^{\mathbb{B}}$ belongs to $K$ because $(1-\chi_i)\eta_i$ is smooth on $\fJ^{-1}(U_i)$.
We want to compute the projection of  $\omega_\delta^{\mathbb{B}}$ to $K/E$.
As mentioned in Remark \ref{Kmap} above, the natural map from $K$ to $H^1(\PP^1, R^1 \fJ_* \RR)=K/E$ is defined by choosing 1-forms on each $U_i$ as $(1-\chi_i)\eta_i$ and then taking the difference between them.
On $U_i \cap U_j$, using \eqref{sisj}, we have
\begin{align}
\begin{split}
\eta_i - \eta_j &\equiv \delta^2 \Big(  (s_{i,1}- s_{j,1}) dx_2 - (s_{i,2}- s_{j,2}) dx_1 \Big) \mod \{ dy, d \bar{y} \}\\
& \equiv \delta^2\big( a_{ij} dx_2 - b_{ij} dx_1 \big) \mod \{ dy, d \bar{y} \}.
\end{split}
\end{align}
So the de Rham class of $\eta_i - \eta_j$ is represented by $\delta^2 ( -b_{ij} dx_1 + a_{ij} dx_2 )$.
Using the identification \eqref{sident}, we see that
$\omega_\delta^{\mathbb{B}}$ projects to $\delta^2\mathbb{B}$ in $K/E$.

We can also define $\Omega_{\cJ}^A$ as $L(s_i)^*\Omega_{\cJ}$. Then $\Omega_{\cJ}^A$ is well defined. There exists a 2-form $\xi$ on $\mathbb{P}^1$ such that we get $\Omega_{\cJ}^A=\Omega_{\cJ}+\fJ^*\xi$.
We then have
\begin{align}
\label{inteqn}
0 = \int_{\cJ} \Omega_{\cJ}^A \wedge (\Psi^{-1})^*\omega_\delta^A = \int_{\cJ} (\Omega_{\cJ} + \fJ^* \xi) \wedge ( \omega_\delta^{\SF} + \omega_\delta^{\mathbb{B}}) = \delta^2 \Big(\int_{\cJ} \Omega_{\cJ} \wedge \mathbb{B} + \int_{\PP^1} \xi \Big).
\end{align}
To see this, the first equality in \eqref{inteqn} is true because $(\Psi^{-1})^*\omega_\delta^A$ and $\Omega_{\cJ}^A$ are locally the pull back of a hyperk\"ahler triple using the same map $L(s_i)$. The second equality in \eqref{inteqn} is true because
\begin{align}
\begin{split}
\int_{\cJ}(\Omega_{\cJ} + \fJ^* \xi)\wedge d(\chi_i\eta_i) & = \lim _{r\to 0} \int_{\fJ^{-1}\{|y|=r\}} (\Omega_{\cJ} + \fJ^* \xi)\wedge \chi_i\eta_i \\
& = \lim _{r\to 0} \int_{\fJ^{-1}\{|y|=r\}} \Omega_{\cJ} \wedge \delta^2(s_{i,1} dx_2 - s_{i,2} dx_1) \\
& = \lim _{r\to 0} \int_{\fJ^{-1}\{|y|=r\}} \delta^2 s_i dx_1\wedge dx_2\wedge dy \\
& = \lim _{r\to 0} \int_{|y|=r} \delta^2 s_i dy = 0
\end{split}
\end{align}
assuming that $y=0$ at the point in $\cS\cap U_i$, if such point exists.
The third equality in \eqref{inteqn} is true because
\begin{align}
\int_{\cJ} \Omega_{\cJ} \wedge \omega_\delta^{\SF}=0, \  \int_{\cJ} \fJ^* \xi \wedge \omega_\delta^{\SF}=\delta^2 \int_{\PP^1} \xi, \  \int_{\cJ} \fJ^* \xi \wedge \omega_\delta^{\mathbb{B}}=0,
\end{align}
and $\omega_\delta^{\mathbb{B}}$ is a representative of $\delta^2\mathbb{B}$. We conclude that $\int_{\PP^1} \xi = \int_{\PP^1} \alpha$. Since $H^2(\PP^1,\RR)=\RR$, $[\xi] = [\alpha] \in H^2 (\PP^1, \RR)$. So
\begin{equation}
[\Omega_{\cJ}^A] = [\Omega_{\cJ}+\fJ^*\xi] = [\Omega_{\cJ}+\fJ^*\alpha] = [(\Psi^{-1})^* \Omega_{\cK}] \in H^2 (\cJ, \RR).
\end{equation}

By the Torelli theorem for K3 surfaces \cite{ShapiroShafarevitch, BurnsRapoport}, we can define a complex structure on $\cK$ using $\Psi^*\Omega_{\cJ}^A$, which is biholomorphic to the original complex structure on $\cK$. Without loss of generality we may therefore assume that this biholomorphism is the identity map and $\Psi^*\Omega_{\cJ}^A=\Omega_{\cK}$ because we can always pull back forms and metrics using this biholomorphism. Then it is natural to define $\omega_\delta^A$ as $\Psi^*(L(s_i)^*\omega_\delta^{\SF})$ on $\fF^{-1}(U_i)$.

\begin{remark}
The $2$-form $\omega_\delta^A$ is K\"ahler with respect to the original complex structure on $\cK$ and therefore determines a Riemannian metric $g_\delta^A$.
\end{remark}

\section{Singular fibers with infinite monodromy}
\label{s:infinite-monodromy}
In the previous section, we have described a process to get a $(1,1)$-form $\omega_\delta^A$ on $\cK$ using the 2-form $\omega_\delta^{\SF}$ on $\cJ$ by $\omega_\delta^A=\Psi^*(L(s_i)^*\omega_\delta^{\SF})$. In this
section, we will give the construction to resolve singular fibers with infinite monodromy of type $\I_{\nu}$ and $\I_{\nu}^*$ for $\nu\in\dZ_+$. The gluing construction near the singular fibers with finite monodromy will be done in Section \ref{s:finite-monodromy}.

\subsection{Resolving fibers of type I$_{\nu}$}
\label{ss:mov}
We begin with the following lemma regarding the Green's function with multiple poles on a product flat manifold $\dR^2\times S^1$,
which will be the key input for the Gibbons-Hawking ansatz.
\begin{lemma}
\label{l:Green's-function}
Let $(Q^3, g_{Q^3})$ be a product space $Q^3\equiv\dR^2\times S^1=\dR^2\times \dR/\dZ$ with the flat product Riemannian metric $ g_{Q^3} \equiv du_1^2 + du_2^2 + du_3^2$. Given a finite set $P \equiv\{p_1,\ldots, p_\nu\} \subset \{0\}\times S^1$, there exists a unique Green's function $G_\nu$  on $\dR^2\times S^1$ such that
\begin{enumerate}
\item $-\Delta_{g_{Q^3}} G_\nu  = 2\pi\sum\limits_{i=1}^{\nu} \delta_{p_i}$.
\item There are constants $R>0$ and $C>0$   such that
\begin{equation}
	\Big|G_\nu(u) -  \nu\log |u_1+\sqrt{-1}u_2|^{-1} \Big| \leq C e^{-2\pi \cdot |u_1+\sqrt{-1}u_2|}
\end{equation}
for any $u=(u_1,u_2,u_3)\in\dR^2\times S^1$ satisfying $d_{g_{Q^3}}(u, P)\geq R$.
\item For any $(u_1,u_2)\in\dR^2$,
\begin{equation}
\int_{\{(u_1, u_2)\}\times S^1} \Big(G_\nu(u_1,u_2,u_3) -  \nu \log |u_1+\sqrt{-1} u_2|^{-1}\Big)du_3=0.
\end{equation}
 \end{enumerate}
\end{lemma}

\begin{proof} This is obtained using superposition of the Green's functions in  \cite[Lemma~3.1]{GW}.
\end{proof}
Next, let $\fF:\cK\to\PP^1$ be an elliptic K3 surface. Near a singular fiber $\fF^{-1}(p)$ of Type  $\I_\nu$, there exists a local holomorphic section $\sigma_i$. We can choose a local coordinate $y$ on the base $\PP^1\setminus\cS$ and a local coordinate $x$ on the universal cover of the $\dT^2$-fiber which gives the standard holomorphic 2-form $\Omega_{\can}=-dx\wedge dy$. Assume that $y=0$ at $p\in\cS$ and $x=0$ on $\sigma_i$. Let $\tau_1(y)$ and $\tau_2(y)$ be the two functions of periods, which are holomorphic in $y$. Assume that $\tau_1$ is single-valued and $\tau_2$ is multi-valued. After replacing $y$ by $\int_0^y\tau_1(z)dz$, we can without loss of generality assume that $\tau_1(y)=1$. Then \begin{equation}\tau_2(y)=\frac{\nu}{2\pi\sq}\log|y|+h(y)\label{e:period-tau_2}\end{equation}
for some holomorphic function $h(y)$.

Throughout this section, for fixed $\nu\in\dZ_+$, we always relate the parameters $0<\delta\ll1$ and $T\gg1$ by
\begin{equation}
T = -\nu \log \delta.
\end{equation}
The following is an obvious generalization of \cite[Proposition~3.2]{GW} to the case of several monopole points,
and the construction is known as the Gibbons-Hawking ansatz.
\begin{proposition}
[Gibbons-Hawking ansatz]\label{p:multi-OV} Let $Q^3 \equiv \dR^2\times S^1$ be the product space equipped with
the flat metric $g_{Q^3} \equiv g_{\dR^2}\oplus g_{S^1}$ such that $\diam_{g_{Q^3}}(S^1)=1$.
Given a set of finite poles $P\equiv\{p_1,\ldots, p_\nu\}\subset \{0^2\}\times S^1$, let $G_P$
be the Green's function given by Lemma \ref{l:Green's-function}.
For any $T=-\nu\log\delta\gg1$ (equivalently $\delta\ll1$), let us define
\begin{equation}
V_T(u_1,u_2,u_3) \equiv T + G_P(u_1,u_2,u_3) + 2\pi \Ima h\Big(\delta(u_1+\sqrt{-1}u_2)\Big),
\end{equation}
where $h$ is the holomorphic function in \eqref{e:period-tau_2}.
Fix a small constant $\delta_0>0$ of definite size. Let $\cO$ be the set
\begin{equation}
\cO \equiv \Big\{u\in Q^3\Big| |\delta(u_1+\sqrt{-1}u_2)| \leq 2\delta_0 \Big\} \subset Q^3.
\end{equation}
Then the following holds:
\begin{enumerate}
\item $V_T>1$ on $\cO\setminus P$. Moreover, there is a principal $S^1$-bundle map
\begin{equation}
S^1\to \mathring{\mathcal{N}}_{\nu}^4 \xrightarrow{\pi} \cO\setminus P
\end{equation}
with a $S^1$-connection $1$-form $\theta$ satisfying the monopole equation
\begin{equation}
d\theta = *_{Q^3}\circ dV_T.
\end{equation}
Recall that a 1-form $\theta$ on $\mathring{\mathcal{N}}_{\nu}^4$ is called a connection 1-form if it is $S^1$-invariant and
\begin{equation}
\int_{\pi^{-1}(u)}\theta=2\pi
\end{equation}
for all $u\in \cO\setminus P$.
\item The above connection $1$-form $\theta$ induces a
hyperk\"ahler triple \begin{align}
\omega_{\delta,\nu} &\equiv \frac{1}{2\pi} (du_3\wedge\theta + V_T du_1\wedge du_2)\\
\omega_{2,\delta,\nu} &\equiv \frac{1}{2\pi} (du_1\wedge\theta + V_T du_2\wedge du_3)\\
\omega_{3,\delta,\nu} &\equiv \frac{1}{2\pi} (du_2\wedge\theta + V_T du_3\wedge du_1),
\end{align}
with associated incomplete Riemannian metric
\begin{align}
g_{\delta,\nu} &\equiv \frac{1}{2\pi} (V_T\cdot ( du_1^2 + du_2^2 + du_3^2 ) + V_T^{-1}\theta^2)
\end{align}
on the total space $\mathring{\mathcal{N}}_{\nu}^4$.
\item  The hyperk\"ahler metric $g_{\delta,\nu}$ extends smoothly to the closure $\mathcal{N}_{\nu}^4=\overline{\mathring{\mathcal{N}}_{\nu}^4}$.
\end{enumerate}

\end{proposition}
Given a parameter $0<\delta\ll1$, the hyperk\"ahler metrics $g_{\delta,\nu}$ given in the above proposition will be called the {\it multi-Ooguri-Vafa metric} in our discussions throughout the paper.
We also make the remark that there is some constant $C>0$ such that for sufficiently large $T=-\nu\log\delta$,
\begin{equation}
\frac{1}{C_{\nu}}\cdot \delta^{-1} \leq \diam_{g_{\delta,\nu}}(\mathcal{N}_{\nu}^4) \leq C_{\nu} \cdot  \delta^{-1}.
\end{equation}

Based on the above construction, we are ready to write down a family of collapsing incomplete Gibbons-Hawking metrics with prescribed scales for the collapsing $\dT^2$-fibers.
For sufficiently small $\delta$, let us rescale the Gibbons-Hawking metric $g_{\delta,\nu}$ by $g_{\delta,\nu}^{\flat} \equiv \delta^2\cdot g_{\delta,\nu}$. The hyperk\"ahler triple is rescaled by
\begin{equation}
(\omega_{\delta,\nu}^{\flat}, \omega_{2,\delta,\nu}^{\flat}, \omega_{3,\delta,\nu}^{\flat})\equiv \delta^2(\omega_{\delta,\nu}, \omega_{2,\delta,\nu}, \omega_{3,\delta,\nu}).
\end{equation}
 Then with respect to the rescaled metric $g_{\delta,\nu}^{\flat}$,  there is some constant $C_{\nu}>0$ such that for all $\delta>0$,  \begin{equation}
 \frac{1}{C_{\nu}} \leq  \diam_{g_{\delta,\nu}^{\flat}} (\mathcal{N}_{\nu}^4) \leq C_{\nu}.
 \end{equation}
The function $y=\delta(u_1+ \sq u_2)$ describes a map from $\mathcal{N}_{\nu}^4$ to a small neighborhood of $0$ in $\CC$, which is an elliptic fibration, and moreover, the area of each fiber using the rescaled  metric $g_{\delta,\nu}^{\flat}$ is
\begin{equation}
\frac{\delta^2}{2\pi}\int_{\pi^{-1}(\{(u_1,u_2)\}\times S^1)} (V_Tdu_1\wedge du_2 + du_3\wedge\theta) =\delta^2\int_{\{(u_1,u_2)\}\times S^1} du_3 = \delta^2.
\end{equation}
As in \cite{GW}, after the choice of a local holomorphic section $\sigma_{\mathcal{N}_{\nu}^4} $ on $\mathcal{N}_{\nu}^4$, there exists a local coordinate $x$ such that \begin{equation}
-\delta d x \wedge d y = \omega_{2,\delta,\nu}^{\flat}+\sqrt{-1}\omega_{3,\delta,\nu}^{\flat}.
 \end{equation}
and $x=0$ on the image of $\sigma_{\mathcal{N}_{\nu}^4}$.
By \cite[Proposition~3.2]{GW}, for a suitable choice of $\theta$, the two periods are $1$ and $\tau_2(y)=\frac{\nu}{2\pi \sq}\log|y|+h(y)$ because
\begin{equation}
\frac{1}{2\pi}\int_{\{(u_1,u_2)\}\times S^1}V_T du_3=\Ima \tau_2(y).
\end{equation}
So the periods on $\mathcal{N}_{\nu}^4$ are the same as the periods on $\cK$, and the map $(x_{\mathcal{N}_{\nu}^4},y_{\mathcal{N}_{\nu}^4})
\mapsto (x_{\cK},y_{\cK})$ is biholomorphic from $\mathcal{N}_{\nu}^4$ to an open subset of $\cK$ which maps $\sigma_{\mathcal{N}_{\nu}^4} $ to $\sigma_i$ and maps the 2-form $\omega_{2,\delta,\nu}^{\flat}+\sqrt{-1}\omega_{3,\delta,\nu}^{\flat}$ to the given holomorphic 2-form $\Omega$ on $\cK$.

We remark that the choice of $\sigma_{ \mathcal{N}_{\nu}^4} $ is not unique. By \cite[Lemma~4.3]{GW}, there exists a choice of $\sigma_{\mathcal{N}_{\nu}^4} $ such that for the biholomorphic map induced by $\sigma_{\mathcal{N}_{\nu}^4}$ ,
\begin{equation}\omega_\delta^{A}-\omega_{\delta,\nu}^{\flat}=\sqrt{-1}\partial\bar\partial\varphi_{\delta,\nu}^{\flat}\end{equation} for some function
$\varphi_{\delta,\nu}^{\flat}$ on $\{0<|y|<2\delta_0\}=\fF^{-1}(B_{2\delta_0}(p))\subset \cK$.

Now we return to the elliptic K3 surface $\fF:\cK\to\PP^1$
and assume that there is singular point $p\in \mathcal{S}$ such that the singular fiber $\fF^{-1}(p)$
is of Type $\I_{\nu}$
for some $\nu\in\dZ_+$.
For each sufficiently small parameter $0<\delta\ll1$, around the above singular fiber $\fF^{-1}(p)$, we will glue the semi-flat K\"ahler forms $\omega_\delta^{A}$
with the rescaled multi-Ooguri-Vafa K\"ahler forms $\omega_{\delta,\nu}^{\flat}$ constructed in Proposition \ref{p:multi-OV}.  This is analogous to the $\I_1$-case in \cite{GW}.

\begin{proposition}[Approximate metrics around $\I_{\nu}$-fibers] \label{p:gluing-I_v}
Let $\fF^{-1}(p)\subset \cK$ be a singular fiber  of Type $\I_{\nu}$ for $\nu\in\dZ_+$. For each sufficiently small parameter $0<\delta\ll1$ and $\delta_0$, around the singular fiber $\fF^{-1}(p)$, the semi-flat K\"ahler form
$\omega_\delta^{A}$ induced by $\fF:\cK\to \PP^1$ can be glued with
an incomplete multi-Ooguri-Vafa K\"ahler form $\omega_{\delta,\nu}^{\flat}\equiv\delta^2\cdot \omega_{\delta,\nu}$ on $\mathcal{N}_{\nu}^4$
such that
the glued K\"ahler form $\omega_\delta^{B}$ on  $
  \mathcal{V}_{4\delta_0}\equiv \fF^{-1}(B_{4\delta_0}(p))\subset \cK$
is constructed as follows. Let $\varphi_{\delta,\nu}^{\flat}$ be the potential on $\fF^{-1}(B_{4\delta_0}(p)\setminus\{p\})$ satisfying
\begin{equation}
\omega_\delta^{A} = \omega_{\delta,\nu}^{\flat} + \sq\p\bp\varphi_{\delta,\nu}^{\flat}.
\end{equation}
Then
\begin{align}
\omega_\delta^{B}
\equiv
\begin{cases}
\omega_\delta^{A} & \text{in}\ \fF^{-1}(A_{2\delta_0,4\delta_0}(p)),
\\
\omega_{\delta,\nu}^{\flat} + \sq
\p\bp(\chi\cdot\varphi_{\delta,\nu}^{\flat}) & \text{in}\ \fF^{-1}(A_{\delta_0,2\delta_0}(p)),
\\
\omega_{\delta,\nu}^{\flat} & \text{in}\ \fF^{-1}(B_{\delta_0}(p)),
\end{cases}
\end{align}
where  $\chi$  is a cutoff function satisfying
\begin{align}
\chi =
\begin{cases}
0, & \text{in}\ \fF^{-1}(B_{\delta_0}(p)),
\\
1, &  \text{in}\ \cK\setminus  \fF^{-1}(B_{2\delta_0}(p)).
\end{cases}
\end{align}

\end{proposition}

\subsection{Resolving fibers of type I$_{\nu}^*$}
\label{ss:qmov}
In this subsection,  we will resolve the singular fibers of Type $\I_\nu^*$ ($\nu\in\dZ_+$).
Let $\fF:\cK\to\PP^1$ be an elliptic K3 surface such that the base $\PP^1$ has a finite singular set $\cS$.
Let $\Omega$ be a fixed holomorphic $2$-form on $\cK$ and assume that there is a singular fiber $\fF^{-1}(p)$ of Type $\I_{\nu}^*$ for some $p\in\cS$ and $\nu\in\dZ_+$.

First, we describe the orbifold structure induced by the  $\I_{\nu}^*$-singular fiber  $\fF^{-1}(p)$.
By \cite{Kodaira1963}, for some $\delta_0\ll 1$, there is an elliptic surface
\begin{align}\widetilde{\fF}: \widetilde\cK\longrightarrow B_{2\delta_0}(p)\subset \PP^1\end{align}
 with a local section $\widetilde\sigma_i:  B_{2\delta_0}(p)\to \widetilde{\cK}$ such that $\cK$ is locally biholomorphic to $\Res(\widetilde\cK/\ZZ_2)$, the resolution of
$\widetilde\cK/\ZZ_2$, near the singular fiber $\fF^{-1}(p)$. More precisely, let $(\tilde x, \tilde y)$ be the local coordinates of $\widetilde\cK$ such that $\tilde y$ is the local coordinate on the base and $\tilde x$ is the local coordinate on the fiber. Assume that $\{\tilde y=0\}$ corresponds to the singular fiber and $\{\tilde x=0\}$ gives the section $\tilde\sigma_i$. Then the $\ZZ_2$-action on $\widetilde{\cK}$ is given by
$(\tilde x, \tilde y)\mapsto (-\tilde x, -\tilde y)$ so that there are four fixed points
\begin{align}
\Sigma_0=\Big\{(0,0),\ \Big(\tilde\tau_1(0)/2,0\Big),\ \Big(\tilde\tau_2(0)/2,0\Big),\ \Big((\tilde\tau_1(0)+\tilde\tau_2(0))/2,0\Big)\Big\}\subset \widetilde{\cK}.\label{e:set-of-fixed-points}
\end{align}
Hence there is a biholomorphism
\begin{equation}
\mathfrak{R}: \Res(\widetilde\cK/\ZZ_2)\longrightarrow \mathcal{V},
\end{equation}
 from the resolution of the four orbifold singular points on $\widetilde\cK/\ZZ_2$
to some neighborhood $\mathcal{V}$ of $\fF^{-1}(p)$.

We remark that given the holomorphic $2$-form $\Omega$ on $\cK$,
the pullback $\mathfrak{R}^*\Omega$ on $\Res(\widetilde\cK/\ZZ_2)$
naturally induces
a holomorphic $2$-form $\widetilde{\Omega}$ on the covering space $\widetilde{\cK}$.
Indeed, denote by $\Sigma_{\dZ_2}$  the projection of the fixed point set $\Sigma_0$ in \eqref{e:set-of-fixed-points} under the quotient $\widetilde{\cK}\to \widetilde{\cK}/\dZ_2$. So there is map \begin{equation}\mathfrak{S}: (\widetilde\cK/\ZZ_2)\setminus \Sigma_{\dZ_2}\longrightarrow \Res(\widetilde\cK/\ZZ_2) \end{equation}
such that the puncture $(\widetilde\cK/\ZZ_2)\setminus \Sigma_{\dZ_2}$ is biholomorphic to its image under $\mathfrak{S}$. Now the pullback holomorphic $2$-form $\mathfrak{S}^*\mathfrak{R}^*\Omega$ on the puncture $(\widetilde\cK/\ZZ_2)\setminus \Sigma_{\dZ_2}$ can be extended to a holomorphic $2$-form $\widetilde{\Omega}$ on the $\dZ_2$-lifting $\widetilde{J}$.
Meanwhile, using the above pullback operations, for each $0<\delta<1$, the semi-flat metric $g_\delta^{A}$ on $\cK$ induces a $\dZ_2$-invariant semi-flat metric $\tilde{g}_\delta^{A}$ on $\widetilde\cK$ whose area of each regular fiber is $\delta^2$.

Now we are ready to describe the construction of the orbifold Gibbons-Hawking metrics on $\widetilde{\cK}/\dZ_2$.
Denote by
$Q^3=\dR^2\times S^1$ the product space equipped with the flat product metric $g_{Q^3}=du_1^2 + du_2^2+ du_3^2$.
Let $\iota: \dR^2\times S^1\to \dR^2\times S^1$ be the involution map given by
\begin{equation}
\iota: (u_1, u_2, u_3)\mapsto (-u_1, -u_2 ,-u_3),
\end{equation}
which has two fixed points $q_- = (0,0,0)$ and $q_+ = (0,0,\frac{1}{2})$.

Choose a finite set of disjoint poles \begin{align}P = \{p_1,\hat{p}_1,p_2,\hat{p}_2,\ldots, p_{\nu},\hat{p}_{\nu}\}\in (\{0^2\}\times S^1) \setminus\{q_-,q_+\}\end{align}
which satisfies $\iota (p_i) =\hat{p}_i$ for any $1\leq i\leq \nu$.
By Proposition \ref{p:multi-OV}, we obtain a multi-Ooguri-Vafa metric
\begin{equation}
g_{\delta,2\nu} = \frac{1}{2\pi}(V_T (du_1^2 + du_2^2 + du_3^2) + (V_T)^{-1}\theta^2)
\end{equation}
 on the completion ${\mathcal{N}}_{2\nu}^4= \overline{\mathring{\mathcal{N}}_{2\nu}^4}$ as in the $\I_{2\nu}$ case, where
\begin{equation}S^1\to \mathring{\mathcal{N}}_{2\nu}^4 \xrightarrow{\pi} \cO\setminus P\subset Q^3\end{equation} is a principal $\U(1)$-bundle. Then the two periods are $1$ and $\tilde\tau_2(\delta(u_1+ \sq u_2))$ after replacing $\theta$ by $\theta+cdu_3$ if necessary (as in \cite[Proposition~3.2]{GW}).
The group $\U(1)$ acts on
$\mathring{\mathcal{N}}_{2\nu}^4$ from the right, so we will denote
this action by $\mathscr{R}_{\gamma}$ for $\gamma \in \U(1)$.
Recall that a connection $\theta \in \Omega^1(\mathring{\mathcal{N}}_{2\nu}^4)$ is a real $1$-form
on $\mathring{\mathcal{N}}_{2\nu}^4$ such that $\sq \theta$ takes values
in the Lie algebra $\mathfrak{u}(1)\equiv \sq\cdot  \dR$. Now the connection satisfies the following properties:
\begin{enumerate}
\item $\sq \theta$ restricted to the fiber $\pi^{-1}(u)$ is
$\i \cdot d u_4$, where $u_4\in\U(1)=\RR/2\pi\ZZ$.
\item $\mathscr{R}_{\gamma}^* \theta = \theta$.
Since the group is abelian,
the {\em{curvature $2$-form}} of the connection is given
by $\W_{\theta} = d \theta \in \i H^2(\mathring{\mathcal{N}}_{2\nu}^4, \RR)$, and this forms descends to $\cO\setminus P$.
\end{enumerate}
Note that we have chosen the Green's function $V_T$ to be invariant under $\iota$.
Then
\begin{align}
\iota^* \Omega_{\theta} = \iota^* (*_{Q^3}\circ d V_T) = - *_{Q^3}\circ  dV_T = - \Omega_{\theta},
\end{align}
since $\iota$ is orientation-reversing on $\cO\setminus P$.

Since $\mathring{\mathcal{N}}_{2\nu}^4$ is a $\U(1)$-bundle, it has a first Chern
class $c_1(\mathring{\mathcal{N}}_{2\nu}^4) \in H^2(\cO\setminus P ; \ZZ)$. It is well known that $\mathring{\mathcal{N}}_{2\nu}^4$ is determined up to smooth bundle equivalence by $c_1(\mathring{\mathcal{N}}_{2\nu}^4)$.
By Chern-Weil theory, the image of $c_1(\mathring{\mathcal{N}}_{2\nu}^4)$ in $\i H^2(\cO\setminus P; \RR)$
is cohomologous to $\W_{\theta}$, for any connection $\theta$ on
$\mathring{\mathcal{N}}_{2\nu}^4$.

Consider the pull-back bundle  $\iota^* \mathring{\mathcal{N}}_{2\nu}^4$.
By naturality,
\begin{align}  c_1( \iota^* \mathring{\mathcal{N}}_{2\nu}^4) = \iota^* c_1(\mathring{\mathcal{N}}_{2\nu}^4) = \iota^* [ \W_{\theta} ] = - [\W_{\theta}] = - c_1(\mathring{\mathcal{N}}_{2\nu}^4).
\end{align}
Consequently, there exists a bundle equivalence $A : \iota^* \mathring{\mathcal{N}}_{2\nu}^4 \rightarrow \overline{\mathring{\mathcal{N}}_{2\nu}^4}$,
which is an equivariant map covering the identity map on $\cO\setminus P$,
and $\overline{\mathring{\mathcal{N}}_{2\nu}^4}$ is the conjugate bundle.
Denote by $\pi_2$ the natural map $\pi_2 : \iota^* \mathring{\mathcal{N}}_{2\nu}^4 \rightarrow \mathring{\mathcal{N}}_{2\nu}^4$,
and choose an identification $C :\mathring{\mathcal{N}}_{2\nu}^4 \rightarrow \overline{\mathring{\mathcal{N}}_{2\nu}^4}$, which is fiberwise
complex conjugation. This is summarized in the following diagram.
\begin{equation}
\label{cd1}
 \CD
\mathring{\mathcal{N}}_{2\nu}^4 @>{C}>>\overline{\mathring{\mathcal{N}}_{2\nu}^4}@>{A^{-1}}>>\iota^*\mathring{\mathcal{N}}_{2\nu}^4 @>{\pi_2}>>\mathring{\mathcal{N}}_{2\nu}^4\\
@VVV @VVV @VVV @VV{\pi}V\\
\cO\setminus P@>{\Id}>>\cO\setminus P@>{\Id}>>\cO\setminus P@>{\iota}>>\cO\setminus P,\\
 \endCD
 \end{equation}
The pull-back
\begin{align}
\theta' = C^* (A^{-1})^* \pi_2^* \theta
\end{align}
is a connection on $\pi: \mathring{\mathcal{N}}_{2\nu}^4 \rightarrow \cO\setminus P$.
Since $\pi_2 \circ A^{-1} \circ C $ covers $\iota$, we have
\begin{align}
\W_{\theta'} = d \theta' = d \big( ( \pi_2 \circ A^{-1} \circ C)^* \theta \big)
=  ( \pi_2 \circ A^{-1} \circ C)^* \W_{\theta} = \iota^* (\W_{\theta}) = - \W_{\theta}.
\end{align}
This shows that $d ( \theta' + \theta ) =0$. Since $\theta' + \theta$
descends to $\cO\setminus P$, we can write
\begin{align}
\theta' = - \theta - \sq  \cdot \pi^* df + c \sq  \pi^*(du_3),
\end{align}
for some constant $c \in \RR$, and $f : \cO\setminus P \rightarrow \RR$,
since $H^1(\cO\setminus P;\RR)$ is generated by $d u_3$.

Define a bundle map $B : \mathring{\mathcal{N}}_{2\nu}^4 \rightarrow \mathring{\mathcal{N}}_{2\nu}^4$ by
$B v =  v \cdot e^{ \sq  f}$ (right action).  Choosing a local fiber coordinate $u_4$ and writing
$\theta'$ as $\theta_1' + \i \cdot d u_4$, we have
\begin{align}\label{g_to_b}
\begin{split}
B^* \theta' = B^* (  \theta_1' + \i \cdot d u_4)
&=  \theta_1' + \i B^* d u_4 \\
&=  \theta_1'+ \i( d u_4 +  \pi^* d f )\\
&= \theta' +  \i  \cdot \pi^* d f = - \theta + c \i \pi^* (d u_3)
\end{split}
\end{align}
Therefore, the mapping
\begin{align}
\Psi = \pi_2 \circ A^{-1} \circ C \circ B
\end{align}
is a mapping covering $\iota$ which satisfies
\begin{align}
\label{Phieqn}
\Psi^* \theta = - \theta + c \i \pi^*d u_3.
\end{align}
Next, notice that $\Psi^2: \mathring{\mathcal{N}}_{2\nu}^4 \rightarrow \mathring{\mathcal{N}}_{2\nu}^4$ is
a bundle mapping covering the identity, so we must have
\begin{align}
\Psi^2 (v) = v \cdot e^{\i h}
\end{align}
for some function $h : \cO\setminus P \rightarrow \RR$.
As above, this implies that
\begin{align}
(\Psi^2)^* \theta = \theta + \i \pi^* d h.
\end{align}
But from \eqref{Phieqn}, we have
\begin{align}
\begin{split}
(\Psi^2)^* \theta = \Psi^* ( \Psi^* \theta)
&= \Psi^* ( - \theta + c \i \pi^*d u_3)\\
&= \theta - c \i  \pi^*d u_3 + c \i \Psi^* (\pi^*d u_3)\\
&= \theta - c \i  \pi^*d u_3 + c \i \pi^* \iota^* d u_3\\
& =  \theta - 2 c \i  \pi^*d u_3.
\end{split}
\end{align}
We conclude that
\begin{align}
\pi^* dh = -2 c \pi^* d u_3.
\end{align}
Equivalently,
\begin{align}
\pi^* ( d h + 2 c d u_3) = 0.
\end{align}
Since $\pi_*$ is surjective at any point on $\mathring{\mathcal{N}}_{2\nu}^4$, this implies that
\begin{align}
d h + 2 c d u_3 = 0
\end{align}
on $\cO\setminus P$, but since $H^1(\cO\setminus P;\RR)$ is generated by $d u_3$,
we conclude that $c = 0$, and $dh = 0$,
so $h = h_0 = constant$, since $\cO\setminus P$ is connected.

We have shown that
\begin{align}
\Psi^* \theta &= - \theta\\
\Psi^2(v) &= v \cdot e^{\sq   h_0}.
\end{align}
The first equation shows that $\Psi$ is an isometry of $g_{\delta,2\nu}$. Moreover, $\Psi$ fixes the complex structures $I, J, K$.
We next show that $h_0 = 0$, i.e., $\Psi$ is an involution.
\begin{proposition}
The lifting $\Psi$ satisfies $\Psi^2 = \Id$.
\end{proposition}

\begin{proof}
Let $u_4$ be a local coordinate on the $S^1=\RR/2\pi\ZZ$ fibration. Restrict to the fiber over $q_-=(0,0,0)$. Then $\Psi (0,0,0,u_4)=(0,0,0,C-u_4)$ for a constant $C$. But then $\Psi (0,0,0,u_4)=(0,0,0,u_4)$, which means that $h_0=0$.
\end{proof}
With the above involution $\Psi$ on $\mathcal{N}_{2\nu}^4$, we are able to define the orbifold Gibbons-Hawking metrics on the $\dZ_2$-quotient. For simplicity, first we study the metric $g_{\delta,2\nu}$ on $\mathcal{N}_{2\nu}^4$  at large scales such that
\begin{equation}
\underline{c}_0\cdot \delta^{-1}\leq \diam_{g_{\delta,2\nu}}(\mathcal{N}_{2\nu}^4) \leq \bar{c}_0 \cdot \delta^{-1}.
\end{equation}
Notice that there are four fixed points $\{q_1,q_2,q_3,q_4\}$ under $\Psi\in\dZ_2$ such that
\begin{equation}
\pi(q_1)=\pi(q_2)=q_-,\quad \pi(q_3)=\pi(q_4)=q_+.\end{equation}
Let $\dZ_2=\{\Id, \Psi\}$, and since the quotient $\mathcal{N}_{2\nu}^4/\ZZ_2$ is an elliptic fibration over the orbifold $B_{2\delta_0/\delta}(0^2)/\dZ_2$ (the radius is given in Proposition \ref{p:multi-OV}), by \cite{Kodaira1963}, there exists a local holomorphic section. Its preimage is a $\ZZ_2$-invariant local holomorphic section. By the proof of  \cite[Lemma~4.3]{GW}, there exists another  $\ZZ_2$-invariant local holomorphic section  such that if we identify $\widetilde\cK$ with $\mathcal{N}_{2\nu}^4$ using this new holomorphic section, then
\begin{align}\tilde{\omega}_\delta^{A}-\omega_{\delta,2\nu}=\i\partial\bar\partial\varphi_{\delta,2\nu}
\end{align} for some function
$\varphi_{\delta,2\nu}$ on $\{0<|\tilde y|<2\delta_0\}$. After replacing
$\varphi_{\delta,2\nu}$ by $\frac{1}{2}(\varphi_{\delta,2\nu}+\Psi^*\varphi_{\delta,2\nu})$, we can assume that $\varphi_{\delta,2\nu}$ is $\ZZ_2$-invariant. So it descends to a biholomorphic map between the orbifolds $\widetilde\cK/\ZZ_2$ and $\mathcal{N}_{2\nu}^4/\ZZ_2$, then the Ooguri-Vafa metric
$g_{\delta,2\nu}$
descends to an orbifold hyperk\"aher metric
$\check{g}_{\delta,\nu}$ on $\mathcal{N}_{2\nu}^4/\dZ_2$.

Now the hyperk\"ahler orbifold $(\mathcal{N}_{2\nu}^4/\dZ_2, \check{g}_{\delta,\nu})$,
as the $\dZ_2$-quotient of the smooth Gibbons-Hawking region $(\mathcal{N}_{2\nu}^4,g_{\delta,2\nu})$,  has $4$ orbifold singularities with tangent cone isometric to $\dR^4/\dZ_2$. Correspondingly, the orbifold K\"ahler form $\check{\omega}_{\delta,\nu}$
is related to the semi-flat K\"ahler form $\omega_\delta^{A}$
by
\begin{equation}
\omega_\delta^{A}=\check{\omega}_{\delta,\nu}+\sq \p\bp\check{\varphi}_{\delta,\nu}.
\end{equation}

In a small neighborhood of the singular fiber, $\cK$ is obtained complex analytically by the resolution of these orbifold singularities,
In the next subsection, we describe the resolution of $\mathcal{N}_{2\nu}^4 / \dZ_2$ and we will construct a family of approximately hyperk\"ahler metrics on it, by gluing on Eguchi-Hanson metrics.

For carrying out the gluing construction, we need the following approximation estimates for the orbifold metric and the Eguchi-Hanson metric. Given an orbifold singularity $q_{\lambda}\in \mathcal{N}_{2\nu}^4/\dZ_2$ for $\lambda\in\{1,2,3,4\}$, denote by $r_{\lambda}(\bx)\equiv d_{\check{g}_{\delta,\nu}}(\bx,q_{\lambda})$
the distance between $\bx$ and $q_{\lambda}$ under the orbifold metric $\check{g}_{\delta,\nu}$.
From now on, we fix sufficiently small parameters $\ee_{\lambda} \in (0,1)$ for each $\lambda\in\{1,2,3,4\}$ which will be eventually determined in Section \ref{ss:proof-of-existence}.

Near $q_\lambda$, the (1,1)-form $\check{\omega}_{\delta,\nu}$ can be written as $\i\p\bp \check{\psi}_{\delta,\nu}$ locally. In the local normal complex coordinates $\{\xi_l\}_{l=1}^2$,
\begin{align}
\Big|\nabla_{g_{\dC^2/\dZ_2}}^k(\check{\psi}_{\delta,\nu} - \frac{|\xi_1|^2+|\xi_2|^2}{2})\Big|_{g_{\dC^2/\dZ_2}} =
\begin{cases}
 O(r_{\lambda}^{4-k}), & k=0, 1, 2, 3
 \\
O(1), & k\geq 4,
\end{cases}\label{e:orbifold-approximation}
\end{align}
 where
$r_{\lambda}\in [0,\ee_{\lambda}]$ and $\ee_{\lambda}\in(0,1)$ satisfies
\begin{align} \ee_{\lambda}\leq  \frac{\eta_0}{(\nu\log(1/\delta))^{\frac{1}{2}}}\label{e:orbifold-parameter}\end{align}
for some sufficiently small constant $\eta_0>0$ independent of $\delta>0$.
Note that the above upper bound of $\ee_{\lambda}$ comes from the length of the collapsing $S^1$-fiber at the orbifold singular point $q_{\lambda}$ which is comparable with $(\nu\log(1/\delta))^{-\frac{1}{2}}$.

Next we exhibit some estimates for the hyperk\"ahler Eguchi-Hanson space $(X^4, g_{\EH})$. To begin with,  let $\CC^2/\dZ_2$ be a flat cone so that we blow up the origin to get the minimal resolution. Take  the standard Euclidean coordinates $(z_1, z_2)$ of $\CC^2$ and let $r(z_1,z_2)\equiv\sqrt{|z_1|^2+|z_2|^2}$ be the distance to the origin. Then the K\"ahler potential for the Eguchi-Hanson metric is explicitly given by (see \cite{BM} for instance),
\begin{equation}
\varphi_{\EH}(z_1, z_2) \equiv \frac{1}{2} \Big(\sqrt{1+r^4} + 2\log r - \log(1+\sqrt{1+r^4})\Big)
\end{equation}
so that the K\"ahler form of the Eguchi-Hanson metric $g_{\EH}$
is given by $\omega_{\EH}\equiv \i\partial\bar\partial\varphi_{\EH}$. Note that as $r\to+\infty$, $\omega_{\EH}$ is very close to the Euclidean K\"ahler form $\omega_{\dR^4/\dZ_2}=\sq\partial\bar\partial( \frac{r^2}{2})$ of the flat cone $\dC^2/\dZ_2\equiv\dR^4/\dZ_2$
with the asymptotic order \begin{align}
\Big|\nabla_{g_{\dC^2/\dZ_2}}^k(\varphi_{\EH} - \frac{1}{2}r^2)\Big|_{g_{\dC^2/\dZ_2}} \leq \frac{C_k}{r^{k+2}},
\label{e:EH-approximation}\end{align}
for all $k \in \mathbb{N} = \{0,1,2,3,.\dots\}$, as  $r\to \infty$.

Let $(X^4,g_{\EH})$ be the hyperk\"ahler Eguchi-Hanson space.  $X^4\setminus \Phi\Big([2\ee_{\lambda}^{-1},+\infty)\Big)$ be a large region in the Eguchi-Hanson space with diameter
comparable with $\ee_{\lambda}^{-1}$. To glue this piece with the above Gibbons-Hawking region, we need to rescale the metric as follows.
 Denote by $\cO_{\ee_{\lambda}\cdot\delta}(q_{\lambda})$ the rescaled region
 with the rescaled Eguchi-Hanson metric $g_{\EH}^{\flat} \equiv\ee_{\lambda}^4 \cdot \delta^2\cdot  g_{\EH}$. Then the diameter has the scale
 \begin{equation}
\frac{1}{D_0}\cdot \ee_{\lambda}\cdot  \delta \leq \diam_{g_{\EH}^{\flat}}(\cO_{\ee_{\lambda}\cdot  \delta}(p)) \leq D_0\cdot \ee_{\lambda} \cdot  \delta
 \end{equation}
for some uniform constant $D_0>0$ independent of
$\delta$ and
$\ee_{\lambda}$.

Now we return to the elliptic K3 surface $\fF:\cK\to\PP^1$
and assume that there is singular point $p\in \mathcal{S}$ such that the singular fiber $\fF^{-1}(p)$
is of Type $\I_{\nu}^*$ for some $\nu\in\dZ_+$.
In the following, we glue the semi-flat metric with an orbifold Ooguri-Vafa metric, and further resolve the orbifold singularities by gluing $4$ copies of hyperk\"ahler Eguchi-Hanson metrics.

\begin{proposition}[Approximate metrics around $\I_{\nu}^*$-fibers] \label{p:gluing-I_v*}
 Let $\fF^{-1}(p)\subset \cK$ be a singular fiber of Type $\I_{\nu}^*$ for $\nu\in\dZ_+$, then for each sufficiently small parameters $0<\delta, \delta_0 \ll 1$, there is a K\"ahler form $\omega_\delta^{B}$ on  $\mathcal{V}_{4\delta_0}\equiv \fF^{-1}(B_{4\delta_0}(p))\subset \cK$   around the singular fiber $\fF^{-1}(p)$ which can be constructed by gluing the semi-flat K\"ahler form $\omega_\delta^{A}$ induced by $\fF:\cK\to \PP^1$, the rescaled orbifold  multi-Ooguri-Vafa K\"ahler form $\check{\omega}_{\delta,\nu}^{\flat}=\delta^2\check{\omega}_{\delta,\nu}^{\flat}$ on $\mathcal{N}_{2\nu}^4/\dZ_2$ and $4$ copies of rescaled Eguchi-Hanson K\"ahler forms $\omega_{\EH}^{\flat}=\delta^2\cdot\ee_{\lambda}^4\cdot\omega_{\EH}$ around the orbifold singularities on $\mathcal{N}_{2\nu}^4/\dZ_2$.  Let $\check{\varphi}_{\delta,\nu}^{\flat}$ be the K\"ahler potential on $\fF^{-1}(B_{4\delta_0}(p)\setminus\{p\})$ such that the semi-flat K\"ahler form $\omega_\delta^{A}$ differs from the orbifold multi-Ooguri-Vafa K\"ahler form $\check{\omega}_{\delta,\nu}^{\flat}$ by
\begin{equation}
\omega_\delta^{A} = \check{\omega}_{\delta,\nu}^{\flat} + \sq\p\bp\check{\varphi}_{\delta,\nu}^{\flat}.
\end{equation}
Then
\begin{align}
\omega_\delta^{B}
=
\begin{cases}
\omega_\delta^{A} & \text{in}\ \fF^{-1}\Big(A_{2\delta_0,4\delta_0}(p)\Big),
\\
\check{\omega}_{\delta,\nu}^{\flat} + \sq
\p\bp(\chi\cdot\check{\varphi}_{\delta,\nu}^{\flat}) & \text{in}\ \fF^{-1}(A_{\delta_0,2\delta_0}(p)),
\\
\check{\omega}_{\delta,\nu}^{\flat} & \text{in}\ \fF^{-1}(B_{\delta_0}(p))\setminus \bigcup\limits_{\lambda=1}^4 B_{2\ee_{\lambda}\delta}(q_{\lambda}),
\\
\sq \p \bp \Big((1-\check{\chi})\cdot \check{\psi}_{\delta,\nu}^{\flat}+\check{\chi} \cdot \varphi_{\EH}^{\flat}\Big)  & \text{in}\ \bigcup\limits_{\lambda=1}^4 B_{2\ee_{\lambda}\delta}(q_{\lambda}),
\end{cases}
\end{align}
where
$\check{\psi}_{\delta,\nu}^{\flat}=\delta^2 \cdot \check{\psi}_{\delta,\nu}$ is the rescaled K\"ahler potential of $\check{\omega}_{\delta,\nu}^{\flat}$, $\varphi_{\EH}^{\flat}=\mathfrak{e}_{l}^4 \cdot \delta^2 \cdot \varphi_{\EH}$ is the rescaled K\"ahler potential of the Eguchi-Hanson metric,
$\chi$ is a cutoff function satisfying
\begin{align}
\chi=
\begin{cases}
0, & \text{in}\ \fF^{-1}(B_{\delta_0}(p)),
\\
1, & \text{in}\ \cK\setminus\fF^{-1}(B_{2\delta_0}(p)),
\end{cases}
\end{align}
and
$\check{\chi}$ is a cutoff function satisfying
\begin{align}
\check{\chi}  =
\begin{cases}
1, & \text{in}\ \bigcup\limits_{\lambda=1}^4 B_{\ee_{\lambda}\delta}(q_{\lambda}),
\\
0, & \text{in}\ \cK\setminus\bigcup\limits_{\lambda=1}^4 B_{\ee_{\lambda}\delta}(q_{\lambda}).
\end{cases}
\end{align}
\end{proposition}

\begin{remark}
The $2$-form $\omega_\delta^B$ is K\"ahler with respect to the original complex structure on $\cK$ and therefore determines a Riemannian metric $g_\delta^B$,
which is hyperk\"ahler outside the gluing transition regions or ``damage zones.''
In Section~\ref{ss:weighted-error}, we will prove the uniform estimates
in appropriate weighted H\"older spaces for the deviation of $g_\delta^B$ from being hyperk\"ahler in the damage zones  (see Proposition \ref{p:weighted-error}).
\end{remark}

\section{Singular fibers with finite monodromy}

\label{s:finite-monodromy}

Given an elliptic K3 surface $\fF: \cK \to \PP^1$ with a fixed holomorphic $2$-form $\Omega$,
we have constructed a family of collapsing metrics $g_\delta^B$ which are defined away from singular fibers with finite monodromy.
 Our main goal in this section is to construct a family of collapsing metrics
$g_\delta^C$ associated to definition triples $(\omega_\delta^C, \Rea(\delta \cdot \Omega), \Ima(\delta \cdot \Omega))$ on $\cK$ which are also well-defined near the singular fibers with finite monodromy.
The main technical point is that, near the singular fibers with finite monodromy,
we will construct $\omega_\delta^C$ by gluing $\omega_\delta^B$ with ALG Ricci-flat K\"ahler forms, so that
some effective error estimates hold in the gluing process and the metrics $g_\delta^C$ are collapsing to the McLean metric $d_{ML}$ on $\PP^1$, that is,
\begin{equation}
(\cK,g_\delta^C)\xrightarrow{GH}(\PP^1,d_{ML}).
\end{equation}

\subsection{Estimates on the semi-flat metric}
\label{ss:semi-flat-estimates}
In the following discussion, we will assume for example that the singular fiber is of Type $\IV$.
The constructions and estimates in other cases are very similar, so we do not include a detailed analysis of all cases here.

 Near a singular fiber $\fF^{-1}(p)$ of Type  $\IV$, we use the local holomorphic section $\sigma_i$ from Section~\ref{s:elliptic-fibration}. We can choose a local coordinate $y$ on the base $\PP^1\setminus\cS$ and a local coordinate $x$ on the universal cover of the $\dT^2$-fiber such that $-dx\wedge dy$ is the pull back of $\Omega$, the given holomorphic 2-form on $\cK$. Assume that $y=0$ at $p\in\cS$ and $x=0$ on $\sigma_i$.
We define new coordinates by
\begin{align}
(u,v) = \Big(y^{2/3}, \frac{3}{2}xy^{1/3} \Big),
\end{align}
and note that $\Omega=du\wedge dv$.

Recall that the periods are defined as $\tau_i(y) = \int_{\gamma_i} \alpha_y$, where $\alpha_y$ is  holomorphic $(1,0)$-form on the fiber, and $\gamma_i$ are a basis of the first homology of the torus fiber, $i = 1,2$.
In Section~\ref{s:semi-flat-metrics}, we chose $\alpha_y = d x$ to be the canonical form to define the periods. However, in the following, we will instead choose the holomorphic $(1,0)$ form by  $\alpha_y =\frac{3}{2} y^{\frac{1}{3}} dx$ to define the periods.

The function $\varrho=\tau_2/\tau_1$ is a multi-valued holomorphic function in $y$. By \cite{Kodaira1963}, we can make a new choice of $y$ such that there exists $h\equiv 2\pmod 3$ satisfying
\begin{align}
\varrho(y)=\frac{\tau_2}{\tau_1}=\frac{e^{\i \frac{2\pi}{3}}-e^{-\i \frac{2\pi}{3}}y^{h/3}}{1-y^{h/3}}.
\end{align}

Note that the notation used in \cite{Kodaira1963} is slightly different from ours. In \cite{Kodaira1963}, the coordinate $\zeta$ on the fiber is given as $\zeta \in \CC/(\ZZ\oplus\ZZ\varrho)$ while in our notation, the coordinate $v$ on the fiber is given by $v \in \CC/(\ZZ\tau_1\oplus\ZZ\tau_2)$. Then relationship is simply $v = \tau_1\zeta$.

Letting $\varsigma = y^{1/3}$, then $\varrho(y)=\varrho(\varsigma^3)$ is a single valued function of $\varsigma$. So we can define
\begin{align}
\cF = \{(\varsigma,\zeta): \varsigma\in\CC, \zeta\in\CC/(\ZZ\oplus\ZZ\varrho)\},
\end{align}
and consider the $\ZZ_3$-action on $\cF$ given by
\begin{equation}
(\varsigma, \zeta) \rightarrow \Big(e^{\i \frac{2\pi}{3}}\varsigma, \frac{\zeta}{-\varrho-1} \Big).
\end{equation}
Note that there are exactly three fixed points on $\cF$, given by
\begin{align}
p_0=(0,0), \ p_1=(0,\frac{1}{3}e^{\i \frac{2\pi}{3}}+\frac{2}{3}), p_2=(0,\frac{2}{3}e^{\i \frac{2\pi}{3}}+\frac{1}{3}).
\end{align}
Let $\widetilde{\cF/\ZZ_3}$ be the blow-up of $\cF/\ZZ_3$ at the three fixed points as in \cite{Kodaira1963}. Letting $\cF_0$ denote
the central fiber, then $\cF_0/\ZZ_3$ induces an exceptional curve $\Theta$ on $\widetilde{\cF/\ZZ_3}$.
Then $\cK$ is locally the blow-down of $\widetilde{\cF/\ZZ_3}$ at $\Theta$.  A computation shows that there is a
$\ZZ_3$ invariant holomorphic form
\begin{align}
\Omega_{\cF}=2\varsigma (1-\varsigma^h) d\varsigma\wedge d\zeta
\end{align}
 on $\cF$. We claim that the ratio between the pull back of $\Omega_{\cK}$ to $\widetilde{\cF/\ZZ_3}$ and the pull back of $\Omega_{\cF}$ to $\widetilde{\cF/\ZZ_3}$ is never 0 and never infinity. In fact, the vanishing order of $\Omega_{\cF}$ is 1 on $\Theta$ because $\Theta$ can be written as $\{\varsigma=0\}$ locally. The vanishing order of $\Omega_{\cK}$ is also 1 on $\Theta$ because $\Theta$ is an exceptional curve. As for the vanishing order of $\Omega_{\cF}$ on the preimage of $p_0$, we can define $\zeta_0=(1-\varsigma^h)\zeta$ as in \cite{Kodaira1963}, then $\Omega_{\cF}=2\varsigma d\varsigma\wedge d\zeta_0$ and moreover, the $\ZZ_3$ action maps $(\varsigma, \zeta_0)$ to $(e^{\i \frac{2\pi}{3}}\varsigma, e^{\i \frac{2\pi}{3}}\zeta_0)$. Note the resolution is locally called $N_{+3}$ in \cite[page~582]{Kodaira1963}. It is easy to see that the vanishing order of $\Omega_{\cF}=2\varsigma d\varsigma\wedge d\zeta_0$ is 0 on the preimage of $p_0$. It is the same as the vanishing order of $\Omega_{\cK}$.

A similar calculation is also true on $p_1$ and $p_2$, Thus the ratio of these $2$-forms
is a non-zero holomorphic function $k$. It is invariant on the fiber direction since a holomorphic function on compact manifold is a constant. So $k$ is in fact a holomorphic function in $y$, and we can write $k(y)=k(0)+O(|y|),$ as $|y| \to 0$. Thus the given $(2,0$)-form $\Omega_{\cK}$ on $\cK$ can be written as
\begin{equation}
\Omega_{\cK}=2k(y)\varsigma (1-\varsigma^h) d\varsigma\wedge d\zeta,
\end{equation}
which implies that
\begin{equation}
\label{tau1yeqn}
\tau_1=k(y)(1-\varsigma^h)=k(u^{3/2})(1-u^{h/2})
\end{equation} and
\begin{equation}
\label{tau2yeqn}
\tau_2=k(u^{3/2})(e^{\i \frac{2\pi}{3}}-e^{-\i \frac{2\pi}{3}}u^{h/2}).
\end{equation}
After rescaling $y$ as well as $u=y^{2/3}$ and $\varsigma=y^{1/3}$, we can assume, without loss of generality, that
\begin{equation}
\label{tau1eqn}
\tau_1=\frac{1}{\sqrt{\Ima\tau}}(1+O(|u|))
\end{equation} then
\begin{equation}
\label{tau2eqn}
\tau_2=\frac{1}{\sqrt{\Ima\tau}}(\tau+O(|u|)),
\end{equation}
as $|u| \to 0$, where $\tau=e^{\i \frac{2\pi}{3}}$ in the $\IV$ case, and we have used $h\ge 2$ in the estimate. Note that this expansion  is also given in \cite[Table~1]{Hein} which includes analogous expansions for the other fiber types.
\begin{definition}
\label{def:distortion}
The {\it distortion order} $\lambda_{\beta}$ is the smallest constant so that
\begin{align}
\Ima(\bar\tau_1\tau_2) -1 = O(|u|^{\lambda_{\beta}})
\end{align}
as $|u| \to 0$, in the coordinates  $(u,v) = (y^{\beta},\frac{1}{\beta}xy^{1-\beta})$,
and where the periods are computed with respect to $\alpha_y = \frac{1}{\beta}y^{1-\beta} dx$.
The list of the optimal distortion orders is found in Table~\ref{Distortiontable}. \end{definition}
\begin{table}[h]
\caption{Distortion order}
\label{Distortiontable}
 \renewcommand\arraystretch{1.2}
\begin{tabular}{|c|c|c|c|c|c|c|c|c|} \hline
$0$ & $\I_0^*$ & $\II^*$ & $\II$ & $\III^*$ & $\III$ & $\IV^*$ & $\IV$\\\hline

 $\beta$ &  $\frac{1}{2}$  & $\frac{1}{6}$ & $\frac{5}{6}$ & $\frac{1}{4}$ & $\frac{3}{4}$ & $\frac{1}{3}$ & $\frac{2}{3}$\\
 [5pt]\hline
 \multirow{2}{*}{$\lambda_{\beta}$} & \multirow{2}{*}{$2$}  & \multirow{2}{*}{$4$}& \multirow{2}{*}{$\frac{2}{5}$} & \multirow{2}{*}{$2$}& \multirow{2}{*}{$\frac{2}{3}$} & \multirow{2}{*}{$1$}&\multirow{2}{*}{$1$} \\
&&  & &  &   & & \\\hline
  \end{tabular}
\end{table}

Our next goal is to estimate the difference between the $2$-form $\omega_\delta^B$ and the flat model $2$-form $\omega_{\delta, \cK}^{\FF}$, in the flat model metric, which are defined as
\begin{align}
\omega_{\delta, \cK}^{\FF} &=  \delta^2\cdot dv_1\wedge dv_2  + \frac{\sq}{2} \cdot du \wedge d \bar u\\
g_{\delta, \cK}^{\FF} &= \delta^2 ( dv_1^2 + dv_2^2) + du_1^2 + du_2^2.
\end{align}

For the gluing construction, we also need the following refined estimate
in an annular region surrounding the singular fiber.
\begin{lemma}[Annulus estimate]\label{l:annulus} Given an elliptic K3 surface $\fF:\cK\to\PP^1$ with a finite singular set $\cS\subset \PP^1$.
Let $\fF^{-1}(p)$ be a singular fiber with finite monodromy for some $p\in\cS$. For any sufficiently small $r_0>0$, we choose an annulus
\begin{equation}
\mathcal{A}_{\frac{r_0}{2},4r_0}(p) \equiv \fF^{-1}(A_{\frac{r_0}{2},4r_0}(p)),\quad A_{\frac{r_0}{2},4r_0}(p) \equiv \{\frac{r_0}{2}<|u|<4r_0\}\subset \PP^1\setminus \cS.
\end{equation}
There exists a $1$-form $\eta^B\in \Omega^1(\mathcal{A}_{r_0,2r_0}(p))$, independent of $\delta$, such that
\begin{equation}
\omega_\delta^B - \omega_{\delta,\mathcal{K}}^{\FF} = d\eta^B,
\end{equation}
and for any $k\in \mathbb{N}$,
\begin{equation}
\sup\limits_{\mathcal{A}_{r_0,2r_0}(p)}\Big|\nabla_{g_{\delta, \cK}^{\FF}}^k (\eta^B)\Big|_{g_{\delta, \cK}^{\FF}} \leq C_k\cdot r_0^{\lambda_\beta+1-k},
\end{equation}
where $C_k>0$ is independent of $\delta$.

\end{lemma}

\begin{proof}

Recall that in Section \ref{ss:construction-semi-flat}, using the real coordinates $(v_1,v_2)$ along the $\dT^2$-fibers, the collapsing semi-flat K\"ahler forms $\omega_\delta^{B}$ are given by
\begin{align}
\label{e:simplified-semi-flat}
\begin{split}
\omega_\delta^{B} &= \delta^2\cdot dv_1\wedge dv_2  + \frac{\sq}{2} \cdot \Ima(\bar\tau_1\tau_2) \cdot d u \wedge d \bar u\\
& =  \omega_{\delta, \cK}^{\FF} + \frac{\sq}{2} \cdot \big(\Ima(\bar\tau_1\tau_2) -1 \big) \cdot d u \wedge d \bar u.
\end{split}
\end{align}
From Definition \ref{def:distortion},  we have
\begin{align}
\Ima(\bar\tau_1\tau_2) -1 = O(|u|^{\lambda_{\beta}}),
\end{align}
as $|u| \to 0$.
Note that the difference
\begin{equation}
\omega_\delta^B - \omega_{\delta, \cK}^{\FF}
= \frac{\sq}{2}\cdot \Big(\Ima(\bar\tau_1\tau_2)-1\Big)d u \wedge d\bar u = \Big(\Ima(\bar\tau_1\tau_2)-1\Big) rdr\wedge d\theta
\end{equation}
is in fact the pullback of a closed $2$-form on the base $\PP^1$, where $u=re^{\i \theta}$ for $\theta\in[0,2\pi\beta]$. Define
\begin{equation}
\eta(r e^{\i \theta}) \equiv \Big(\int_{r_0}^r \{ \Ima(\bar\tau_1(s e^{\i \theta})\tau_2(s e^{\i \theta}))-1)\} s d s \Big) d \theta.
\end{equation}
Then $\eta^B \equiv \fF^* \eta$ satisfies the required estimate using the expansions of $\tau_1$, $\tau_2$
given in \eqref{tau1yeqn}, \eqref{tau2yeqn}.
\end{proof}

\subsection{ALG hyperk\"ahler 4-manifolds}
\label{ss:ALG}

In this subsection, we introduce some background material regarding ALG spaces. To begin with, let us define the notion of {\it standard ALG model}. Roughly, it is a singular flat space in dimension $4$, which can be viewed as a $\dT^2$-bundle over a flat sector in $\dR^2$.

  \begin{definition}[Standard ALG model]\label{d:ALG-model}
  Let $\beta\in(0,1]$ and $\tau\in\mathbb{H}\equiv\{\tau\in\dC|\Ima\tau>0\}$ be the parameters in Table~\ref{ALGtable}.
 Let $\cC_{\beta,\tau}$ be the  manifold obtained by identifying $(\mathscr{U},\mathscr{V})$ with $(e^{\sq\cdot 2\pi  \beta}\mathscr{U},e^{- \sq  \cdot 2\pi\beta}\mathscr{V})$ in the space \begin{equation}\{(\mathscr{U},\mathscr{V}) \ | \ \mathrm{arg} \mathscr{U}\in[0,2\pi\beta]\}\subset(\mathbb{C}\times \mathbb{C})/(\dZ\oplus \dZ),
\end{equation}
where $\dZ\oplus \dZ$ acts on $\mathbb{C}\times \mathbb{C}$ by
\begin{equation}
(m,n)\cdot (\mathscr{U},\mathscr{V})= \Big(\mathscr{U}, \mathscr{V}+\frac{m+n\tau}{\sqrt{\Ima \tau}}\Big),\ (m,n)\in\dZ\oplus \dZ.
\end{equation}
Then there is a singular flat hyperk\"ahler metric $h^{\FF}$ on $\cC_{\beta,\tau}$ with a K\"ahler form and a holomorphic $2$-form
  \begin{align}
   \omega^{\FF}&=\frac{\sq}{2}(d\mathscr{U}\wedge d\mathscr{\bar U}+d \mathscr{V}\wedge d\mathscr{\bar V}),
  \\
  \Omega^{\FF}&=d\mathscr{U}\wedge d \mathscr{V}.
  \end{align}
Each flat space $(\cC_{\beta,\tau}, h^{\FF})$ given as the above  is called a {\it standard ALG model}.
\end{definition}

\begin{table}[h]
\caption{Invariants of ALG spaces}
\label{ALGtable}
 \renewcommand\arraystretch{1.5}
\begin{tabular}{|c|c|c|c|c|c|c|c|c|} \hline
  0  & $\I_0^*$ & $\II^*$ & $\II$ & $\III^*$ & $\III$ & $\IV^*$ & $\IV$\\\hline

  $\infty$ & $\I_0^*$ & $\II$ & $\II^*$ & $\III$ & $\III^*$ & $\IV$ & $\IV^*$\\\hline

 $\beta\in(0,1]$ &  $\frac{1}{2}$  & $\frac{1}{6}$ & $\frac{5}{6}$ & $\frac{1}{4}$ & $\frac{3}{4}$ & $\frac{1}{3}$ & $\frac{2}{3}$\\[5pt]\hline
  $\tau\in\mathbb{H}$ & Any & $e^{\sq \cdot \frac{2\pi}{3}}$ & $e^{\sq \cdot \frac{2\pi}{3}}$ & $\sq$ & $\sq$ & $e^{\sq \cdot \frac{2\pi}{3}}$ & $e^{\sq \cdot \frac{2\pi}{3}}$ \\\hline
Intersection & \multirow{2}{*}{$\widetilde\D_4$} & \multirow{2}{*}{$\widetilde\E_8$} & \multirow{2}{*}{$\widetilde\A_0$}& \multirow{2}{*}{$\widetilde\E_7$}  & \multirow{2}{*}{$\widetilde\A_1$} & \multirow{2}{*}{$\widetilde\E_6$} &    \multirow{2}{*}{$\widetilde\A_2$} \\
matrix  &&&&&&& \\
  \hline
    \end{tabular}
\end{table}

  Notice that there is a holomorphic map $\mathscr{Y}:\cC_{\beta,\tau}\to\mathbb{C}$ defined as $\mathscr{Y} = \mathscr{U}^{\frac{1}{\beta}}$, which  provides a singular elliptic surface structure. We can resolve the singularity in the central fiber as Kodaira has done in \cite{Kodaira1963}, to obtain a smooth elliptic surface $\cG$ called an {\it isotrival ALG manifold}. The types  of the central fibers are listed in the row 0 of the table. The pull back of $d\mathscr{U}\wedge d\mathscr{V}$ provides a holomorphic 2-form $\Omega^{\cG} \equiv\omega_2^{\cG}+\sq\omega_3^{\cG}$ on the ALG manifold $\cG$ outside the central fiber. It can be extended to a nowhere vanishing holomorphic 2-form on all of $\cG$, as we will see in Lemma~\ref{l:hol2formg}

  An isotrivial ALG manifold $\cG$ has a complex analytic compactification. For example, the isotrivial ALG manifold with central fiber of Type $\II$ is the complement of a rational elliptic surface minus a fiber of Type $\II^*$. The type of the fiber at infinity is listed in the row $\infty$ of the Table \ref{ALGtable}

  Any isotrivial ALG manifold deformation deformation retracts to the central fiber. By \cite{Kodaira1963}, the intersection form on $H^2(\cG)$ forms an extended Dynkin diagram in the last row of Table~\ref{ALGtable}. In particular, the rank of $H^2(\cG)$ is the subscript of the extended Dynkin diagram plus 1. In other words, the rank of $H^2(\cG)$ is $5, 9, 1, 8, 2, 7, 3$ respectively.

  \begin{definition}\label{d:ALG-space} A complete $4$-manifold
  $(\cG,g^{\cG})$ with a hyperk\"ahler triple  $(\omega^{\cG}, \omega_2^{\cG}, \omega_3^{\cG})$,
is called ALG hyperk\"ahler of order $\aleph$ if there exist $R>0$, a compact subset $K \subset \cG$, and a diffeomorphism $\Phi: \{(\mathscr{U},\mathscr{V})\in\cC_{\beta,\tau}\ \big|\ |\mathscr{U}|> R\}\rightarrow \cG \setminus K$ such that
  \begin{align}
 \big|\nabla_{h^{\FF}}^k(\Phi^*g^{\cG}-h^{\FF})\big|_{h^{\FF}}=O(|\mathscr{U}|^{-k-\aleph}), \\
 \big|\nabla_{h^{\FF}}^k(\Phi^*\omega ^{\cG}- \omega^{\FF})\big|_{h^{\FF}}=O(|\mathscr{U}|^{-k-\aleph}), \\
\big|\nabla_{h^{\FF}}^k(\Phi^*\Omega ^{\cG}- \Omega^{\FF})\big|_{h^{\FF}}=O(|\mathscr{U}|^{-k-\aleph}),
  \end{align}
as $|\mathscr{U}| \to \infty$, for any $k\in\dN$, where $\Omega^{\cG} \equiv \omega_2^{\cG}+\sq\omega_3^{\cG}$,
and $h^{\FF}, \omega^{\FF}, \Omega^{\FF}$ are as in Definition~\ref{d:ALG-model}  on $\cC_{\beta,\tau}$ for some $(\beta,\tau)$ in Table~\ref{ALGtable}.
  \label{ALG-definition}
  \end{definition}

Given any ALG flat model space $\cC_{\beta,\tau}$, there are many known examples of complete non-compact hyperk\"ahler ALG spaces $(\cG, g^{\cG})$ with asymptotic geometry given by $\cC_{\beta,\tau}$ in the sense of Definition \ref{ALG-definition}.
For the developments in this direction including the construction techniques, the analysis on ALG spaces and related classification results,
we refer the readers to \cite{TianYau, CherkisKapustinALG, BB, BM, Hein, CCIII} and also the references therein.

In the known constructions, when the fiber at infinity has Type $\II^*$, $\III^*$ or $\IV^*$, the optimal asymptotic rate of convergence is given by
  \begin{equation}
  \aleph=2-\frac{1}{\beta},
\label{e:smaller-rate}
  \end{equation}
see \cite[Theorem~4.6]{CCII}.
 However, in our gluing constructions in later sections, we will see that, the slow convergence rate \eqref{e:smaller-rate} gives a {\it large error term} which is not enough for applying the implicit function theorem.
Therefore we will restrict to the following class of ALG hyperk\"ahler metrics:
\begin{enumerate}
\item  $(\cG, g)$ is isotrivial,
\item  $(\cG, g)$ is ALG hyperk\"ahler of order at least $2$.
\end{enumerate}
%Therefore, will restrict ourselves to ALG metrics of order at least 2 on the isotrivial ALG manifolds.
Note that there do exist examples in all cases which are ALG of order at least $2$ \cite[Remark 1.7 (ii)]{Hein}.

In the gluing construction near a singular fiber with finite monodromy, we also need
an approximation estimate between a hyperk\"ahler ALG metric and its asymptotic model metric.
In our case,
the isotrivial ALG space $(\cG, g)$ has a central fiber of Type $\IV$ and singular fiber of Type $\IV^*$ at infinity. Therefore, the asymptotic cone $T_{\infty}\cG$ is isometric to the $2$-dimensional flat cone $C(S_{4\pi/3}^1)$.
Moreover, the standard ALG model of $\cG$ is the singular flat space $\cC_{\beta,\tau}$ for $\beta=\frac{2}{3}$ and $\tau=e^{\sq\cdot \frac{2\pi}{3}}$. We also identify the model space $\cC_{\beta,\tau}$
with the topological product space $(0,+\infty)\times \Sigma_{\beta,\tau}^3$, where $\Sigma_{\beta,\tau}^3$
is a flat manifold which is diffeomorphic to $\dT^3/\dZ_3$. So the definition of the ALG space gives a diffeomorphism
\begin{equation}
\Phi: (R,+\infty)\times \Sigma_{\beta,\tau}^3\to \cG\setminus K \label{e:ALG-diffeomorphism-for-gluing}
\end{equation}
for some $R>0$ and compact subset $K\subset \cG$.

Now we are in a position to calculate the approximation order between the $2$-forms $\Phi^*\omega^{\cG}$ and $\omega^{\FF}$
on the model space $\cC_{\beta,\tau}$. To achieve this, we need a general standard lemma
for estimating the approximation between closed $2$-forms.

\begin{lemma}
\label{l:asymp-closed-2-form-on-product}
Consider a topological product space $X=(r_0,+\infty)\times Y$
with a coordinate system $\bx=(r,\by)$ and a Riemannian metric $g$.
Let $\omega_0$ and $\omega$ be closed $2$-forms on $X$ satisfying the property that for some $\aleph>1$,
\begin{equation}
|\nabla_g^k(\omega - \omega_0)|_g \leq C_k \cdot r^{-\aleph-k} \ \text{as} \ r\to+\infty,\label{e:2-forms-close}
\end{equation}
holds for all $k\in\dN$. Then there exists a $1$-form $\eta$ such that
$\omega - \omega_0  = d\eta$ and for any $k\in\dN$,
\begin{align}
|\nabla_g^k\eta|_g & \leq C_k \cdot r^{1-\aleph-k}\ \text{as}\ r\to+\infty.
\end{align}

\end{lemma}

\begin{proof}
In terms of the coordinates $(r,\by)$, we can write
\begin{equation}
\omega - \omega_0 = \varphi +  dr \wedge \psi,
\end{equation}
 where $\varphi$ is a $2$-form and $\psi$ is a $1$-form satisfying $\partial_r \lrcorner \varphi = \partial_r \lrcorner \psi = 0$.
Notice that condition $d\omega=d\omega_0=0$ implies
\begin{equation}
d_{Y}\varphi = 0, \ d_r\varphi = d_{Y}\psi.\label{e:closedness}
\end{equation}
Moreover, by \eqref{e:2-forms-close}, for some $\aleph>1$ we have
\begin{equation}|\nabla_g^k\psi|_g \leq C_k\cdot r^{-\aleph-k}\ \text{as}\ r\to+\infty\end{equation} for any $k\in\dN$.
Now we define a $1$-form
\begin{equation}
\eta \equiv -\int_r^{\infty}\psi dr.\label{e:integral-def}
\end{equation}

The asymptotic behavior of $\psi$ gives the integrability of \eqref{e:integral-def} and hence $\eta$ is well-defined.
To verify $d\eta= \omega - \omega_0$, it suffices to check \begin{equation}d\eta  = \varphi  + dr\wedge \psi,\end{equation}
which follows directly from the definition of $\eta$ and \eqref{e:closedness}.
Therefore, $\eta$ is the desired $1$-form satisfying the asymptotic behavior \begin{equation}
|\nabla_g^k \eta|_g \leq C_k\cdot r^{1-\aleph+k}\ \text{as}\ r\to+\infty,
\end{equation}
for any $k\in\dN$.

\end{proof}

Since we will glue rescaled ALG spaces with semi-flat metrics on $\cK$
which have bounded diameter, we need to work with the ALG space and prove the approximation estimate at small scales.
Based on Lemma \ref{l:asymp-closed-2-form-on-product}, we have the following.

\begin{proposition}\label{p:exact-error-ALG}
Let $(\cG, g^{\cG}, p)$ be a complete hyperk\"ahler isotrivial ALG space of order $2$ with a corresponding K\"ahler form $\omega^{\cG}$
such that for some $\beta\in(0,1]$ and $\tau\in\mathbb{H}$ there is an ALG coordinate system\begin{equation}
\Phi: \{(\mathscr{U},\mathscr{V})\in\cC_{\beta,\tau}\ \big|\ |\mathscr{U}|> R\} \to \cG\setminus K,
\end{equation}
outside a compact set $K\subset \cG$ equipped with a model ALG K\"ahler form $\omega^{\FF}$ on $\cC_{\beta,\tau}$.
Let $0<\delta\ll1$, we rescale the metrics and K\"ahler forms by
\begin{align}g_\delta^{\cG} &= \delta^2 g^{\cG},\ \omega_\delta^{\cG} = \delta^2 \omega^{\cG}\\
h_\delta^{\FF} &= \delta^2 h^{\FF},      \ \omega_\delta^{\FF} = \delta^2 \omega^{\FF},
\end{align}
then there exists a  $1$-form $\eta_\delta$  satisfying
\begin{equation}
\Phi^*\omega_\delta^{\cG} -\omega_\delta^{\FF} =d\eta_\delta
\end{equation}
 with respect to the rescaled metric $h_\delta^{\FF}$
\begin{equation}
\big|\nabla_{h_\delta^{\FF}}^k \eta_\delta\big|_{h_\delta^{\FF}} (\mathscr{U},\mathscr{V}) \leq C_k\cdot \delta^2  \Big(\frac{1}{\delta \cdot |\mathscr{U}|}\Big)^{1+k},\ \delta\cdot|\mathscr{U}|\geq \delta\cdot R
\end{equation}
for any $k\in \mathbb{N}$, where $C_k$ is  independent of $\delta$.

\end{proposition}

\begin{proof}
Under the map $\Phi$, the complement $\cG\setminus K$ is diffeomorphic to a topological product
$(R,+\infty)\times \Sigma_{\beta,\tau}^3$, where $\Sigma_{\beta,\tau}^3$ is a compact flat $3$-manifold. So the coordinate $r$ in Lemma~\ref{l:asymp-closed-2-form-on-product}
can be chosen as the distance function to the origin. Now the asymptotics immediately
follows from Lemma~\ref{l:asymp-closed-2-form-on-product} and standard rescaling computations.

\end{proof}

\subsection{Construction of approximate solutions}
\label{ss:approximate-metric-ALG}
With the above technical preparations, we are ready to construct the approximate metrics around singular fiber with finite monodromy.

We first choose a diffeomorphism which identifies a large ball in the
ALG metric with a tubular neighborhood of the singular fiber in $\cK$.
Define a local diffeomorphism $\underline{\Psi}$ from $\cC_{\beta, \tau}$ to $\cF/\ZZ_3$ by
\begin{equation}\underline{\Psi}: \Big(\mathscr{U}, \mathscr{V} = \frac{1}{\sqrt{\Ima \tau}}v_1 + \frac{\tau}{\sqrt{\Ima \tau}} e^{\sq\cdot \frac{2\pi}{3}} v_2 \Big) \mapsto  (u,v) = (\delta \mathscr{U}, \tau_1(\delta \mathscr{U})v_1 + \tau_2(\delta \mathscr{U}) v_2)\label{e:blow-down-diff}
\end{equation} for $v_1,v_2\in\dR/\dZ$. Since $\underline{\Psi}$ maps the fixed points to the fixed points, it induces a local diffeomorphism $\tilde\Psi$ from the resolution $\widetilde{\cC_{\beta, \tau}}$ to $\widetilde{\cF/\ZZ_3}$. It maps $\Theta_{\widetilde{\cC_{\beta, \tau}}}$ to $\Theta_{\widetilde{\cF/\ZZ_3}}$, so it induces a local diffeomorphism $\Psi$ from the isotrivial ALG manifold $\cG$ to $\cK$ which satisfies
\begin{align}
\Psi^* ( \omega^{\FF}_{\delta, \cK}) = \omega^{\FF}_\delta, \ \Psi^* (g^{\FF}_{\delta, \cK} ) = h^{\FF}_\delta,
\end{align}
away from the singular fiber.

The initial step of constructing gluing metrics is to fix the size of the gluing region.
From now on,
we fix a parameter $\ell\in(0,1)$ which will be determined later.
Let $\fF:\cK\to \PP^1$ be the fixed elliptic K3 surface with a singular fiber $\fF^{-1}(p)$  of Type $\IV$.
Let $\delta\in(0,1)$ be sufficiently small, then we work with a small neighborhood of the above singular fiber
\begin{equation}
\cO_{2\delta^{\ell}}(p) \equiv \fF^{-1}(B_{2\delta^{\ell}}(p)),\quad B_{2\delta^{\ell}}(p) \subset \PP^1,
\end{equation}
 such that for some $\underline{c}_0,\bar{c}_0>0$, \begin{equation}
\underline{c}_0\cdot \delta^{\ell} \leq \diam_{g_\delta^B} \Big(\cO_{2\delta^{\ell}}(p) \Big) \leq \bar{c}_0\cdot \delta^{\ell}.
 \end{equation}
Therefore, for fixed $0<\delta\ll1$ and $\ell\in(0,1)$, we choose a large cutoff region in $\cG$ as follows
\begin{equation}
\cG(\delta^{\ell-1}) \equiv \cG \setminus \Phi\Big((\delta^{\ell-1},\infty) \times \Sigma_{\beta,\tau}^3\Big)
 \end{equation}
so that
\begin{align}
& \Psi: \cG (\delta^{\ell-1})  \longrightarrow \cO_{2\delta^{\ell}}(p) \subset \cK,
\\
& \underline{c}_0\cdot \delta^{\ell-1} \leq \diam_{g^{\cG}}(\cG(\delta^{\ell-1}))\leq \bar{c}_0\cdot \delta^{\ell-1}.\end{align}

To proceed with the gluing construction, we also need to make the scales consistent.
For this purpose, we rescale the complete ALG metric
by choosing
$g_\delta^{\cG} = \delta^2 g^{\cG}$ so that the $\dT^2$-fiber at the infinity of $\cG$ has diameter proportional to $\delta$.
Also let us denote by $\omega_\delta^{\cG}$ the  K\"ahler
form with respect to the rescaled ALG metric $g_\delta^{\cG}$.
Then the rescaled region $(\widetilde{\cG}(\delta^{\ell}),g_\delta^{\cG}))\equiv (\delta \cdot \cG(\delta^{\ell-1}), \delta^2\cdot g_\delta^{\cG})$
satisfies \begin{equation}
\underline{c}_0\cdot \delta^{\ell} \leq \diam_{g_\delta^{\cG}}(\widetilde{\cG}(\delta^{\ell})) \leq \bar{c}_0 \cdot \delta^{\ell}.
\end{equation} Therefore, under the diffeomorphism
\begin{align}
\Psi: \widetilde{\cG}(\delta^{\ell}) \longrightarrow \cO_{2\delta^{\ell}}(p)
\end{align}
we obtain a metric
$\Psi_*g_\delta^{\cG}$ and a K\"ahler form $\Psi_*\omega_\delta^{\cG}$ on the open set
$\cO_{2\delta^{\ell}}(p)\subset\cK$.

By Lemma \ref{l:annulus} and Proposition \ref{p:exact-error-ALG}, there are $1$-forms $\eta_\delta^B$
 and $\eta_\delta^{\cG} \equiv \Phi_* \eta_\delta$
in the annulus region $\mathcal{A}_{\delta^{\ell},2\delta^{\ell}}(p)\equiv \fF^{-1}(A_{\delta^{\ell},2\delta^{\ell}}(p))\subset \cK$ such that the following holds in $\mathcal{A}_{\delta^{\ell},2\delta^{\ell}}(p)$:\begin{align}
\omega_\delta^B = \omega_{\delta, \cK}^{\FF} + d\eta_\delta^B,
\quad
\omega_\delta^{\cG} = \Phi_* \omega_\delta^{\FF} + d\eta_\delta^{\cG},
\end{align}
and for any $k\in\dN$,
\begin{align}
\Big|\nabla_{g_{\delta,\cK}^{\FF}}^k(\Psi_*\eta_\delta^{\cG})\Big|_{g_{\delta,\cK}^{\FF}} \leq C_k \cdot \delta^{2-\ell(k+1)},\quad \Big|\nabla_{g_{\delta, \cK}^{\FF}}^k(\eta_\delta^B)\Big|_{g_{\delta,\cK}^{\FF}} \leq C_k \cdot \delta^{\ell(\lambda_\beta+1-k)}.\end{align}

Then $2$-form $\omega_\delta^C$ is then defined as follows.
\begin{proposition}
Let $\fF:\cK\to \PP^1$
be an elliptic K3 surface with a fixed holomorphic volume $2$-form $\Omega$.  For $0<\delta\ll1$, there are a family of $2$-forms
\begin{align}
\omega_\delta^{C} \equiv
\begin{cases}
\Psi_* \omega_\delta^{\cG} , & \text{on}\ \cO_{\delta^{\ell}}(p),\\
\omega_{\delta, \cK}^{\FF} + d\Big(\chi\cdot \Psi_*\eta_\delta^{\cG} + (1-\chi)\cdot \eta_\delta^{B}\Big), & \text{on}\ \mathcal{A}_{\delta^{\ell},2\delta^{\ell}}(p),
\\
\omega_\delta^{B}, & \text{on}\  \cK\setminus\cO_{2\delta^{\ell}}(p),
\end{cases}\label{e:gluing-metric-ALG}
\end{align}
satisfying the error estimate in transition region $\mathcal{A}_{\delta^{\ell},2\delta^{\ell}}(p)$
\begin{align}
\sup\limits_{\mathcal{A}_{\delta^{\ell},2\delta^{\ell}}(p)}\Big|\nabla_{g_{\delta, \cK}^{\FF}}^k(\omega_\delta^C - \omega_{\delta, \cK}^{\FF} )\Big|_{g_{\delta, \cK}^{\FF}} \leq C_k \cdot ( \delta^{2-\ell(k+2)}+ \delta^{\ell(\lambda_\beta-k)}),
\end{align}
for $k\in\mathbb{N}$, for $C_k$ independent of $\delta$, and
where the smooth cutoff function is defined by
\begin{equation}
\chi = \begin{cases}
1, & \text{on}\ \cO_{\delta^{\ell}}(p),
\\
0, & \text{on}\ \cK\setminus\cO_{2\delta^{\ell}}(p).
\end{cases}
\end{equation}

\end{proposition}

\begin{remark}By Remark~\ref{r:triple}, since we have a holomorphic $(2,0)$-form $\Omega_\delta = \delta \Omega_{\cK}$, there is a Riemannian metric $g_\delta^C$ associated to $\omega_\delta^C$.
\end{remark}

For a rescaled ALG space $(\cG,g_\delta^{\cG})$ with an infinity model $\cC_{\beta,\tau}$,
the pushforward $\Psi_* g_\delta^{\cG}$ on $\cO_{\delta^{\ell}}(p)$ does not give a hyperk\"ahler metric with respect to the complex structure of $\cK$. Indeed, the error term is of polynomial rate given by the above distortion order $\lambda_{\beta}$.
We next give some quantitative estimates of this error term.
To begin, note we can also do the above resolution procedure as in Subsection~\ref{ss:semi-flat-estimates} on the isotrivial ALG manifold $\cG$, which we again illustrate in the case of type $\IV$.  Then $\cF/\ZZ_3$ is replaced by $\cC_{\beta, \tau}$ for $\beta=\frac{2}{3}$ and $\tau=e^{\i \frac{2\pi}{3}}$ and the blow up of $\cG$ is $\widetilde{\cC_{\beta, \tau}}$.
\begin{lemma}
\label{l:hol2formg}
The pull back of $\Omega^{\cG}$ to $\widetilde{\cC_{\beta, \tau}}$ is the same as the pull back of $d \mathscr{U}\wedge d \mathscr{V}$.
\end{lemma}
\begin{proof}
 In this case, $\varrho_{\cG}=\tau$, so all the terms involving $h$ disappear, and
\begin{equation}
{\Omega}^{\FF} =2\hat{\varsigma}d\hat{\varsigma}\wedge d\hat{\zeta}=d\mathscr{U}\wedge d\mathscr{V},
\end{equation}
where $\hat{\varsigma}\equiv \mathscr{Y}^{1/3}=\mathscr{U}^{1/2}$ and $\hat{\zeta}=\mathscr{V}$.
Then we see that the pull back $\tilde{\Omega}^{\cG}$ of $\Omega^{\cG}$ satisfies
\begin{align}
\tilde{\Omega}^{\cG} = k_{\cG}(\mathscr{Y})  \tilde{\Omega}^{\FF},
\end{align}
where $\tilde{\Omega}^{\FF}$ is the pullback of $\Omega^{\FF}$.
We know that $k_{\cG}\to 1$ as $\mathscr{Y}\to\infty$, so by the maximum principle, $k_{\cG}=1$.
\end{proof}

We end this section with the following proposition which give precise error estimates in the ALG regions.
\begin{proposition}
\label{p:complexerror}
Let $R$ be the fixed constant in Definition \ref{d:ALG-space} only depending on $g^{\cG}$, then for all $k\in\mathbb{N}$, there exists a constant $C_k$, independent of $\delta$,   for $\delta$ sufficiently small such that
on  $\cO_{\delta R}(p)$, we have
\begin{align}
|\nabla_{\Psi_* g_\delta^{\cG}}^k(\delta\Omega_{\cK}-\delta^2 \Psi_* \Omega^{\cG})|_{\Psi_* g_\delta^{\cG}}
&\leq C_k \delta^{\frac{2}{5}-k}\\
|\nabla_{g_\delta^C}^k (g_\delta^C-\Psi_* g_\delta^{\cG})|_{g_\delta^C}&\le C_k \delta^{\frac{2}{5}-k},
\end{align}
and on  $\mathcal{A}_{\delta R, 2 \delta^{\ell}}(p)
= \cO_{2 \delta^{\ell}}(p) \setminus  \cO_{\delta R}(p)$, we have
\begin{align}
|\nabla_{\Psi_* g_\delta^{\cG}}^k(\delta\Omega_{\cK}-\delta^2 \Psi_* \Omega^{\cG})|_{\Psi_* g_\delta^{\cG}}
&\leq C_k |u|^{\frac{2}{5}-k},\\
|\nabla_{g_\delta^C}^k (g_\delta^C-\Psi_* g_\delta^{\cG})|_{g_\delta^C}&\le C_k |u|^{\frac{2}{5}-k}.
\end{align}
\end{proposition}

\begin{proof}
We again use the type $\IV$ case as an example. First, we use the metric $\Psi_*g _\delta^{\cG}$ in the deep region $\cO_{\delta^{\ell}}(p)$ given by the pushforward metric from the rescaled ALG space.

We first consider the region $|\mathscr{U}|\le R$. In this case, we can pull back everything to $\widetilde{C_{\beta, \tau}}$. Recall that $\widetilde{C_{\beta, \tau}}$ is the blow up of $\cG$ and $\{\hat\varsigma=0\}$ is the exceptional divisor. A straightforward calculation shows that there exists a constant $C > 0$ such that
the pull back $\tilde{g}_{\delta}^{\cG}$ of $g_\delta^{\cG}=\delta^2 g^{\cG}$ to $\widetilde{C_{\beta, \tau}}$ satisfies
\begin{align}
\tilde{g}_{\delta}^{\cG} - C^{-1} \delta^2\big(|\hat\varsigma|^2 (d v_1^2 + d v_2^2)+ (d \Rea \hat\varsigma)^2 + (d \Ima \hat\varsigma)^2 \big)
\end{align}
is positive definite. By \eqref{e:SF2}, we know that
\begin{equation}
\delta\Omega_{\cK}=\delta d y\wedge (\tau_1 d v_1+ \tau_2 d v_2)=2\delta\varsigma d \varsigma \wedge (\tau_1 d v_1+ \tau_2 d v_2).
\end{equation}
From Lemma~\ref{l:hol2formg}, the pull back of $\delta^2 \Omega^{\cG}$ to $\widetilde{C_{\beta, \tau}}$ is
\begin{equation}
2\delta^2 \hat\varsigma d \hat\varsigma \wedge \Big(\frac{1}{\sqrt{\Ima\tau}}d v_1 + \frac{\tau}{\sqrt{\Ima\tau}} d v_2 \Big).
\end{equation}
However, recall that $\varsigma^2=u=\delta \mathscr{U}=\delta\hat\varsigma^2$, $\tau_1=\frac{1}{\sqrt{\Ima\tau}}+O(|u|)$ and $\tau_2=\frac{\tau}{\sqrt{\Ima\tau}}+O(|u|)$. So we see that
\begin{align}
|  \delta\Omega_{\cK}-\delta^2 \Psi_*\Omega^{\cG}    |_{\Psi_* g_\delta^{\cG}}=O(\delta)
\end{align}
 inside $|\mathscr{U}|<R$ in the $\IV$ case. Here the bounds may depend on $R$. In general, we will have
\begin{align}
|  \delta\Omega_{\cK}-\delta^2\Psi_*\Omega^{\cG}    |_{\Psi_* g_\delta^{\cG}}=O(\delta^{\lambda_\beta})
\end{align}
where $\lambda_{\beta}$ is the distortion order. Recall that from Table \ref{Distortiontable}, the minimal order $\lambda_{\beta}$ is equal to $\frac{2}{5}$.

Next, we look at the region $R \le |\mathscr{U}|\le \delta^{\ell-1}$. In this region, the ALG metric $g^{\cG}$ is equivalent to the flat metric $h^{\FF}$.
A similar calculation as above shows that
\begin{align}
| \delta\Omega_{\cK}-\delta^2 \Psi_*\Omega^{\cG}|_{\Psi_* g_\delta^{\cG}}=O(|u|^{\lambda_\beta}).
\end{align}
By a similar argument, we can obtain the bounds on higher order derivatives, the detailed calculations are omitted.

To compare the metric $g_{\delta}^C$ with $\Psi_* g_\delta^{\cG}$, recall that the metric $g_{\delta}^C$ is defined as the metric associated to the triple $\left(\Psi_* \omega_\delta^{\cG}, \Rea (\delta \Omega_{\cK}), \Ima (\delta \Omega_{\cK})\right)$ in the region $\cO_{\delta^{\ell}}(p)$ and the metric $\Psi_* g_\delta^{\cG}$ is associated to the triple $\left(\Psi_* \omega_\delta^{\cG}, \Rea (\delta^2 \Psi_* \Omega^{\cG}), \Ima (\delta^2 \Psi_* \Omega^{\cG})\right)$.  The metric error estimates then follow from the previous estimates.

Finally, the required estimates in the damage zone region $\mathcal{A}_{\delta^{\ell}, 2 \delta^{\ell}}(p)
= \cO_{2 \delta^{\ell}}(p) \setminus  \cO_{ \delta^{\ell}}(p) $ follow from Lemma~\ref{l:annulus} and Proposition~\ref{p:exact-error-ALG}.

\end{proof}

\section{Metric geometry and regularity of the approximate solutions}

\label{s:metric-geometry}

In this section, we will analyze
the singularity behavior for the collapsing Ricci almost-flat metrics constructed in Section \ref{s:infinite-monodromy} and Section \ref{s:finite-monodromy} in a quantitative way. Specifically, in each of the above cases, we will give uniform estimates for the {\it regularity scales} in terms of the collapsing parameters. Based on these effective estimates, we will set up the package of the weighted analysis which will be used for the perturbative analysis in Section \ref{s:proof-of-main-theorem}.

\subsection{Singularity behavior and decomposition of the approximate metric}

To begin with, we introduce some basic notions
for discussing the singularity behavior in a quantitative way.
The following concept of {\it regularity scale}
is commonly used and very convenient to study the singularity behavior for a sequence of metrics in both non-collapsing and collapsing settings.

\begin{definition}[Local regularity]
\label{d:local-regularity} Let $(M^n,g)$ be a Riemannian manifold. Given $r, \epsilon>0$, $k\in\dN$, $\alpha\in(0,1)$, we say $(M^n,g)$ is $(r,k+\alpha,\epsilon)$-regular at $x\in M^n$ if the Riemannian metric $g$ is at least $C^{k, \alpha}$ in $B_{2r}(x)$ such that the following holds: let $(\widehat{B_{2r}(x)},\hat{x})$ be the Riemannian universal cover of $B_{2r}(x)$, then $B_r(\hat{x})$ is diffeomorphic to a Euclidean disc $\dD^n\subset \dR^n$ such that the lifting metric $\hat{g}$ in coordinates  satisfies for each $1\leq i,j\leq n$,
\begin{equation}
|\hat{g}_{ij}-\delta_{ij}|_{C^0(B_r(\hat{x}))}+\sum\limits_{|m|\leq k} r^{|m|}\cdot|\p^m \hat{g}_{ij}|_{C^0(B_r(\hat{x}))} +  r^{k+\alpha}[\hat{g}_{ij}]_{C^{k,\alpha}(B_r(\hat{x}))} < \epsilon,
\end{equation}
where $m$ is a multi-index, and the last term is the H\"older semi-norm.
\end{definition}

\begin{definition}
[$C^{k,\alpha}$-regularity scale] Let $(M^n,g)$ be a Riemannian manifold with a smooth Riemannian metric $g$:
\begin{enumerate}
\item  The $C^{k,\alpha}$-regularity scale at $x\in M^n$, denoted by $r_{k,\alpha}(x)$, is defined as
the supremum of all $r>0$ such that $M^n$  is $(r,k+\alpha,10^{-9})$-regular at $x$.
\item  We can also define the $\epsilon$ curvature scale $r_{|\Rm|}(x)$ at $x$ as the supremum of all $r>0$ such that $r^2\cdot |\Rm|$ of $g$ in $B_r(x)$ is bounded by $10^{-9}$.
\end{enumerate}
\end{definition}
\begin{remark}
\label{r:1-Lipchitz}
It is well known that there exists a constant $C>0$ only depending on $n, k, \alpha$ such that if $g$ is Ricci-flat, then by taking local universal covers and using standard Schauder estimate, it follows that
\begin{align}
\label{e:curvature-scale}
C^{-1}\cdot  r_{|\Rm|}(x) \leq  r_{k,\alpha}(x) \leq C \cdot r_{|\Rm|}(x),
\end{align}
(see \cite{Petersen} or \cite{PWY} for more discussions about this). Also note that
the regularity scale $r_{k,\alpha}$ is $1$-Lipschitz on a Riemannian manifold $(M^n,g)$, i.e.,
\begin{align}
\label{1lipreg}
|r_{k,\alpha}(x)-r_{k,\alpha}(y)| \leq d_g(x,y),\quad \forall x,y\in M^n,
\end{align}
see \cite[Section 1]{CheegerTian}.
\end{remark}

We will also need the following notion of canonical bubble.

\begin{definition}
[Canonical bubble limit]\label{d:canonical-bubble}
Let  $(M_j^n, g_j,x_j)$ be a sequence of Riemannian manifolds, and let $r_j\equiv r_{k,\alpha}(x_j)$ be the regularity scale at $x_j$. Then
$(X_{\infty}, d_{\infty}, x_{\infty})$ is called a canonical bubble limit at $x_j$ if passing to a subsequence, the following pointed Gromov-Hausdorff convergence holds,
\begin{equation}
(M_j^n, r_j^{-2} g_j, x_j) \xrightarrow{GH} (X_{\infty}, d_{\infty}, x_{\infty}).\label{e:GH-convergence-to-bubble}
\end{equation}

\end{definition}

\begin{remark}[Noncollapsing bubble limits] Let $g_j$ be Einstein and let the rescaled sequence be non-collapsing, i.e. $\Vol_{\tilde{g}_j}(B_1(x_j))\geq v_0>0$
with $\tilde{g}_j \equiv r_j^{-2}g_j$. If  $X_{\infty}$ is smooth,  then the $\epsilon$-regularity for non-collapsing Einstein manifolds tells us that the Gromov-Hausdorff convergence \eqref{e:GH-convergence-to-bubble} can be improved to $C^k$-convergence for any $k\in\mathbb{N}$. If $g_j$
satisfy the uniform Ricci curvature bound $|\Ric_{g_j}|\leq n-1$, then in the non-collapsing setting, \eqref{e:GH-behavior} can be improved to $C^{1,\alpha}$-convergence.
 For more details, see \cite[Theorems 7.2 \& 7.3]{ChC1}.
\end{remark}

Let us return to our context and take
 a fixed elliptic K3 surface $\fF:\cK\to \PP^1$ with a finite singular set $\mathcal{S}\subset \PP^1$.
In the previous sections, we have constructed a family of approximately hyperk\"ahler
metrics $g_{\delta}^C$.
Our main goal in this section is to understand the quantitative singularity behavior of $g_\delta^C$ by obtaining effective estimates of the {\it regularity scale} for each point $\bx\in\cK$. This is a necessary technical part for implementing the weighted analysis in Section \ref{s:proof-of-main-theorem}.

Next, we outline how to prove the effective regularity scale estimates.
The fundamental strategy is to analyze the bubble limits by appropriately rescaling the collapsing sequence.
We first assume that $\Ric_{g_\delta^C}= 0$. This is of course not true. For example, in the deep region $\cO_{\delta^{\ell}}(p)$ in Section \ref{s:finite-monodromy}, $g_{\delta}^C$ is not the same as the Ricci-flat metric $\Psi_* g_\delta^{\cG}$. However, in this case, we can study the regularity scale of $\Psi_* g_\delta^{\cG}$ and use Proposition \ref{p:complexerror} to study the regularity scale of $g_\delta^C$. Similar method applies to the damaged zone in all cases.

To start with, let us consider the case of codimension 1 collapse: assume that for the rescaled space $(\cK,\tilde{g},\bx)$ with a reference point $\bx\in\cK$ and a rescaled metric $\tilde{g}\equiv \lambda^2\cdot g_\delta^C$ which satisfies $\Ric_{\tilde{g}}= 0$ on $B_2^{\tilde{g}}(\bx)$, the sequence satisfies the following Gromov-Hausdorff behavior for a {\it sufficiently small number} $\epsilon>0$,
\begin{align}
d_{GH}(B_2^{\tilde{g}}(\bx), B_2(\bx_{\infty})) < \epsilon , \quad B_2(\bx_{\infty})\subset \dR^3.\label{e:GH-behavior}
\end{align}
In this case, it is known that $\Gamma_{\epsilon}(\bx)$ has a finite-index cyclic subgroup  either $\dZ$ or $\dZ_p$ $(p\in\dZ_+)$  so that $\rank(\Gamma_{\epsilon}(\bx))\leq 1$, where
\begin{align}
\Gamma_{\epsilon}(\bx)\equiv\Image[\pi_1(B_{\epsilon}^{\tilde{g}}(\bx))\to\pi_1(B_{1/4}^{\tilde{g}}(\bx))],
\end{align}
see \cite[Theorem~6.1]{KW}. In the general setting, we define the nilpotent rank for a finitely generated nilpotent group.

\begin{definition}
For a finitely generated nilpotent group $\mathcal{N}$, its nilpotent rank is defined as the sum of the abelian ranks $\rank(A_j)$  arising from the lower central series
\begin{equation}
\mathcal{N}\equiv \mathcal{N}_0 \rhd \mathcal{N}_1 \rhd \mathcal{N}_2 \rhd  \mathcal{N}_0 =\{e\},\quad A_j\equiv \mathcal{N}_{j-1}/\mathcal{N}_j.
\end{equation}
For any finitely generated group $\Gamma$, all finite-index subgroups $\mathcal{N}$ share the same nilpotent rank. In this case, we just define $\rank(\Gamma)\equiv \rank(\mathcal{N})$ for any nilpotent subgroup $\mathcal{N}$ satisfying $[\Gamma:\mathcal{N}]<\infty$.
\end{definition}

The regularity at $\bx$ can be related to the topology at $\bx$ in the following way (see  \cite[Theorem~1.1]{NaberZhang} or \cite[Lemma~7.7]{HSVZ}).
\begin{lemma}\label{l:maximal-rank} There is some $10^{-3}>\epsilon_0>0$ such that if\eqref{e:GH-behavior} holds for $\epsilon\leq \epsilon_0$  and
the group $\Gamma_{\epsilon}(\bx)$ satisfies $\rank(\Gamma_{\epsilon}(\bx))=1$,  then
\begin{align}\sup\limits_{B_{1/8}(\bx)}|\Rm_{\tilde{g}}|\leq C_0\label{e:uniform-curvature-bound}\end{align} for some absolute constant $C_0$. Conversely, assuming the same Gromov-Hausdorff behavior \eqref{e:GH-behavior}, then \eqref{e:uniform-curvature-bound} implies  $\rank(\Gamma_{\epsilon}(\bx))=1$.
\end{lemma}

Now we return to our regularity scale estimates, which will be proved using the following general argument. First, the upper bound estimate of $r_{k,\alpha}$ is given by the following. Lemma~\ref{l:maximal-rank} tells us that,  as $\epsilon$ sufficiently small, if $\Gamma_{\epsilon}(\bx_0)$ has finite order, then $\rank(\Gamma_{\epsilon}(\bx_0))=0$ and hence curvatures around $\bx_0$ become unbounded.  Using the estimate \eqref{e:curvature-scale} and the Lipschitz property \eqref{1lipreg},  we conclude that  $r_{k,\alpha}(\bx)\leq 2$ for any $\bx\in B_1^{\tilde{g}}(\bx_0)$.
 On the other hand, if $d_{\tilde{g}}(\bx, \bx_0) \geq \frac{1}{2}$ and $\rank(\Gamma_{\epsilon}(\bx))=1$, then Lemma~\ref{l:maximal-rank} and the estimate \eqref{e:curvature-scale} also tell us that $r_{k,\alpha}(\bx)\geq \underline{v}_{k,\alpha}$, where $\underline{v}_{k,\alpha}>0$
is a uniform constant independent of
$\epsilon$.
Therefore, with respect to the original metric $g_\delta$,  the regularity scale estimate
$\underline{v}_{k,\alpha}\cdot \lambda^{-1}\leq r_{k,\alpha}(\bx) \leq 2\lambda^{-1}$ holds if $\rank(\Gamma_{\epsilon}(\bx))=1$ and $\rank(\Gamma_{\epsilon}(\bx_0))=0$
with $1/2 \leq d_{\tilde{g}} (\bx,\bx_0) \leq 1 $.

If higher codimensional collapsing occurs, the main theorem in \cite{NaberZhang} gives the following generalization of Lemma~\ref{l:maximal-rank}.
\begin{theorem}
[\cite{NaberZhang}]\label{t:maximal-nilpotent-rank} Let $(M^n,g,\bx)$ be Einstein with $|\Ric_g|\leq n-1$ and let $(Z^k,h, \underline{z})$ be a $k$-dimensional Riemannian manifold, then there is some constant $0<\epsilon_0<10^{-3}$ depending only on $n$ and the injectivity radius at $\underline{z}\in Z^k$ such that if
\begin{align}
d_{GH}(B_2(\bx), B_2(\underline{z})) < \epsilon
\end{align}
 holds for $\epsilon \le \epsilon_0$, then $\Gamma_{\epsilon}(\bx)\equiv \Image[\pi_1(B_{\epsilon}(\bx))\to \pi_1(B_{1/4}(\bx))]$ has a nilpotent subgroup $\mathcal{N}$ with $\rank(\mathcal{N})\leq n-k$.
Furthermore, the following regularity property holds:  $\rank(\mathcal{N})= n-k$ is equivalent to the uniform curvature estimate
\begin{align}
\sup\limits_{B_{1/8}(\bx)}|\Rm_g|\leq C_0
\end{align}
for some uniform constant $C_0>0$ depending only on $n$ and the injectivity radius at $\underline{z}\in Z^k$.
\end{theorem}

We will next apply the above ideas in our situation.
For each $\bx\in\cK$, we will appropriately choose metric rescaling factor with respect to $\bx$
so that the rescaled convergence has singularity $\bx_0$ with definite distance away from $\bx$.
For each parameter $\delta\in(0,1)$, the space $(\cK, g_\delta^C)$ is divided in the following regions:
\begin{enumerate}
\item For each $p \in \PP^1$ corresponding to a singular fiber of Type $\I_{\nu}$ for some  $\nu\in\dZ_+$, let $\cS_{\I_{\nu}}(p) =\fF^{-1}(B_{2\delta_0}(p))$ be the region equipped with the approximate metric $g_\delta^C$ which satisfies the diameter estimate
\begin{align}
\frac{1}{C_0}\leq \diam_{g_\delta^C}(\cS_{\I_{\nu}}(p))\leq C_0
\end{align}
for some uniform constant $C_0>0$ independent of $\delta$.

\item  Similarly, near each singular fiber of Type $\I_{\nu}^*$ for some  $\nu\in\dZ_+$, we can define $\cS_{\I_{\nu}^*}(p)$ such that \begin{equation}
\frac{1}{C_0}\leq \diam_{g_\delta^C}(\cS_{\I_{\nu}^*}(p)) \leq C_0
\end{equation}
for some uniform constant $C_0>0$ independent of the parameter $\delta$.

\item For each type of singular fibers with finite monodromy, as constructed in Section \ref{s:finite-monodromy},
 denote  by  $\cS_{\II}, \cS_{\III}$, $\cS_{\IV}$, $\cS_{\II^*}$, $\cS_{\III^*}$,  $\cS_{\IV^*}$ and $\cS_{\I_0^*}$,
the corresponding small neighborhoods of singular fibers
such that
the diameter
 is comparable with $\delta^{\ell}$ as $\delta\to0$.

\item
For each $\delta$, denote by $\cR_\delta$ the complement
of the above regions in the total space $ \cK$, on which the approximate metrics $g_\delta^C$ are semi-flat metrics.
\end{enumerate}

The main technical part is to compute the regularity scales and the bubble limits of each singular region in the above list.
First, we study the regularity scale $r_{k,\alpha}$ in $\mathcal{R}_\delta$ as $\delta\to 0$.
It directly follows from
Theorem~\ref{t:maximal-nilpotent-rank} that curvatures are uniformly bounded in any compact subset of $\cR_\delta$
 which has definite distance to the union of singular fibers $\fF^{-1}(\mathcal{S})$.
 Around each boundary component of $\cR_\delta$, let  $\fF^{-1}(p)$ be the closest singular fiber to $\p \cR_\delta$.
 Fix some small constant $\delta_0>0$ independent of $\delta$, then we define a smooth function $\fs_{\cR}(\bx)$ by
\begin{align}\fs_{\cR}(\bx) \equiv
\begin{cases}
d(\bx,\fF^{-1}(p)), & d(\bx,\fF^{-1}(p))\leq 8 \delta_0
\\
1, & d(\bx,\fF^{-1}(p))\geq 16 \delta_0,
\end{cases}
\end{align}
with smooth interpolation in the annulus.
We claim that there are uniform constants $\underline{v}_{k,\alpha}>0$ and $\bar{v}_{k,\alpha}>0$  depending only $k$ and $\alpha$ such that \begin{equation}
\underline{v}_{k,\alpha} \leq \frac{r_{k,\alpha}(\bx)}{\fs_{\cR}(\bx)} \leq \bar{v}_{k,\alpha}.\label{e:equivalent-ratio}
\end{equation}
Indeed,
at the  point $\bx\in\cR_\delta$, if we take the rescaled metric $\tilde{g}_\delta^C\equiv \fs_{\cR}(\bx)^{-2} g_\delta^C$, it follows that $\bx$ has unit distance to $\fF^{-1}(p)$ under the metric $\tilde{g}_\delta^C$.
 Since
$\rank(\Gamma_{\epsilon_0}(\bx))=2$ for small enough $\delta$, all $\bx\in \cR_\delta$ and the constant $\epsilon_0$ in Theorem \ref{t:maximal-nilpotent-rank}, we can obtain the curvature estimates with respect to $\tilde{g}_\delta^C$,
\begin{align} \sup\limits_{B_{1/8}(\bx)}|\Rm_{\tilde{g}_\delta^C}|\leq C_0,
\end{align}
and hence in terms of the original metric $g_\delta^C$, we have
\begin{align}
 	\sup\limits_{B_{\frac{\fs_{\cR}(\bx)}{8}}(\bx)}|\Rm_{g_\delta^C}| \leq \frac{C_0}{\fs_{\cR}(\bx)^2},
\end{align}
where $C_0>0$ is independent of $\delta$. This implies that
\begin{align}
	\frac{r_{k,\alpha}(\bx)}{\fs_{\cR}(\bx)} \geq \underline{v}_{k,\alpha}
\end{align}
for some uniform constant $\underline{v}_{k,\alpha}
$ depending only upon $k$ and $\alpha$.
The upper bound estimate of $r_{k,\alpha}$ follows immediately from the $1$-Lipschitz continuity of $r_{k,\alpha}$ and the fact that curvatures are unbounded in terms of $\delta$ around singular fibers.

The inequality \eqref{e:equivalent-ratio} completely describes
the regularity scales of the regular region $\cR_\delta$.
In the following subsections, we will discuss in detail the regularity scales and bubble limits of the collapsing metrics $g_\delta^C$ in the singular regions $\cK\setminus \cR_\delta$

\subsection{Regularity scales in the cases of finite monodromy}

\label{ss:regularity-scale-ALG}

In this subsection, we will compute the regularity scales near the singular fibers in the cases of finite monodromy. As discussed before, there are seven types of singular fibers in this category:  $\II$, $\III$,  $\IV$, $\II^*$, $\III^*$, $\IV^*$ and $\I_0^*$. It can be seen from the construction in Section \ref{s:finite-monodromy} that the regularity
scales of the singular regions in all those cases
behave in the similar way, so
the singular regions in the above cases are uniformly denoted by
\begin{equation}\cS_{\ALG}\equiv\cS_{\II}\cup \cS_{\III}\cup\cS_{\IV}\cup\cS_{\II^*}\cup\cS_{\III^*}\cup\cS_{\IV^*}\cup\cS_{\I_0^*}.
\end{equation}

Let us briefly recall the construction of  $\cS_{\ALG}$ introduced in Section \ref{s:finite-monodromy}. Let $(\mathcal{G}, g^{\cG}, p)$  be an ALG space with a fixed reference point $p\in \mathcal{G}$. For fixed $\ell\in(0,1)$ and small $\delta$, as in Section \ref{ss:ALG}, we pick a large compact set
 \begin{equation}
\cG(\delta^{\ell-1}) \equiv \cG \setminus \Phi\Big((\delta^{\ell-1},\infty) \times \Sigma^3\Big).
 \end{equation}
Here $\Phi$ is the map defined in Proposition \ref{p:exact-error-ALG}. It immediately follows that
\begin{equation}
\underline{c}_0\cdot \delta^{\ell-1}
 \leq \diam_{g^{\cG}}(\cG(\delta^{\ell-1})) \leq\bar{c}_0\cdot  \delta^{\ell-1}
\end{equation}
for some constants $\underline{c}_0>0$ and $\bar{c}_0>0$ independent of the parameter $\delta$.
In the gluing construction in Section \ref{ss:approximate-metric-ALG}, we rescaled the large subset $\cG(\delta^{\ell-1})$ and glued it with $\omega_\delta^{B}$ around the corresponding connected component of $\p\mathcal{R}_\delta$
 so that $\cS_{\ALG}$ is diffeomorphic to $\cG(\delta^{\ell-1})$ and its diameter yields to the estimate
\begin{equation}
\underline{c}_0\cdot \delta^{\ell}
 \leq \diam_{g_\delta^C}(\cS_{\ALG}) \leq\bar{c}_0\cdot  \delta^{\ell}.
\end{equation}

The proposition below gives the bubbling analysis and regularity scale
estimates for the points in the region $\cS_{\ALG}$.

\begin{proposition}[Regularity scale and bubble limits near finite monodromy fibers]
\label{p:regularity-scale-ALG}
For any small parameter $\delta\ll1$, let $g_\delta^C$ be the approximately hyperk\"ahler metric on $\cS_{\ALG}$, then the following holds:
\begin{enumerate}
\item Given $k\in\mathbb{N}$ and $\alpha\in(0,1)$,
there are uniform constants $\underline{v}_{k,\alpha}>0$ and $\bar{v}_{k,\alpha}>0$ independent of $\delta>0$ such that the $(k,\alpha)$-regularity scale $r_{k,\alpha}(\bx)$ at $\bx\in \cS_{\ALG}$ has the following bound
\begin{equation}
\underline{v}_{k,\alpha}\cdot   \fs_{\cG}(\bx) \leq r_{k,\alpha}(\bx)\leq \bar{v}_{k,\alpha} \cdot\fs_{\cG}(\bx) ,
\end{equation}
where the function $\fs_{\cG}(\bx)$ is smooth on $\cS_{\ALG}$ and  explicitly given by
\begin{align}
\fs_{\cG}(\bx)
=
\begin{cases}
\delta, &   d_{g_\delta^C}(\Psi(p),\bx) \leq \delta,
\\
  d_{g_\delta^C}(\Psi(p),\bx), &  d_{g_\delta^C}(\Psi(p),\bx)\geq 2\delta.\end{cases}
\end{align}

\item The canonical bubble of $(\cK,\bx)$ for any $\bx\in\cS_{\ALG}$ is either a complete hyperk\"ahler ALG space $\cG$ or its asymptotic $2$-dimensional cone $T_{\infty}(\cG)\equiv C(S_{2\pi\beta}^1)$ with the flat metric $d_{C,2 \pi\beta}$.

\end{enumerate}
\end{proposition}

\begin{proof}

The proof can be achieved in the following way. Now suppose the rescaled sequence at $\bx$
\begin{equation}
(\cK, \fs(\bx)^{-2}g_\delta^C, \bx) \xrightarrow{GH} (\widehat{X}_{\infty},\hat{d}_{\infty}, \bx_{\infty}),\quad \delta\to 0,
\end{equation}
satisfies the property that the regularity scale at $\bx$ with respect to the rescaled metrics are uniformly bounded from above and below. Then under the original metric $g_\delta^C$,
the regularity scale at $\bx$ is comparable with $\fs(\bx)$ and $\widehat{X}_{\infty}$ is the canonical bubble limit at $\bx$. Notice that by Proposition \ref{p:complexerror}, the rescaling argument is the same for $g_\delta^C$ and $\Psi_* g_\delta^{\cG}$.

To accomplish this, we will
divide the above singular region $\cS_{\ALG}$ into several pieces.
\begin{flushleft}
{\bf Region $\cS_{\ALG,1}$ (deepest ALG bubbles):}
\end{flushleft}
This region consists of the points $\bx\in\cS_{\ALG}$ satisfying
$d_{g_\delta^C}(\bx,\Psi(p)) \leq \delta$.
As $\delta\to0$, taking any sequence of points $\bm{x}_\delta\in \cS_{\ALG,1}$, we rescale the metric $g_\delta^C$  by
\begin{equation}
\tilde{g}_\delta^C = \delta^{-2}g_\delta^C.
\end{equation}

By Proposition \ref{p:complexerror},
\begin{equation}
(\cK, \tilde{g}_\delta^C, \bx_\delta) \xrightarrow{C^{\infty}}  (\cG, g^{\cG}, \bx_{\infty})\ \text{as}\ \delta\to 0,
\end{equation}
where the limit space is a hyperk\"ahler ALG space and the convergence is in the pointed $C^k$-topology for any $k\in\mathbb{N}$. Since $(\cG,g^{\cG},\bx)$ is not flat, so under $g^{\cG}$, $\underline{v}_{k,\alpha}\leq r_{k,\alpha}(\bx)\leq \bar{v}_{k,\alpha} $.
Changing back to $g_\delta^C$, we have $\delta\cdot \underline{v}_{k,\alpha}\leq r_{k,\alpha}(\bx)\leq \delta \cdot \bar{v}_{k,\alpha}$ and hence
the function $\fs(\bx)$ in $\cS_{\ALG,1}$ can be simply chosen as $\delta$.

\begin{flushleft}
{\bf Region $\cS_{\ALG,2}$ (ALG bubble damage zone):}
\end{flushleft}
The points $\bx$ in this region satisfy
\begin{equation}
d_{g_\delta^C}(\bx, \Psi(p)) \geq 2\delta.
\end{equation}
Now we consider the rescaled metric for a sequence of reference points $\bx_\delta\in \cS_{\ALG,2}$,
\begin{equation}
\tilde{g}_\delta^C =   d_{g_\delta^C}(\Psi(p),\bx_\delta)^{-2} \cdot g_\delta^C.
\end{equation}
 Then there are two different rescaled limits depending upon the distance $d_{g_\delta^C}(\Psi(p),\bx_\delta)$: \begin{flushleft}
Case (a): As $\delta\to0$, the sequence  $\{\bx_\delta\}$ satisfies that there is some $\sigma_0>0$ independent of $\delta>0$ such that  $0< \sigma_0 \leq \frac\delta{d_{g_\delta^C}(\bx_\delta,\Psi(p))} \leq \frac{1}{2}$.
\end{flushleft}

In this case, the rescaled limit $(\widehat{X}, \hat{d}, \hat{\bx})$ is a complete ALG space which is a finite rescaling of the original ALG bubble
$(\cG, g^{\cG}, p)$.

\begin{flushleft}
Case (b): As $\delta\to0$, the sequence $\{\bx_\delta\}$ satisfies $\frac\delta{d_{g_\delta^C}(\bx_\delta,\Psi(p))}\to 0$.
\end{flushleft}
In this case, the rescaled limit is isometric to
a flat cone $(C(S_{2\pi\beta}^1),d_{C,2\pi\beta},\hat{\bx})$ in $\dR^2$ which is exactly the asymptotic cone of the ALG space $(\cG, g^{\cG}, p)$. Notice that in this case, we have $d_{C, 2 \pi \beta}(\hat\bx,0^2)=1$ and the ball $B_{10^{-3}}(\hat\bx)$ is contractible in the flat cone $(C(S_{2\pi\beta}^1), d_{C,2\pi\beta})$. So for the fixed constant $\epsilon=10^{-4}\epsilon_0$ as in Theorem~\ref{t:maximal-nilpotent-rank}, for small enough $\delta$, the group $\Image[\pi_1(B_{\epsilon}(\hat\bx))\to\pi_1(B_{10^{-3}}(\hat\bx))]$ is $\ZZ\oplus\ZZ$, where the geodesic balls $B_{\epsilon}(\hat\bx)$ and $B_{10^{-3}}(\hat\bx)$ are measured under the rescaled metric $\tilde{g}_{\delta}^C$. So by Theorem~\ref{t:maximal-nilpotent-rank}, we get the required lower bound on the regularity scale. On the other hand, the upper bound on the regularity scale comes from Remark~\ref{r:1-Lipchitz} and the calculation in Region $\cS_{\ALG,1}$.

This completes the proof of the proposition.
\end{proof}

\subsection{Bubbling analysis around the singular fiber $\I_{\nu}$}

\label{ss:regularity-scale-I_v}

We will next study the regularity scales and classify the canonical bubbling limits in the singular region $\cS_{\I_{\nu}}$ for some fixed $\nu\in\dZ_+$ which is a neighborhood
\begin{align}
\cS_{\I_{\nu}} \equiv \fF^{-1} (B_{4\delta_0}(p))
\end{align}
of the singular fiber $\fF^{-1}(p)$ of Type $\I_{\nu}$, where $\delta_0$ is a constant independent of $\delta$.

In Region $\cS_{\I_{\nu}}$, as shown in Section \ref{s:infinite-monodromy},
 there are approximately hyperk\"ahler metrics $g_\delta^B$ (which agree in this region with the metrics $g_\delta^C$ constructed in Section~\ref{s:finite-monodromy}), which are constructed by gluing the multi-Ooguri-Vafa metric on $\mathcal{N}_{\nu}^4$ with semi-flat metrics, so that for some constant $C_0>0$ independent of $\delta$,
\begin{equation}
\frac{1}{C_0} \leq \diam_{g_\delta^C} (\cS_{\I_{\nu}}) \leq C_0.\end{equation}

\begin{remark}
In the following, for the convenience of computations, the metrics $g_\delta^C$ will be parametrized in terms of
\begin{equation}
\label{e:T-delta-relation}
T \equiv -\nu\log\delta,
\end{equation} and denoted by $g_T$.
\end{remark}

Recall that the metric $g_T$ in $\fF^{-1} (B_{\delta_0}(p))$ coincides with the Gibbons-Hawking metric
\begin{equation}
g_{\delta, \nu} \equiv \frac{1}{2\pi}e^{-\frac{2T}{\nu}}\cdot \Big(V_T\cdot (du_1^2 + du_2^2 + du_3^2)  + (V_T)^{-1}\theta^2\Big),\quad (u_1,u_2,u_3)\subset \mathcal{O}\subset \dR^2\times S^1,
\end{equation}
 where the region $\mathcal{O}$ consists all points \begin{align}\underline{\bx}\equiv(u_1,u_2,u_3)\in Q^3=\dR^2\times S^1 \end{align} satisfying
 $|\delta(u_1+\sqrt{-1}u_2)|\le 2\delta_0$.
Let us denote by $\pi_{Q^3}$ the $S^1$-principal bundle map
\begin{equation}
S^1\to \mathcal{N}_{\nu}^4 \xrightarrow{\pi_{Q^3}} \cO\subset \dR^2\times S^1,
\end{equation}
which sends every $\bx\in \mathcal{N}_{\nu}^4 $ to $\underline{\bx}=\pi_{Q^3}(\bx)\subset \cO$.

We need the following conventions and notations:

\begin{enumerate}

\item

Denote by $d_{Q^3}:Q^3\times Q^3 \to [0,+\infty)$ the distance function on $Q^3=\dR^2\times S^1$. Then we obtain
 a smooth function $\fr:Q^3 \to\dR_+$ by slightly mollifying the distance function $d_{Q^3}$ as follows:  For $\underline{\bx}\in Q^3$,
\begin{align}
\fr(\underline{\bx}) = \begin{cases}
T^{-1},  &  d_{Q^3}(\underline{\bx},p_i) \leq  T^{-1}\ \text{for some} \ p_i\in P,
\\
d_{Q^3}(\underline{\bx}, p_i), & 2T^{-1} \leq  d_{Q^3}(\underline{\bx},p_i) \leq \iota_0  \ \text{for some} \ p_i\in P,
\\
d_{Q^3}(\underline{\bx}, 0^3),  & T_0\leq d_{Q^3}(\underline{\bx},p_i) \leq  e^{\frac{T-2}{\nu}}\ \text{for all} \ p_i\in P,
\\
e^{\frac{T-2}{\nu}},       & d_{Q^3}(\underline{\bx},p_i)\geq e^{\frac{T-1}{\nu}} \ \text{for all} \ p_i\in P.
\end{cases}\label{e:smoothing-flat-r}
\end{align}
Here $P=\{p_i\}_{i=1}^{\nu}\subset Q^3$ is the finite set of monopole points and
\begin{align}\iota_0 & \equiv \frac{1}{2} \min \big\{d_{Q^3}(p_i,p_j) \ | \ p_i,p_j\in Q^3, \ i\neq j \big\},
\\
T_0 & \equiv 2\max\{d_{Q^3}(p_i,0^3) \ | \ p_i,p_j\in Q^3, \ 1\leq i\leq\nu\}.
\end{align}

\item Let $L_0:[0,+\infty)\rightarrow \dR$ be a smooth function  satisfying
\begin{align} \label{e:l(r)}
L_0(r)\equiv
\begin{cases}
-\nu\log r, & r\geq 4,
\\
0 ,  & 0\leq r\leq 2,
\end{cases}
\end{align}
and let
\begin{equation} \label{e:L_T(r)}
\quad L_T(\bx)\equiv T+L_0(d_{Q^3}(\pi_{Q^3}(\bx),0^3)).
\end{equation}

\end{enumerate}

We next describe the regularity scales and classify the canonical bubble limits for all $\bx\in\cS_{\I_{\nu}}$.
\begin{proposition}
[Regularity scale and bubble limits near $\I_{\nu}$ fibers]
\label{p:regularity-scale-I_v} Given any large parameter $T\gg1$ (equivalently $\delta\ll1$),  let $g_T = g_\delta^C$ be the approximately hyperk\"ahler metric  on $\cS_{\I_{\nu}}$ as the above, then the following holds:
\begin{enumerate}
\item
Given $k\in\mathbb{N}$ and $\alpha\in(0,1)$,
there are uniform constants $\underline{v}_{k,\alpha}>0$ and $\bar{v}_{k,\alpha}>0$ independent of $\delta>0$ such that the $(k,\alpha)$-regularity scale $r_{k,\alpha}(\bx)$ at $\bx\in \cS_{\I_{\nu}}$ has the following bound
\begin{equation}
\underline{v}_{k,\alpha}    \leq \frac{r_{k,\alpha}(\bx)}{\fs_{\nu}(\bx)}\leq \bar{v}_{k,\alpha},
\end{equation}
where the function $\fs_{\nu}(\bx)$ is explicitly given by
\begin{align}
\fs_{\nu}(\bx) = e^{-\frac{T}{\nu}} \cdot L_T(\bx)^{\frac{1}{2}}\cdot  \fr(\underline{\bx}).
\end{align}

\item As $T\to +\infty$ (equivalently $\delta\to 0$), all the canonical bubbles of $(\cK,\bx)$ for $\bx\in\cS_{\I_{\nu}}$ are as follows: the Ricci-flat Taub-NUT space $(\dC^2, g_{TN})$, the flat spaces $\dR^3$, $\dR^2\times S^1$, $\dR^2$, and singular limit $(\PP^1,d_{ML})$ with bounded diameter.
\end{enumerate}

\end{proposition}

\begin{proof}

The basic framework  is similar to the proof of Proposition \ref{p:regularity-scale-ALG}. We aim at obtaining the required estimate of the regularity scale for every $\bx\in\cS_{\I_{\nu}}$, so it suffices to show that for all $T\to+\infty$ (or equivalently $\delta\to 0$) and for all $\bx\in \cS_{\I_{\nu}}$, there is  a rescaled limit
\begin{equation}
(\cK, \fs_{\nu}(\bx)^{-2}g_\delta^C, \bx) \xrightarrow{GH} (\widehat{X}_{\infty},\hat{d}_{\infty}, \bx_{\infty}),
\end{equation}
so that the regularity scale at $\bx$ with respect to the rescaled metrics $\fs(\bx)^{-2}g_\delta^C$ is uniformly bounded from above and below.

\begin{remark}For the convenience of the discussion,
we will take a sequence $T_j\to+\infty$ and denote $g_j\equiv g_{T_j}$ in the following proof.
\end{remark}
The singular region $\cS_{\I_{\nu}}$ will be further divided in the several pieces:

\begin{flushleft}
{\bf Region $\cS_{\I_{\nu},1}$ (deepest ALF bubbles):}
\end{flushleft}

 This region consists of the points $\bx\in\cS_{\I_{\nu}}$ satisfying $d_{Q^3}(\underline{\bx},p_i) \leq T_j^{-1}$ for some $1\leq i\leq \nu$, where $\underline{\bx}=\pi_{Q^3}(\bx)$.
Let $\bx_j$ be a sequence of reference points in this region and
we rescale the metric by $\tilde{g}_{j} \equiv \lambda_j^2\cdot g_{j}$ with  \begin{equation}\lambda_j\equiv e^{\frac{T_j}{\nu}}\cdot T_j^{\frac{1}{2}}.\end{equation}
If we rescale the coordinates by
\begin{equation}
(\tilde{u}_1, \tilde{u}_2, \tilde{u}_3) = T_j \cdot ( (u_1, u_2, u_3) - p_i ),
\end{equation}
then the rescaled metrics converge in the pointed $C^{\infty}$-topology,
 \begin{align}
 (\cK ,  \tilde{g}_{j},  \bx_j)  \xrightarrow{C^{\infty}} (\dC^2, g_{TN}, \bx_{\infty}),
 \end{align}
where the limit $(\dC^2,g_{TN},\bx_{\infty})$ is a rescaled Taub-NUT metric. By definition, the Taub-NUT metric is the hyperk\"ahler ALF metric given by the Gibbons-Hawking ansatz over $\RR^3$ with a single monopole point and harmonic function $V = 1 + (2r)^{-1}$, see \cite{Hawking}.
The rescaling computations are straightforward and hence we skip the details here, see  \cite[Lemma~7.9]{HSVZ} for more details. Since the $(\dC^2, g_{TN})$ is not flat with bounded curvatures and the convergence is $C^{\infty}$, the regularity scale at $\bx_j$ with respect to the original metrics $g_{j}$ satisfies
\begin{align}
\underline{v}_{k,\alpha}
\cdot T_j^{\frac{1}{2}}\leq r_{k,\alpha}(\bx_j) \leq \bar{v}_{k,\alpha}\cdot T_j^{\frac{1}{2}},\end{align} where $\underline{v}_{k,\alpha}>0$ and  $\bar{v}_{k,\alpha}>0$ are uniform constants independent of $T_j$.

\begin{flushleft}
{\bf Region $\cS_{\I_{\nu},2}$ (ALF bubble damage zone):}
\end{flushleft}
This region contains all the points $\bx\in\cS_{\I_{\nu}}$ satisfying
$2T^{-1}\leq d_{Q^3}(\underline{\bx},p_i) \leq \iota_0$  for some $1\leq i\leq \nu$.
Now take a sequence of reference points $\bx_j$ in  Region $\cS_{\I_{\nu},2}$ and let us denote
\begin{equation}
\fr_j\equiv \fr(\underline{\bx}_j)
\end{equation}
for $\underline{\bx}_j=\pi_{Q^3}(\bx_j)$.
The rescaled metrics $\tilde{g}_j=\lambda_j^2 g_j$ is given by
\begin{equation}
\lambda_j = e^{\frac{T_j}{\nu}}\cdot T_j^{-\frac{1}{2}}\cdot \fr_j^{-1}.
\end{equation}
To effectively work with the explicit rescaled metric tensor $\tilde{g}_j$, we also rescale the coordinates by taking
\begin{equation}
(\tilde{u}_1, \tilde{u}_2, \tilde{u}_3)  \equiv \mu_j\cdot ( (u_1, u_2, u_3)-p_i),
\end{equation}
where the coordinate rescaling factor is chosen as
\begin{equation}
\mu_j \equiv \frac{1}{\fr_j}.
\end{equation}

To understand the rescaled limits under the above rescaled metric $\tilde{g}_j$,
we need to consider the following cases and study the rescaled limits separately:

\begin{flushleft}
  Case (a): there is some $\sigma_0>0$ such that  $2T_j^{-1}\leq \fr_j \leq  \frac{1}{(\sigma_0)^2}\cdot  T_j^{-1}$.
\end{flushleft}

\begin{flushleft}
  Case (b): $\fr_j$ satisfies
  \begin{equation}
  \frac{\fr_j}{T_j^{-1}}\to +\infty, \quad \fr_j \to 0.
  \end{equation}
\end{flushleft}

\begin{flushleft}
  Case (c):
there is some uniform constant $r_0>0$ such that $r_0 \leq \fr_j \leq \iota_0$.
\end{flushleft}

Case (a) is similar to Region $\cS_{\I_{\nu},1}$, so  the rescaled limit is the Ricci-flat Taub-NUT space. The metric tensor can be written explicitly in terms of the rescaled coordinates.

In Case (b), we take a fixed region in $Q^3$ where $\underline{\bx}=(u_1,u_2,u_3)$ satisfies $d_{Q^3}(\underline{\bx},p_i) \leq \iota_0$ and investigate the Gromov-Hausdorff convergence
under the rescaled metrics $\tilde{g}_j$. First, the $S^1$-fibers are collapsing as $j\to+\infty$. Moreover, it can be directly verified that, up to local universal covers,
the rescaled metrics converge in the $C^k$-topology for any $k\in\mathbb{N}$. This in particular implies that
the
rescaled coordinate system $(\tilde{u}_1,\tilde{u}_2,\tilde{u}_3)\in Q^3$ in fact
converges to some limiting coordinate system $(u_{1,\infty},u_{2,\infty}, u_{3,\infty})$.
After a fixed rescaling so that the limiting reference point $\bx_{\infty}$ satisfies $d_{\tilde{g}_{\infty}}(\bx_{\infty},0^3)=1$, with respect to the limiting coordinates, the limiting metric tensor $\tilde{g}_{\infty}$ can be
written as
\begin{equation}
\tilde{g}_{\infty} = du_{1,\infty}^2 + du_{2,\infty}^2 + du_{3,\infty}^2,
\end{equation}
which is precisely the Euclidean metric in $\dR^3$.
Detailed computations about this can also be found in the bubbling analysis in Region II of \cite[Section~7.3]{HSVZ}.
  Moreover, the Gibbons-Hawking ansatz gives an $S^1$-fibration over $B_{R}(0^3)\subset \dR^3$ for any large $R>0$, which is smooth away from $0^3$. Since $0^3\not\in B_{\frac{2}{3}}(\bx_{\infty})$, it immediately follows from the $S^1$-principal bundle structure that, for sufficiently small $\epsilon$,
\begin{align}
\Image[\pi_1(B_{\epsilon}^{\tilde{g}_j}(\bx_j))\to\pi_1(B_{2/3}^{\tilde{g}_j}(\bx_j))]\cong \dZ
\end{align}
so that its rank is $1$. Then Lemma \ref{l:maximal-rank} shows that curvatures at $\bx_j$ are uniformly bounded and hence $r_{k,\alpha}(\bx_j)$ with respect to the rescaled metrics $\tilde{g}_j$
is uniformly bounded from below. To see the uniform upper bound of $r_{k,\alpha}(\bx)$, it suffices to use the $1$-Lipschitz continuity of $r_{k,\alpha}$ and the calculation in Region $\cS_{\I_{\nu},1}$.

Now we study the canonical bubble in
 Case (c).  We will work a domain in $\mathcal{N}_{\nu}^4$ which consists of the points $\bx\in\cS_{\I_{\nu}}$ satisfying
 \begin{equation}
\fr(\underline{\bx}) \leq T_j.\end{equation}
for $\underline{\bx}=\pi_{Q^3}(\bx)$.
 Notice that, by our definition of the metric rescaling factor $\lambda_j$ and the coordinate rescaling factor $\mu_j$,
 it follows that the rescaled coordinate system $(\tilde{u}_1, \tilde{u}_2, \tilde{u}_3)$
 is a bounded rescaling for the original one $(u_1,u_2,u_3)$.
 Therefore, the limiting metric $\tilde{g}_{\infty}$ is up to a fixed rescaling, given by
 \begin{equation}
\tilde{g}_{\infty} = du_{1,\infty}^2 + du_{2,\infty}^2 + du_{3,\infty}^2,
\end{equation} which is the standard flat product metric on $\dR^2\times S^1$.
This corresponds to the case of $\dT^2\times \dR$ in \cite[Section~7.3]{HSVZ}.
The remainder of the arguments are the same as Case (b) by noticing that there are $\nu$ monopoles of the Green's function on $\dR^2\times S^1$ and $\bx$ has definite distance to those monopoles under the rescaled metrics.

This completes the bubble classification in Region $\cS_{\I_{\nu},2}$.

\begin{flushleft}
{\bf Region $\cS_{\I_{\nu},3}$ (large scale regions):}
\end{flushleft}

This region consists of the points $\bx\in \cS_{\I_{\nu}}$ which satisfy
\begin{equation} d_{Q^3}(\underline{\bx},p_i)  \geq T_0\ \text{and} \ L_T(\bx) ) \geq  2,\end{equation} where $\underline{\bx}=\pi_{Q^3}(\bx)\in Q^3$.
For a sequence of reference points $\bx_j\in \cS_{\I_{\nu},3}$ with projected coordinates $\underline{\bx}_j=(u_{1,j},u_{2,j},u_{3,j})\in Q^3$, let us define the rescaled metric $\tilde{g}_j = \lambda_j^2 \cdot g_j$ by
\begin{equation}
\lambda_j = e^{\frac{T_j}{\nu}}\cdot L_j^{-\frac{1}{2}}\cdot \fr_j^{-1},
\end{equation}
where
\begin{equation}
L_j\equiv L_{T_j}(\bx_j) \ \text{and}\ \fr_j \equiv \fr(\underline{\bx}_j).
\end{equation}
For explicit computations, we will also rescale the coordinate system $\underline{\bx} = (u_1, u_2, u_3)\in\mathcal{N}_{\nu}^4$ and center around the reference point $\bx_j$ with base coordinate $\underline{\bx}_j=(u_{1,j},u_{2,j},u_{3,j})\in Q^3$ to
\begin{equation}
\tilde{u}_1 = \mu_j\cdot u_1,\ \tilde{u}_2 = \mu_j\cdot u_2,\ \tilde{u}_3 = \mu_j\cdot u_3,
\end{equation}
where the rescaling factor $\mu_j$ is chosen as
$\mu_j = \fr_j^{-1}$.

In the following computations, Region $\cS_{\I_{\nu},3}$ will be further subdivided in the following cases depending upon the location of $\bx_j$:

\begin{flushleft}
  Case (a): There is some $R_0>0$ such that  $\frac{1}{2}\leq \fr_j \leq  R_0$.
\end{flushleft}

\begin{flushleft}
Case (b): This case is given by the condition
\begin{equation}\fr_j\to+\infty \quad \text{and}\quad  L_j \to +\infty.\end{equation}
\end{flushleft}

\begin{flushleft}
Case (c): There is some $T_0>0$ such that $2\leq L_j  \leq T_0$.
\end{flushleft}

In Case (a), the metric rescaling is equivalent to Case (c)
of Region $\cS_{\I_{\nu,2}}$. Moreover, the coordinate system is rescaled by uniformly bounded constants $\mu_j$ which gives the rescaled limit
$\dR^2\times S^1$ with up to a fixed rescaling, the standard flat product metric
\begin{equation}
\tilde{g}_{\infty}  = d u_{1,\infty}^2 + d u_{2,\infty}^2 + d u_{3,\infty}^2.
\end{equation}
Let $P\subset \dR^2\times S^1$ be the finite set of monopoles, then $\bx_j$ has definite distance to $\mathcal{P} = \pi_{Q^3}^{-1}(P)$. So applying the same argument as in Case (c) of Region $\cS_{\I_{\nu},2}$, we have
\begin{align}
\underline{v}_{k,\alpha}
\leq r_{k,\alpha}(\bx_j) \leq \bar{v}_{k,\alpha}
\end{align}
with respect to the rescaled metrics $\tilde{g}_j$. Therefore, changing back to $g_j$, we have
\begin{align}
\underline{v}_{k,\alpha} \cdot e^{-\frac{T_j}{\nu}} \cdot L_j^{\frac{1}{2}}\cdot \fr_j
\leq r_{k,\alpha}(\bx_j) \leq \bar{v}_{k,\alpha} \cdot e^{-\frac{T_j}{\nu}} \cdot L_j^{\frac{1}{2}}\cdot \fr_j,
\end{align}
where $\underline{v}_{k,\alpha}>0$ and $\bar{v}_{k,\alpha}>0$ are uniform constants independent of $T_j$.

In Case (b), for any fixed $\xi>1$
we choose a sequence of large regions in $\mathcal{N}_{\nu}^4$ defined by
\begin{equation}
\mathcal{U}_j(\xi) \equiv \Big\{\bx\in\mathcal{N}_{\nu}^4\Big|\fr_j^{-1}\cdot d_{Q^3}(\underline{\bx},0^3)\leq \xi,\ \underline{\bx}=\pi_{Q^3}(\bx) \Big\}.
\end{equation}
We will show that $\tilde{g}_j$ on
$\mathcal{U}_j(\xi)$ converges to the Euclidean metric of $B_\xi(0^2) \subset\dR^2$ up to a fixed rescaling.  The rescaled metrics are given by \begin{align}
\tilde{g}_j =  \lambda_j^2 g_j & = \frac{1}{2\pi}\fr_j^{-2} \cdot \frac{1}{L_j}\cdot  \Big(V_{T_j}\cdot (d u_{1}^2 + d u_{2}^2 + d u_{3}^2)  + (V_{T_j})^{-1}\theta^2\Big)
\nonumber\\
& =  \frac{1}{2\pi}\frac{1}{L_j}\cdot  \Big(V_{T_j}\cdot (d \tilde{u}_{1}^2 + d \tilde{u}_{2}^2 + d \tilde{u}_{3}^2)  + (V_{T_j})^{-1}\theta^2\Big).
\end{align}
Let
\begin{equation}
\mathcal{V}_j(\xi) \equiv \Big\{\bx\in\mathcal{N}_{\nu}^4\Big|\fr_j^{-1}\cdot d_{Q^3}(\underline{\bx},0^3)\leq \xi^{-1},\ \underline{\bx}=\pi_{Q^3}(\bx) \Big\}.
\end{equation}
 In the following, the main part is to show that for each $\bx\in \mathcal{U}_j(\xi)\setminus \mathcal{V}_j(\xi)$,
it holds that as $j\to+\infty$,
\begin{equation}
\frac{V_{T_j}(\bx)}{L_j} \to 1.
\end{equation}
Indeed, let us take any arbitrary point $\bx\in \mathcal{U}_j(\xi) \setminus \mathcal{V}_j(\xi)$ over $\underline{\bx}=(u_1,u_2,u_3)
\in Q^3$, we have
 \begin{equation}
 \xi^{-1} \leq \frac{\fr(\bx)}{\fr_j} \leq \xi.
 \end{equation}
So it follows that
\begin{align}
\frac{V_{T_j}(\bx)}{L_j}  = 1 - \frac{\nu}{L_j}\log\Big(\frac{\fr(\bx)}{\fr_j}\Big) + o(1) = 1 + o(1).
\end{align}
Since $\xi$ is arbitrary, and $\lim\limits_{\xi\to \infty}\limsup\limits_{j\to \infty}\diam_{\tilde{g}_j} \mathcal{V}_j(\xi)=0$,
this tells us that $\tilde{g}_j$ converges to
\begin{equation}
\tilde{g}_{\infty} = d u_{1,\infty}^2 + d u_{2,\infty}^2,
\end{equation}
up to a fixed rescaling, where $(u_{1,\infty} , u_{2,\infty})$ is the
canonical coordinate system on $\dR^2$ which corresponds to the limit of the rescaled coordinates $(\tilde{u}_1,\tilde{u}_2)$.
Moreover, the finite set of monopoles $\mathcal{P}$ converges to $0^2 \in\dR^2$ and $\bx_j$ converges to $\bx_\infty \in \partial B_1(0^2) \subset \dR^2$.

To finish Case (b), the regularity scale estimate at $\bx_j$ can obtained from Theorem \ref{t:maximal-nilpotent-rank}.
 Topologically, $B_{1/2}^{\tilde{g}_j}(\bx_j)$ is the trivial torus fibration over the base. Let $\epsilon=\epsilon(n)>0$ defined in Theorem \ref{t:maximal-nilpotent-rank}, it is easy to see that \begin{equation}
 \Image[\pi_1(B_{\epsilon}^{\tilde{g}_j}(\bx_j))\to\pi_1(B_{1/2}^{\tilde{g}_j}(\bx_j))]\cong \dZ \oplus \dZ.
 \end{equation}
  It has rank 2. So we get the required lower bound on $r_{k,\alpha}(\bx_j)$ by Theorem \ref{t:maximal-nilpotent-rank}. The upper bound on $r_{k,\alpha}(\bx_j)$ follows from Remark \ref{r:1-Lipchitz} and the calculation in Case (a). So the proof in Case (b) is done.

The rescaled geometry in Case (c) is much simpler. In fact, the metric rescaling factor $\lambda_j$ is uniformly bounded in this case so that
the rescaled limit
 is
$(\PP^1,d_{ML},\bx_{\infty})$ such that the McLean metric $d_{ML}$ is singular on a finite set $\mathcal{S}\subset \PP^1$
and the limit of the reference points $\bx_{\infty}\in\PP^1\setminus\mathcal{S}$.
The uniform estimate for the regularity scale at $\bx_j$ just immediately follows.

This completes the proof of the proposition.
\end{proof}

\subsection{Bubbling analysis around the singular fiber $\I_{\nu}^*$}

This subsection is dedicated to the analysis of the singularity behavior around a singular fiber of Type $\I_{\nu}^*$ for some $\nu\in\dZ_+$.
The associated singular region $\cS_{\I_{\nu}^*}$ was constructed in Section \ref{ss:qmov} which was given by resolving the four singular points from an orbifold multi-Ooguri-Vafa region.
In terms of the elliptic fibration $\fF:\cK\to\PP^1$, the singular region $\cS_{\I_{\nu}^*}$
can be represented as the tubular neighborhood
$\cS_{\I_{\nu}^*}\equiv \fF^{-1}(B_{4\delta_0}(p))$
 of the singular fiber $\fF^{-1}(p)$ of Type $\I_{\nu}^*$ ($\nu\in\dZ_+$).

To be precise, given a family of small parameters $0<\delta\ll1$, let  $g_\delta^C$ be the approximately hyperk\"ahler metrics on the resolution of the orbifold region  $\mathcal{N}_{2\nu}^4/\dZ_2$,  which are constructed in Section \ref{ss:qmov}
which is the model of $\cS_{\I_{\nu}^*}$.
Recall that the main geometric features of the region
can be summarized as follows:
\begin{enumerate}
\item
Denote by $(\mathcal{N}_{2\nu}^4,g_{\delta,2\nu})$ the incomplete domain  endowed with a family of collapsing multi-Ooguri-Vafa metric $g_{\delta,2\nu}$ with $\nu$-pairs of monopoles such that  $\mathcal{N}_{2\nu}^4 $ is a principal $S^1$-bundle over an open subset in $Q^3=\dR^2\times S^1$ with circles vanishing at monopole points $P\equiv\{p_1,\hat{p}_1,\ldots, p_{\nu},\hat{p}_{\nu}\}\subset Q^3$.

\item $\mathcal{N}_{2\nu}^4$ admits an isometric $\dZ_2$-action which
descends to the base $Q^3=\dR^2\times S^1$
such that the four $S^1$-fixed points $\{q_1,q_2,q_3,q_4\}$ on $ \mathcal{N}_{2\nu}\setminus \pi_{Q^3}^{-1}(P)$ are sent to the two $S^1$-fixed points
\begin{align}
\pi_{Q^3}(q_1)=\pi_{Q^3}(q_2)=q_-=(0,0,0),\\
 \pi_{Q^3}(q_3)=\pi_{Q^3}(q_4)=q_+ = (0,0,\frac{1}{2})
\end{align}
on the base $Q^3=\dR^2\times S^1$.

\item The $\dZ_2$-invariant multi-Ooguri-Vafa metric $g_{\delta,2\nu}$
has a natural $\dZ_2$ quotient orbifold metric $\check{g}_{\delta,\nu}$
on $\mathcal{N}_{2\nu}^4/\dZ_2$.
To resolve the four orbifold singularities, we constructed in Section \ref{ss:qmov} a family of approximately hyperk\"ahler metrics, denoted by $g_\delta^C$,  by gluing four copies of Eguchi-Hanson metrics
with $\check{g}_{\delta,\nu}$ around the four singular points.

\end{enumerate}

Following the notation in Section \ref{ss:regularity-scale-I_v}, let $T\gg1$ satisfy $\delta= e^{-\frac{T}{2\nu}}$ and
let the approximate metrics $g_\delta^C$ parametrized in terms of $T$, denoted by $g_T$.
 We introduce the following notation  in the current $\I_{\nu}^*$-case:
\begin{enumerate}
\item Let $g_{2\nu,T}$ be the rescaled metric
$g_{2\nu,T} = e^{\frac{T}{\nu}} g_T$,
which has diameter comparable with $e^{\frac{T}{2\nu}}$ for $T\gg1$.
There are positive constants $\underline{\iota}_0>0$  and $\bar{\iota}_0>0$  such that for $ i,j\in\{1,2,3,4\} $,
\begin{equation}
\frac{1}{\underline{\iota}_0}\cdot T^{-\frac{1}{2}}\leq d_{g_{2\nu,T}}(q_i,q_j) \leq   \underline{\iota}_0 \cdot T^{-\frac{1}{2}}
\end{equation} if $\pi_{Q^3}(q_i)=\pi_{Q^3}(q_j)$,
and
\begin{equation}
4 \cdot \bar{\iota}_0 \cdot T^{\frac{1}{2}}\leq d_{g_{2\nu,T}}(q_i,\bx) \leq \frac{1}{\bar{\iota}_0} \cdot T^{\frac{1}{2}}
\end{equation}
if $d_{Q^3}(\pi_{Q^3}(q_i), \pi_{Q^3}(\bx))=\iota_0$, where $\iota_0$ is half times the minimal $d_{Q^3}$ distance between two different points in $P \cup \{q_-, q_+\}$, the union of monopole points with the $\dZ_2$ fixed points.

\item
Given any sufficiently small parameter $0<\epsilon\ll T^{-\frac{1}{2}}$,
we slightly mollify the distance to the orbifold singularities  and obtain a smooth function as follows,
\begin{align}
\fd_*(\bx)
=\begin{cases}
\ee_{\lambda}^2 , & d_{g_{2\nu,T}}(q_\lambda,\bx) \leq  \ee_{\lambda}^2\\
& \text{for some}\ 1\leq \lambda\leq 4,
\\
&
\\
\min_{\lambda=1}^{4} d_{g_{2\nu,T}}(\bx,q_\lambda), &  d_{g_{2\nu,T}}(q_\lambda,\bx) \geq 2\ee_{\lambda}^2 \ \text{for all}\ 1\leq \lambda\leq 4 \ \text{and} \\
& d_{g_{2\nu,T}}(q_\lambda,\bx) \leq \frac{1}{4}\cdot \bar{\iota}_0 \cdot T^{\frac{1}{2}} \ \text{for some}\ 1\leq \lambda\leq 4,
\\
&
\\
T^{\frac{1}{2}}, &  d_{g_{2\nu,T}}(q_\lambda,\bm{x}) \geq  \frac{1}{2} \cdot \bar{\iota}_0 \cdot T^{\frac{1}{2}} \\
& \text{for all}\ 1\leq \lambda\leq 4.
\end{cases}
\end{align}

\end{enumerate}

With the above preparations, we are ready to give the
regularity scale analysis around
an $\I_{\nu}^*$-fiber in the following proposition.

\begin{proposition}
[Regularity scale and bubble limits near $\I_{\nu}^*$ fibers] \label{p:regularity-scale-I_v*}
Given any large parameter $T\gg1$ (or equivalently $\delta\ll1$), let $g_T = g_\delta^C$ be the approximately hyperk\"ahler metric  on $\cS_{\I_{\nu}^*}$, then the following holds:

\begin{enumerate}
\item

 Given $k\in\mathbb{N}$ and $\alpha\in(0,1)$,
there are uniform constants $\underline{v}_{k,\alpha}>0$ and $\bar{v}_{k,\alpha}>0$ independent of $\delta>0$ such that the $(k,\alpha)$-regularity scale $r_{k,\alpha}$ at $\bx\in \cS_{\I_{\nu}^*}$ has the following bound
\begin{equation}
\underline{v}_{k,\alpha}   \leq \frac{r_{k,\alpha}(\bx)}{\fs_{\nu}^*(\bx)}\leq \bar{v}_{k,\alpha},
\end{equation}
where the function $\fs_{\nu}^*(\bx)$ is explicitly given by
\begin{align}
\fs_{\nu}^*(\bx) =
\begin{cases}
   e^{-\frac{T}{2\nu}}\cdot  \fd_*(\bx), & d_{g_{T}}(q_\lambda,\bm{x}) \leq  \bar{\iota}_0  \cdot  e^{-\frac{T}{2\nu}}\cdot T^{\frac{1}{2}} \\
   & \text{for some}\ 1\leq \lambda\leq 4,
\\
&
\\
 e^{-\frac{T}{2\nu}}\cdot L_T(\bx)^{\frac{1}{2}}\cdot  \fr(\pi_{Q^3}(\bx)), & d_{g_{T}}(q_\lambda,\bm{x}) \geq  2\bar{\iota}_0  \cdot e^{-\frac{T}{2\nu}} \cdot T^{\frac{1}{2}}\\
 & \text{for all}\ 1\leq \lambda\leq 4.
\end{cases}
\end{align}

\item As $\delta\to0$, the canonical bubbles of $(\cK,\bx)$ for $\bx\in \cS_{\I_{\nu}^*}$ can be listed as follows:
\begin{enumerate}
\item[(a)]
the Ricci-flat Taub-NUT space $(\dC^2, g_{TN})$,
\item[(b)]  Euclidean space $\dR^3$,
\item[(c)]
ALE Eguchi-Hanson space $(X_{EH}^4,g_{EH})$,
\item[(d)]  flat orbifolds $\dR^4/\dZ_2$, $(\dR^3\times S^1)/\dZ_2$, $\dR^3/\dZ_2$, $(\dR^2\times S^1)/\dZ_2$ and  $\dR^2/\dZ_2$,
\item[(e)] the McLean metric $(\PP^1,d_{ML})$ with bounded diameter.
\end{enumerate}

\end{enumerate}

\end{proposition}
\begin{proof}The  proof is very similar to that of Proposition \ref{p:regularity-scale-I_v}.
Most of these bubbles also appeared in Proposition \ref{p:regularity-scale-I_v},
so  in the following proof, we will mainly focus on the cases where new type of bubbles limits occur.
Notice that, the singular region $\cS_{\I_{\nu}^*}$
can be subdivided into
the deepest ALF bubble region $\cS_{\I_{\nu}^*,1}$, the bubble damage zone region $\cS_{\I_{\nu}^*,2}$ and the large scale region $\cS_{\I_{\nu}^*,3}$. In addition, in our current case, the large scale region $\cS_{\I_{\nu}^*,3}$ will be further subdivided into two more pieces for analyzing the singularity behavior around the ALE bubbles.

Given a sequence $T_j\to+\infty$, let $g_j \equiv g_{T_j}$ be a sequence of approximately hyperk\"ahler metrics on $\cS_{\I_{\nu}^*}$ given by the gluing construction in Section \ref{ss:qmov}.
First, let us summarize the rescaled geometries as follows which is similar to the regularity and bubbling analysis around $\I_{2\nu}$-fibers:

\begin{flushleft}
{\bf Region $\cS_{\I_{\nu}^*,1}$ (deepest ALF bubbles):}
\end{flushleft}
This region consists of the points $\bx\in\cS_{\I_{\nu}^*}$ satisfying $d_{Q^3}(\underline{\bx},p_i) \leq T_j^{-1}$ for some $1\leq i\leq \nu$, where $\underline{\bx}=\pi_{Q^3}(\bx)\in Q^3$.
It was shown in the proof of Proposition \ref{p:regularity-scale-I_v} that, if the rescaled metric $\tilde{g}_j=\lambda_j^2 g_j$ is given by $\lambda_j= e^{\frac{T_j}{2\nu}}\cdot T_j^{\frac{1}{2}}$, then the canonical bubble limit
in this region is the Ricci-flat Taub-NUT space $(\dC^2,g_{TN},\bx_{\infty})$.
The regularity scale estimate just follows from the
non-flatness of $(\dC^2,g_{TN},\bx_{\infty})$
and the $C^{\infty}$-regularity of the convergence.

\begin{flushleft}
{\bf Region $\cS_{\I_{\nu}^*,2}$ (ALF bubble damage zone):}
\end{flushleft}
This region contains all the points $\bx\in\cS_{\I_{\nu}^*}$ satisfying $2T_j^{-1}\leq d_{Q^3}(\underline{\bx},p_i) \leq \iota_0$   for some $1\leq i\leq \nu$.
Let us choose the rescaled metric $\tilde{g}_j =\lambda_j^2 g_j $ with the rescaling factor
\begin{equation}
\lambda_j = e^{\frac{T_j}{2\nu}}\cdot T_j^{-\frac{1}{2}}\cdot \fr_j^{-1},
\end{equation}
 then the canonical bubbles arising in this region are listed as follows:
\begin{enumerate}
\item[(a)] If a  sequence of reference points $\bx_j$ satisfy that there is some $\sigma_0>0$ such that
\begin{equation}
2T_j^{-1}\leq \fr_j\leq \frac{1}{(\sigma_0)^2} \cdot T_j^{-1} ,
\end{equation}
 then the canonical bubble is a
Ricci-flat Taub-NUT space $(\dC^2,g_{TN},\bx_{\infty})$,

\item[(b)] If a sequence of reference points $\bx_j$ satisfy the condition
\begin{equation}
\frac{\fr_j}{T_j^{-1}}\to +\infty\ \text{and}\ \fr_j\to0,
\end{equation}
then the canonical bubble is
the Euclidean space $(\dR^3,g_{\dR^3},\bx_{\infty})$ such that the limiting reference point $\bx_{\infty}$ satisfies $d_{g_{\dR^3}}(0^3,\bx_{\infty})=1$.
\item[(c)] Now we consider the complement of the Case (a) and Case (b). That is, for any sequence of reference points $\bx_j$, there is some $r_0>0$ such that
$r_0 \leq \fr_j \leq \iota_0$. In the case,
the canonical bubble is
the flat orbifold $((\dR^2\times S^1)/\dZ_2,g_{\FF},\bx_{\infty})$, where $g_{\FF}$ is the $\dZ_2$-quotient of the standard flat product metric on $\dR^2\times S^1$.
\end{enumerate}
The remaining computations coincide with those in the proof of Proposition \ref{p:regularity-scale-I_v}.

\begin{flushleft}
{\bf Region $\cS_{\I_{\nu}^*,3}$ (large scale regions):}
\end{flushleft}
In comparison with the large scale region $\cS_{\I_{\nu}^*,3}$ in the proof of Proposition \ref{p:regularity-scale-I_v}, we need two additional pieces to characterize the rescaling geometry around the ALE bubbles, which are removed from this region for the moment.  That is,
Region $\cS_{\I_{\nu}^*,3}$ contains all the points $\bx\in \cS_{\I_{\nu}^*}$ determined by the following conditions
\begin{align}
d_{g_T}(\bx,q_\lambda) & \geq 2\ee_{\lambda}\cdot e^{-\frac{T}{2\nu}}, \  \forall \lambda,  1\leq \lambda\leq 4,
\\
d_{Q^3}(\underline{\bx},p_i)  & \geq 2\iota_0, \  \forall i, 1\leq i\leq \nu,
\\
L_T(\bx) & \geq  2.
\end{align}

To analyze the regularity scales in Region $\cS_{\I_{\nu}^*,3}$, we will choose two different rescaling metrics depending on if the reference points $\bx_j$ are close to the orbifold resolution loci $\{q_1,\ldots, q_4\}$.

\begin{flushleft}
{\bf Case (A):}
\end{flushleft}
First, we consider the case that
a sequence of reference points $\bx_j\in \cS_{\I_{\nu}^*,3}$ far from any point in $\{q_1,q_2,q_3, q_4\}$, that is,  $\bx_j$ satisfies \begin{equation}d_{g_j}(\bx_j,q_\lambda)\geq 2\bar{\iota}_0\cdot e^{-\frac{T}{2\nu}}\cdot T_j^{\frac{1}{2}}\
\text{for all}\  1\leq \lambda\leq 4.\end{equation}
The arguments in this case are similar to the  case of $\I_{\nu}$-fibers  and  the rescaled metrics $\tilde{g}_j=\lambda_j^2 \cdot g_j$
are determined by
\begin{equation}
\lambda_j = e^{-\frac{T_j}{2\nu}}\cdot L_j^{-\frac{1}{2}}\cdot \fr_j^{-1}.
\end{equation}
The canonical bubbles far from $q_m$ can be classified as follows:
\begin{enumerate}
\item If $\bx_j$ satisfies $\frac{1}{2}\leq \fr_j\leq R_0$, then the canonical bubble is
the flat Riemannian orbifold $((\dR^2\times S^1)/\dZ_2,g_{\FF},\bx_{\infty})$ such that the distance from the limiting reference point $\bx_{\infty}$ to the orbifold singularity equals $1$.

 \item If $\bx_j$ satisfies $\fr_j\to+\infty$ and $L_j\to +\infty$, then the  canonical bubble is
the  flat cone $(\dR^2/\dZ_2, g_{\FF}, \bx_{\infty})$ such that the distance from $\bx_{\infty}$ to the cone vertex is equal to $1$.
\item If  $\bx_j$ satisfies $2\leq L_j \leq T_0$ for some positive constant $T_0>0$ independent of $T$, then the canonical bubble is $(\PP^1, d_{ML},\bx_{\infty})$ such that $\bx_{\infty}\in \PP^1\setminus \mathcal{S}$ and $\mathcal{S}$ is the finite singular set in $\PP^1$.
\end{enumerate}
The remaining computations are exactly same as the proof of Proposition \ref{p:regularity-scale-I_v} and hence we omit the details. To see the regularity scale estimates at $\bx_j$, we can observe that $\Gamma_{\epsilon_0}(\bx_j)$ has maximal rank in each of the above cases, where $\epsilon_0$ is the constant in Theorem \ref{t:maximal-nilpotent-rank}. So Theorem \ref{t:maximal-nilpotent-rank} gives uniform curvature estimate
at $\bx_j$ and hence $r_{k,\alpha}(\bx_j)\geq \underline{v}_{k,\alpha}>0$ uniformly bounded from below. On the other hand, around the orbifold singularities, one can use the curvature blowing-up behavior of the Eguchi-Hanson space and the Lipschitz property of $r_{k,\alpha}$, which guarantees that
\begin{align}
r_{k,\alpha}(\bx_j)\leq 2.
\end{align}
  Rescaling back to the original metrics $g_j$, the desired regularity scale estimate follows immediately.

\begin{flushleft}
{\bf Case (B):}
\end{flushleft}
Next, we will handle the case
that a sequence of reference points $\bx_j$ satisfy
 \begin{equation}
 d_{g_j}(\bx_j, q_\lambda) \geq 2\ee_{\lambda}\cdot e^{-\frac{T_j}{2\nu}}
 \end{equation}
 for all $1\leq \lambda\leq 4$ and
 \begin{equation}
 d_{g_j}(\bx_j, q_\lambda) \leq \bar{\iota}_0\cdot e^{-\frac{T_j}{2\nu}}\cdot T^{\frac{1}{2}}
 \end{equation}
 for some $1\leq \lambda\leq 4$.
In this case, we will use another rescaling to see additional bubbles around $q_\lambda$.   Let
$\tilde{g}_j = \lambda_j^2 g_j$ be the rescaled metrics given by
\begin{equation}
\lambda_j = d_{g_j}(\bx_j,q_\lambda)^{-1},
\end{equation}
then rescaling geometries can be further subdivided into the following cases:
\begin{enumerate}
\item If $\bx_j$ are chosen such that
\begin{equation}
\frac{e^{\frac{T_j}{2\nu}}\cdot  d_{g_j}(\bx_j,q_\lambda)}{T_j^{-\frac{1}{2}}} \to 0 \ \text{for some}\ 1\leq \lambda\leq 4,
\end{equation}
then the rescaled limit is  the asymptotic cone of the Eugchi-Hanson
space  which is
isometric to the flat cone
$(\dR^4/\dZ_2,g_{\FF},\bx_{\infty})$ such that the distance from $\bx_{\infty}$ to the cone tip is equal to $1$.

\item
This case contains the points
 $\bx_j$ which satisfy the condition that there is some constant $\sigma_0>0$ such that
\begin{equation}
\frac{e^{\frac{T_j}{2\nu}}\cdot  d_{g_j}(\bx_j,q_\lambda)}{T_j^{-\frac{1}{2}}} \geq \frac{1}{\sigma_0}
\end{equation} for all $1\leq \lambda\leq 4$ and
 \begin{equation}
\frac{e^{\frac{T_j}{2\nu}}\cdot  d_{g_j}(\bx_j,q_\lambda)}{T_j^{-\frac{1}{2}}} \leq \sigma_0
\end{equation}
for some $1\leq \lambda\leq 4$. The canonical bubble in this case becomes the flat orbifold $((\dR^3\times S^1)/\dZ_2,g_{\FF},\bx_{\infty})$

\item In this case, $\bx_j$ satisfies the condition that there is some $c_0>0$ such that
\begin{equation}
\frac{e^{\frac{T_j}{2\nu}}\cdot  d_{g_j}(\bx_j,q_\lambda)}{T_j^{-\frac{1}{2}}} \to +\infty
\end{equation}
for all $1\leq \lambda\leq 4$ and
\begin{equation}
\frac{e^{\frac{T_j}{2\nu}}\cdot  d_{g_j}(\bx_j,q_\lambda)}{T_j^{\frac{1}{2}}} \to 0
\end{equation}
for some $1\leq \lambda\leq 4$.
The canonical rescaled limit is isometric to the flat orbifold $(\dR^3/\dZ_2,g_{\FF},\bx_{\infty})$ and the orbifold metric $g_{\FF}$ is the $\dZ_2$-quotient of the Euclidean metric $g_{\dR^3}$, where  the finite group $\dZ_2\leq \Isom(\dR^3)$
acts on $\dR^2$ by rotation with angle $\pi$ and acts on $\dR$ by reflection so that there is only one orbifold singular point $0^3$. Moreover, $d_{g_{\FF}}(\bx_{\infty},0^3)=1$.

\item In this case, $\bx_j$ satisfies the condition that there is some $0<c_0<\frac{\bar{\iota}_0}{10}$ such that
\begin{equation}
d_{g_j}(\bx_j, q_\lambda) \geq c_0\cdot  e^{-\frac{T_j}{2\nu}} \cdot T_j^{\frac{1}{2}}
\end{equation}
for all $1\leq \lambda\leq 4$ and
\begin{equation}
d_{g_j}(\bx_j, q_\lambda) \leq \bar{\iota}_0\cdot e^{-\frac{T_j}{2\nu}}\cdot T_j^{\frac{1}{2}}.
\end{equation}
for some $1\leq \lambda\leq 4$.
Then it is straightforward to check that the rescaled limit is
 $((\dR^2\times S^1)/\dZ_2,g_{\FF},\bx_{\infty})$
with $d_{g_{\FF}}(\bx_{\infty},0^3)=1$.
\end{enumerate}
The regularity scale estimates just follow from the same arguments as before, hence
the bubble classification in Region $\cS_{\I_{\nu}^*,3}$ is done.

Next, we will finish the proof by computing  the regularity scale around the four ALE gluing loci inside Region  $\cS_{\I_{\nu}^*,3}$.

\begin{flushleft}
{\bf Region $\cS_{\I_{\nu}^*,4}$ (deepest ALE bubbles):}
\end{flushleft}
This region contain the points $\bx$ satisfying
$d_{g_\delta^C}(\bx,q_\lambda)\leq \ee_{\lambda}^2\cdot e^{-\frac{T_j}{2\nu}}$ for some $1\leq \lambda\leq 4$.
Now take a sequence of reference points $\bx_j$ in this region.
So it straightforward follows from the construction of the gluing region  that if we rescale the metric $g_j$ by
\begin{equation}
\tilde{g}_j \equiv \ee_{\lambda}^{-4} \cdot e^{\frac{T_j}{\nu}} \cdot g_j,
\end{equation}
 then the rescaled limit is isometric to a Ricci-flat Eguchi-Hanson space  $(X_{EH}^4, g_{EH},\bx_{\infty})$.

\begin{flushleft}
{\bf Region $\cS_{\I_{\nu}^*,5}$ (ALE bubble damage zone):}
\end{flushleft}

In this region, all the points $\bx$ satisfy
$ 2 \ee_{\lambda}^2\cdot e^{-\frac{T_j}{2\nu}}\leq d_{g_\delta^C}(\bx,q_\lambda)\leq \ee_{\lambda}\cdot e^{-\frac{T_j}{2\nu}}$ for some $1\leq \lambda\leq 4$. Now the metrics $g_j$ will be rescaled by
\begin{align}
\tilde{g}_j & \equiv \lambda_j^{-2} \cdot g_j,
\\
\lambda_j & \equiv d_{g_j}(\bx_j,q_\lambda).
\end{align}
 This region can be decomposed into two pieces:
\begin{enumerate}
\item If $\bx_j$ are chosen such that there is some $\sigma_0>0$ such that
\begin{equation}
2\ee_{\lambda}^2\cdot e^{-\frac{T_j}{2\nu}}\leq d_{g_\delta^C}(\bx,q_\lambda)\leq \sigma_0\cdot \ee_{\lambda}^2\cdot e^{-\frac{T_j}{2\nu}},
\end{equation}
then the rescaled limit is isometric to an Eguchi-Hanson space  $(X_{EH}^4, g_{EH},\bx_{\infty})$.

\item  If $\bx_j$ satisfies
\begin{equation}
\frac{d_{g_\delta^C}(\bx,q_\lambda)}{\ee_{\lambda}^2\cdot e^{-\frac{T_j}{2\nu}}} \to +\infty  \ \text{and}\ d_{g_\delta^C}(\bx,q_\lambda)\leq \ee_{\lambda}\cdot e^{-\frac{T_j}{2\nu}},
\end{equation}
then the rescaled limit is the asymptotic cone of the Eugchi-Hanson
space which is isometric to the flat cone $(\dR^4/\dZ_2,g_{\FF},\bx_{\infty})$
with $d_{g_{\FF}}(\bx_{\infty}, 0^*)=1$,
where $0^*$ is the cone tip of $\dR^4/\dZ_2$.
\end{enumerate}

This completes the proof of the proposition.
\end{proof}

\begin{remark}
\label{d2remark}
Recall that the scale parameter $\ee_{\lambda}$ in the definition of $\mathfrak{d}_*(\bx)$ satisfies the constraint (see \eqref{e:orbifold-parameter})
\begin{equation}
\ee_{\lambda} \leq \eta_0\cdot T^{-\frac{1}{2}},
\end{equation}
where $\eta_0 > 0$ is sufficiently small but independent of $T$. For simplicity, in our paper, $\ee_{\lambda}$ is further required to satisfy $\ee_{\lambda}/T^{-\frac{1}{2}} \to 0$ as $T\to+\infty$ which will give four copies of Eguchi-Hanson bubbles. However, our arguments still hold if
\begin{equation}\underline{\eta}_0\leq \ee_{\lambda}/T^{-\frac{1}{2}} \leq  \eta_0\end{equation} for some definite constant $\underline{\eta}_0>0$ as $T\to+\infty$. It is possible to show that in this case, two copies of ALF-$D_2$ spaces will appear as bubble limits. See Section \ref{ss:more-bubbles} for more discussions corresponding to the second case.
\end{remark}

\subsection{Weighted H\"older spaces}

\label{ss:weighted-elliptic-regularity}

In this subsection, we define weighted H\"older space using the regularity scale function. In general, the weighted H\"older space can be defined in the following way.

\begin{definition}[Weighted H\"older spaces]
\label{d:weighted-space}
Let $(M,g)$ be a Riemannian manifold. Let $\rho>0$ be a smooth function. Let $K$ be any subset of $M$, then for any $k\in\mathbb{N}$ and $\alpha\in(0,1)$, with respect to the function $\rho$, the weighted H\"older norm of a tensor field $\chi\in T^{p,q}(M)$ on $K$ is defined as follows:
\begin{equation}
\|\chi\|_{C_{\mu}^{k,\alpha}(K)}
\equiv \sum\limits_{m= 0}^{k}\Big\|\rho^{\mu+m}\cdot \nabla^m\chi\Big\|_{C^0(K)} + [\chi]_{C_{\mu}^{k,\alpha}(K)},
\end{equation}
and
\begin{equation}
[\chi]_{C_{\mu}^{k,\alpha}(K)} \equiv \sup_{\substack{\gamma\ \text{is a geodesic}\\ \gamma(0), \gamma(1)\in K}} \Big\{\min\{\rho^{\mu+ k +\alpha}(\gamma(0)),\rho^{\mu+ k +\alpha}(\gamma(1))\}\cdot\frac{|\nabla^k\chi(\gamma(0))-\nabla^k\chi(\gamma(1))|}{L(\gamma)^{\alpha}}\Big\},
\end{equation}
where $L(\gamma)$ means the length of $\gamma$ and the difference of the two covariant derivatives is defined in terms of the parallel translation along~$\gamma$.
\end{definition}

We will choose the function $\rho$ on $\cK$ as the global smooth function $\fs(\bx)$ defined as follows,
\begin{align}
\label{d:weight}
\fs(\bx)
=
\begin{cases}
\fs_{\cG}(\bx), & \bx \in \cS_{\ALG},
\\
\fs_{\nu}(\bx), & \bx \in \cS_{\I_{\nu}},
\\
\fs_{\nu}^*(\bx), & \bx \in \cS_{\I_{\nu}^*},
\\
\fs_{\cR}(\bx), & \bx \in \cR_\delta.
\end{cases}
\end{align}
The previous subsections give explicit estimates for the $C^{k,\alpha}$-regularity scales $r_{k,\alpha}(\bx)$ for all points $\bx$ in $\cK$. We have proved that there are uniform constants $\underline{v}_{k,\alpha}>0$ and $\bar{v}_{k,\alpha}>0$ independent of $\delta>0$ such that
\begin{equation}
\label{e:regularity-scale-estimates}
\underline{v}_{k,\alpha}   \leq \frac{r_{k,\alpha}(\bx)}{\fs(\bx)}\leq \bar{v}_{k,\alpha}.
\end{equation}

We next prove a uniform weighted Schauder estimate for the elliptic operator
\begin{align}
d^+_{g_{\delta}^C} :  \mathring{\Omega}^1(\cK) \rightarrow \Omega^2_+(\cK)
\end{align}
defined by
\begin{equation}
d^+_{g_{\delta}^C} \eta \equiv \frac{1}{2}(d\eta + *_{g_{\delta}^C} d\eta)
\end{equation}
with
\begin{align}
\mathring{\Omega}^1(\cK)\equiv\{\eta\in\Omega^1(\cK)| d^*_{g_{\delta}^C}\eta =0\},
\end{align}
which will be used in the weighted analysis in Section \ref{s:proof-of-main-theorem}.

\begin{proposition}[Weighted Schauder estimate]
\label{p:weighted-schauder}
For any sufficiently small positive parameter $\delta\ll1$, let $g_\delta^C$ be the approximately hyperk\"ahler metric on $\cK$.
Then there exists $C>0$  independent of $\delta>0$ such that for any $\eta\in\mathring{\Omega}^1(\cK)$ it holds that
	\begin{equation}
		\|\eta\|_{C_{\mu}^{1,\alpha}(\cK)} \leq C(\| d^+_{g_\delta^C} \eta\|_{C_{\mu+1}^{0,\alpha}(\cK)} + \| \eta \|_{C_{\mu}^{0}(\cK)}).\label{e:global-weighted-schauder}
	\end{equation}
\end{proposition}

\begin{proof}
The proof follows from the explicit expression of the $C^{k,\alpha}$-regularity scales in the above discussions.
Under the canonical rescaling,
\begin{align}
\tilde{g}_{\delta}^C= r_{k,\alpha}(\bx)^{-2}\cdot  g_\delta^C,\end{align}
the geodesic ball $B_{1/2}(\bm{x})$ under the rescaled metric $g_{\delta}^C$ has uniformly bounded $C^{k,\alpha}$-geometry (independent of $\delta$) for each $\alpha\in(0,1)$ and $k\in\{0,1\}$. So there is a uniform constant $C>0$ (independent of $\delta$) such that the standard Schauder estimate holds for every $\eta\in\mathring{\Omega}^1(\cK)$ and $\bm{x}\in \cK$ under the rescaled metric $\tilde{g}_{\delta}^C$, \begin{equation}
\|\eta\|_{C^{k+1,\alpha}(B_{1/4}(\bm{x}))} \leq C\Big(\| d^+_{\tilde{g}_{\delta}^C}\eta \|_{C^{k,\alpha}(B_{1/2}(\bm{x}))} +\|\eta\|_{C^{0}(B_{1/2}(\bm{x}))}\Big).\label{e:ball-standard-schauder}
\end{equation}
Then the global weighted Schauder estimate \eqref{e:global-weighted-schauder} just follows from rather standard rescaling arguments. In fact, the only crucial point is to verify that for every $\bm{x}\in \cK$,
 the function $\fs$ is roughly a constant in the ball $B_{\frac{1}{4}r_{k,\alpha}(\bx)}(\bm{x})$, in the sense that there is a uniform constant $C>0$ such that for any $\by\in B_{\frac{1}{4}r_{k,\alpha}(\bx)}(\bm{x})$,
 \begin{equation}
C^{-1}\cdot  \fs(\bx)\leq  \fs(\by) \leq  C\cdot \fs(\bx).\label{e:weight-function-control}
 \end{equation}
The verification of (\ref{e:weight-function-control}) follows from (\ref{e:regularity-scale-estimates}) and Remark \ref{r:1-Lipchitz}. The remainder of the proof is almost identical to those given in  \cite[Section~8]{HSVZ} or \cite[Section~4]{SZ}.

\end{proof}

\section{Liouville theorems}
\label{s:Liouville}

The proof of Theorem \ref{t:main-theorem} will require
vanishing theorems on various bubble limits, which we refer to as Liouville theorems. In Subsection~\ref{ss:Liouville-noncompact}, we will consider the case
of non-compact bubble limits. Then in Subection~\ref{ss:Liouville-compact},  we will consider the case of a compact limit.

\subsection{Liouville type theorems on non-compact spaces}

\label{ss:Liouville-noncompact}

First, we will prove a Liouville type theorem
on flat sectors in $\dR^2$.
This vanishing result will be used in proving the uniform injectivity estimate as the contradiction sequence concentrates around the ALG bubbles. In addition, it will also be used to prove Proposition~\ref{l2thm}.

To begin with, we fix some notation. Given the polar coordinate $x=(r,\theta)\in\dR^2$ and given a real number $\beta\in(0,1]$,
let us denote by
\begin{align}
\Sec(\beta)\equiv\Big\{(r,\theta)\in\dR^2 \ \Big|  \ r\in(0,+\infty),\  \theta\in(0,2\pi\beta)\Big\},
\end{align}
a flat open sector and its closure by $\overline{\Sec(\beta)}$. We also denote by $\punsec(\beta)\equiv \overline{\Sec(\beta)}\setminus\{0^2\}$ the punctured sector.
In terms of polar coordinates in $\dR^2$, the Euclidean metric on $\overline{\Sec(\beta)}$ is
\begin{equation}
g_0 = dr^2 + r^2 d\theta^2, \ \theta\in(0,2\pi\beta).
\end{equation}

Given an angle parameter $\beta\in(0,1]$, the flat sector $(\overline{\Sec(\beta)},g_0)$, as a warped product, has a compact cross section $[0,2\pi\beta]$ at $r=1$.
The lemma below gives the spectrum and the associated Fourier expansion with prescribed boundary condition.

\begin{lemma}[Fourier series on flat sectors]\label{l:Fourier-expansion}
Given $\beta\in (0,1]$ and $\sigma \in\dR$, let $\{\varphi_j(\theta)\}_{j\in\dZ}$ be a complete orthonormal basis of  $L^2([0,2\pi\beta])$ which solves
\begin{align}
-\varphi_j''(\theta) = \Lambda_j \cdot\varphi_j(\theta) \label{e:eigenfunction-segment}
\end{align}
with the boundary condition
\begin{align}
\varphi_j(2\pi\beta) = e^{\sq\cdot 2\pi\sigma}\cdot
\varphi_j(0).\label{e:boundary-condition-segment}
\end{align}
Then the following holds:
\begin{enumerate}
\item For each $j\in\dZ$,
$\varphi_j(\theta)  = e^{-\sq\cdot \lambda_j \theta}$ with
\begin{equation}
\label{ljdef}
\lambda_j  =  \frac{j-\sigma}{\beta}
\end{equation}
and $\Lambda_j=\lambda_j^2$.

\item Let $U(r,\theta)$ be a complex-valued $C^{\infty}$-function on $\punsec(\beta)$ satisfying
\begin{equation}
U(r,2\pi\beta)=e^{\sq\cdot 2\pi\sigma}\cdot U(r,0),\quad r>0,
\end{equation}
 then the Fourier series
\begin{equation}
\sum\limits_{j\in\dZ} U_j(r)\cdot \varphi_j(\theta)\label{e:Fourier-expansion-U}
\end{equation}
converges to $U(r,\theta)$ in the $C^{\infty}$-topology in any compact subset of $\punsec(\beta)$, where
\begin{equation}
U_j(r) \equiv \frac{1}{2\pi\beta} \int_0^{2\pi\beta}U(r,\theta)\cdot \overline{\varphi_j(\theta)}d\theta.
\end{equation}

\end{enumerate}

\end{lemma}

\begin{proof}
First, we compute the eigenvalues.
For any $j\in\dZ$, let us take $\varphi_j(\theta)=e^{-\sq \cdot \lambda_j \cdot \theta}$, then $\varphi_j(\theta)$ is a solution to \eqref{e:eigenfunction-segment} when we choose $\Lambda_j$ as $\lambda_j^2$.
Now the boundary condition
\eqref{e:boundary-condition-segment} gives rise to the relation
\begin{equation}
2\pi(\lambda_j \cdot \beta + \sigma) = j\cdot 2\pi, \quad j\in\dZ,
\end{equation}
and \eqref{ljdef} follows from this.
Moreover, $\{\varphi_j\}_{j\in\dZ}$ is an orthogonal basis of $L^2([0,2\pi\beta])$ with
\begin{equation}
\frac{1}{2\pi\beta}\int_0^{2\pi\beta}\varphi_j(\theta)\cdot \overline{\varphi_k(\theta)}d\theta = \delta_{jk}.
\end{equation}
The convergence result in Item (2) then follows from standard Fourier theory.
\end{proof}

\begin{proposition}[Liouville theorem on flat sectors]\label{p:Liouville-sector}
Given $\beta\in(0,1]$ and $\sigma\in\dR$,
let us define a positive number $\iota_{\beta,\sigma}$ by
\begin{align}
\iota_{\beta,\sigma} \equiv
\begin{cases}
1/\beta,& \sigma\in\dZ,
\\
(\sigma-[\sigma])/\beta, & \sigma-[\sigma]\in(0, \frac{1}{2}],
\\
([\sigma]+1-\sigma)/\beta, & \sigma-[\sigma] \in (\frac{1}{2},1).
\end{cases}
\end{align}
Let $U(r,\theta) \equiv f(r,\theta)+\sq \cdot h(r,\theta)$ be a
complex-valued function on the punctured sector $\punsec(\beta)\subset \dC$ such that both $f(r,\theta)$ and  $h(r,\theta)$
are real-valued harmonic functions on
$\punsec(\beta)$.
Also assume that $U$ satisfies the boundary condition \begin{align}
U(r,\beta) = e^{\sq\cdot 2\pi\sigma}\cdot U(r,0),  \ \forall r>0.
\end{align}
Then the following properties hold:
\begin{enumerate}
\item With $\lambda_j$
and $\varphi_j$ as in Lemma~\ref{l:Fourier-expansion}, $U$ has the following types of $C^{\infty}$-converging Fourier expansions
\begin{align}
  U(r,\theta) &= \kappa_0  + c_0\cdot  \log r + \sum\limits_{j\neq \sigma} (C_j \cdot r^{\lambda_j} + C_j^* \cdot r^{-\lambda_j})\cdot \varphi_j(\theta),  \  \sigma\in\dZ,
   \\
     U(r,\theta) &= \sum\limits_{j\in \dZ} (C_j \cdot r^{\lambda_j} +  C_j^* \cdot r^{-\lambda_j})\cdot \varphi_j(\theta) , \  \sigma\not\in\dZ.
\end{align}

\item
If there is some
$\mu\in(0,\iota_{\beta,\sigma})$ such that for any $x\in \punsec(\beta)$,
\begin{equation}
|U(x)|\leq \frac{C}{r(x)^{\mu}}, \label{e:sector-growth}\end{equation}
 then $u= 0$ on the whole closed sector $\overline{\Sec(\beta)}$.
 \end{enumerate}

\end{proposition}

\begin{proof}
First, we prove Item (1).
We will compute the expansion of $U$ by using separation of variables and Fourier series in Lemma~\ref{l:Fourier-expansion} along the cross section $[0,2\pi\beta]$.
The Euclidean Laplacian $\Delta_0$  in terms of polar coordinates $x=(r,\theta)\in \dC$ with $\theta\in[0,2\pi\beta]$ is given by
\begin{equation}
\Delta_0 U = \frac{\partial^2 U}{\partial r^2} + \frac{1}{r}\cdot\frac{\partial U}{\partial r} + \frac{1}{r^2}\cdot \frac{\p^2 U}{\p\theta^2}.
\end{equation}
Since the harmonic function $U$ is $C^{\infty}$ in  the punctured sector $\punsec(\beta)$,
\eqref{e:Fourier-expansion-U} gives a $C^{\infty}$-converging Fourier series of $U$,
\begin{equation}
U(r,\theta) = \sum\limits_{j\in\dZ} U_j(r)\varphi_j(\theta).
\end{equation}
Plugging the above expansion into $\Delta_0 U=0$, we obtain the Cauchy-Euler equation for $U_j(r)$,
\begin{equation}
U_j''(r) +\frac{U_j'(r)}{r}-\frac{\lambda_j^2\cdot U_j(r)}{r^2} = 0,\label{e:CE-eq}
\end{equation}
where $\lambda_j = (j-\sigma)/\beta$.
The indicial equation of \eqref{e:CE-eq} in $m_j$ is given by
\begin{equation}
m_j^2 - \lambda_j^2 = 0, \ j\in\dZ,\label{e:indicial-equation}
\end{equation}
which gives the indicial roots
\begin{equation}
m_j = \pm\lambda_j,\quad j\in\dZ.
\end{equation}
So there are two types of expansions depending on the indicial roots.

First, if $\sigma\in\dZ$, then $m_j=\lambda_j$
vanishes at $j=\sigma$. In this case,
\begin{eqnarray}
U_j(r)
=
\begin{cases}
\kappa_0 + c_0 \log r, & j=\sigma,
\\
C_j \cdot r^{\lambda_j} + C_j^*\cdot r^{-\lambda_j}, & j\neq \sigma.
\end{cases}
\end{eqnarray}
Therefore,
\begin{equation}
U(r,\theta) = \kappa_0  + c_0\cdot  \log r + \sum\limits_{j\neq \sigma} (C_j \cdot r^{\lambda_j} + C_j^* \cdot r^{-\lambda_j})\cdot \varphi_j(\theta).
\end{equation}
Next, we consider the case $\sigma\not\in\dZ$, which gives the solutions
\begin{equation}
U_j(r) = C_j \cdot r^{\lambda_j} + C_j^*\cdot r^{-\lambda_j}, \quad j\in\dZ.
\end{equation}
Hence, $U$ has the expansion
\begin{equation}
     U(r,\theta) = \sum\limits_{j\in \dZ} (C_j \cdot r^{\lambda_j} +  C_j^* \cdot r^{-\lambda_j})\cdot \varphi_j(\theta).
\end{equation}
This completes the proof of Item (1).

We will next prove Item (2). In accordance with the Fourier expansions obtained in Item (1), there are two cases to analyze.

First, if $\sigma\in\dZ$, then the minimal nonzero indicial root is
\begin{equation}
m_1 =\lambda_1 = \pm 1/\beta.
\end{equation}
If there is some $\mu\in (0,1/\beta)$ such that $U$ satisfies the growth condition \begin{equation}|U(x)|\leq \frac{C}{r(x)^{\mu}},\quad \forall x\in\punsec(\beta),\end{equation}
then $c_0=\kappa_0=0$ and $C_j=C_j^*=0$ for all $j\in\dZ$. Therefore $U= 0$.

Second , we consider the case $\sigma\not\in\dZ$, so we have $|\lambda_j|>0$ for each $j\in\dZ$.
The minimal nonzero $|\lambda_j|$ among $j\in\dZ$ is achieved at either $j=  [\sigma]$ or $j=1+[\sigma]$, which gives
 \begin{align}
\min\limits_{j\in\dZ}
|\lambda_j|  =
 \begin{cases}
 (\sigma-[\sigma])/\beta, & \sigma-[\sigma]\in(0, \frac{1}{2}],
\\
([\sigma]+1-\sigma)/\beta, & \sigma-[\sigma] \in (\frac{1}{2},1).
 \end{cases}
\end{align}
Finally,  the growth condition on $U$ in this case reads
\begin{equation}
|U(r,\theta)|\leq \frac{C}{r^{\mu}},\quad \forall r>0
\end{equation}
with $\mu<\iota_{\beta,\sigma}=\min\limits_{j\in\dZ}|\lambda_j|$, which immediately
implies $C_j=C_j^*=0$ for all $j\in\dZ$.
So we have proved that $U$ vanishes on $\overline{\Sec(\beta)}$ in Case (2).

This completes the proof of the proposition.

\end{proof}

Since $\overline{\Sec(1)} = \dR^2$, Proposition~\ref{p:Liouville-sector} immediately implies the following corollary.

\begin{corollary}
\label{c:Liouville-plane}
Let $u$ be a real harmonic function on $\dR^2 \setminus \{0^2\}$. Assume that there is some $\mu\in(0,1)$
such that $u$ satisfies the following growth condition
\begin{equation}
|u(x)| \leq \frac{C}{|x|^{\mu}}, \quad \forall  x\in\dR^2\setminus\{0^2\},
\end{equation}
then $u \equiv 0$ on $\dR^2$.
\end{corollary}

The following removable singularity result involves
the harmonic functions with slower growth rate than the Green's function, which is rather standard in the literature, see for example \cite{GilbargTrudinger}.
\begin{lemma}
\label{l:removable-singularity}
Let $(M^n,g)$ be a Riemannian manifold with $n\geq 3$. Given a point $p\in M^n$, assume that
$u$ is a harmonic function on $B_r(p)\setminus\{p\}$
which satisfies
\begin{align}
|u(x)|\leq
\frac{C}{d_g(x,p_i)^{\mu}} \quad \text{for every}\ x\in B_{s_0}(p_i)\setminus\{p_i\},
\end{align}
for some $s_0\in(0,r)$ and $\mu\in(0,n-2)$, then $u$ extends to a harmonic function on $B_{r}(p)$.

\end{lemma}

As a quick corollary of the above removable singularity lemma, we have the following Liouville type result for harmonic functions,  the proof of which follows easily from the maximum principle.

\begin{corollary}
\label{c:standard-Liouville-noncompact}
Let $(M^n,g)$ be a complete non-compact Riemannian manifold. Given a finite subset $\cS\subset M^n$ Assume that $u$ is a harmonic function on $M^n\setminus \cS$ satisfying
\begin{align}
\begin{cases}
|u(x)| \leq \frac{C}{d_g(x,p_i)^{\mu}} , &  x\in B_{s_0}(p_i)\setminus\{p_i\},
\\
|u(x)|\to 0, & d_g(x,\cS)\to+\infty,
\end{cases}
\end{align}
for some $s_0>0$ and $\mu\in(0,n-2)$, then $u\equiv 0$ on $M^n$.
\end{corollary}

We finish the discussion of this subsection by the following vanishing result for harmonic $1$-forms.

\begin{lemma}\label{l:harmonic-1-form-Liouville}
Let $(M^n,g, p)$ be a complete non-compact Riemannian manifold with $\Ric_g \geq 0$. If $\eta$ is a harmonic $1$-form on $M^n$ such that $|\eta(x)|\to0$ as $d_g(p,x)\to\infty$, then $\eta = 0$.
\end{lemma}

\begin{proof}
The proof is routine. Applying Bochner's formula for the harmonic $1$-form $\eta$,
\begin{equation}
\frac{1}{2}\Delta_g|\eta|^2 = |\nabla \eta|^2 + \Ric_g(\eta,\eta) \geq 0,
\end{equation}
then $|\eta|^2$ is subharmonic.
Since  $\lim\limits_{d_g(p,x)\to\infty}|\eta(x)|=0$, so the maximum principle implies that $\eta = 0$ on $M^n$.
\end{proof}

\subsection{A Liouville type theorem on the $\PP^1$ limits}
\label{ss:Liouville-compact}

In this subsection, we prove a Liouville type theorem
in the context that a sequence of approximate metrics $g_j^C$ on the elliptic K3 surface $\cK$ Gromov-Hausdorff converging to the compact limit $(\PP^1, d_{ML})$
\begin{align}
(\cK, g_j^C)\xrightarrow{GH}(\PP^1, d_{ML}),
\end{align}
 where $d_{ML}$ is the McLean metric on $\PP^1$ with bounded diameter and non-smooth along a finite singular set $\cS\subset \PP^1$.

Consider a sequence of co-closed real $1$-forms $\eta_j\in\mathring{\Omega}^1(\cK)$ and we will identify $\eta_j$ with a $4$-tuple of functions away from the singular fibers.
Following the notation in Section~\ref{s:semi-flat-metrics}, in terms of the holomorphic coordinates $y$ and $x_{\delta_j}$ given by the holomorphic section,
\begin{equation}\eta_j=f_j^{(y)}dy+f_j^{(\bar{y})}d\bar y+ f_j^{(3)}\cdot e^3 +f_j^{(4)}\cdot e^4,\end{equation}
where $e^3$ and $e^4$ as chosen in Subsection~\ref{ss:rescaling-semi-flat},
such that $\overline{f_j^{(y)}}=f_j^{(\bar{y})}$ and $f_j^{(3)}$, $f_j^{(4)}$ are real. Therefore, the $1$-form $\eta_j$ on $\cR_{\delta_j}$ is identified with the functions $(f_j^{(y)}, f_j^{(\bar{y})}, f_j^{(3)}, f_j^{(4)})$.
Based on the above discussions, the convergence of $1$-forms can be converted to the convergence of functions on $\cR_{\delta_j}$.
We also remark that as $\delta_j\to 0$, \begin{equation}
(\cR_{\delta_j},g_j^C)\xrightarrow{GH} (\PP^1\setminus\mathcal{S},d_{ML}).
\end{equation}
Moreover, curvatures are uniformly bounded in any compact subset which has {\it definite distance} (independent of $\delta_j$) away from singular fibers.

The lemma below gives a notion for the convergence of functions away from the singular set,
using the  concept of equivariant-Gromov-Hausdorff convergence, which was discussed in Subsection \ref{ss:rescaling-semi-flat}.
The assumptions below can be weakened, but for our purposes, the following version will suffice.
Since this is rather standard, we will omit the proof.
\begin{lemma}\label{l:convergence-lemma}
In the above notation, let $(\cK,g_j^C)$ collapse to $(\PP^1,d_{ML})$ with a finite singular set $\cS\subset \PP^1$ such that for any $q_{\infty}\in \PP^1\setminus\mathcal{S}$, there exists $s>0$ such that the diagram of equivariant convergence holds
\begin{equation}
\xymatrix{
\Big(\widehat{B_{2s}(q_j)}, \hat{g}_j^C, \Gamma_j,\hat{q}_{j}\Big)\ar[rr]^{eqGH}\ar[d]_{\pr_j} &   & \Big(Y_{\infty}, \hat{g}_{\infty}, \Gamma_{\infty}, \hat{q}_{\infty}\Big)\ar [d]^{\pr_{\infty}}
\\
 \Big(B_{2s}(q_j), g_j^C, q_j\Big)\ar[rr]^{GH} && \Big(B_{2s}(q_{\infty}), d_{ML}, q_\infty\Big).
}\label{e:equivariant-convergence}
\end{equation}
Let $f_j$ be a sequence of smooth functions on $\cR_{\delta_j}$
and assume that for each $t\in(0,\frac{1}{100})$, there exists a constant $C_t>0$ independent of $j$ such that
\begin{equation}
\|f_j\|_{C^{1,\alpha}(\mathcal{R}_t)} \leq C_t,
\end{equation}
then there is a function $f_{\infty}\in C^{1,\alpha'}(\PP^1\setminus\mathcal{S})$ for any $\alpha'\in(0,\alpha)$ with the following properties:
Denote by  $\hat{f}_j$ the lifting of $f_j$ on $\widehat{B_{2s}(q_j)}$, then passing to a subsequence, $\hat{f}_j$ converges to a $\Gamma_{\infty}$-invariant function $\hat{f}_{\infty}\in C^{1,\alpha'}(Y_{\infty})$ in the $C^{1,\alpha'}$-topology for any $\alpha'\in(0,\alpha)$ such that $\hat{f}_{\infty}$ descends to $f_{\infty}$ with respect to the projection $\pr_{\infty}$.

\end{lemma}

Next, we have our main Liouville theorem in the case of a compact limit space.
\begin{proposition}\label{p:Liouville-Mclean}
Let $(\cK, g_j^C)$ be a collapsing sequence with bounded diameters
which are constructed in Section \ref{s:infinite-monodromy} and Section \ref{s:finite-monodromy} such that
as  $j\to+\infty$,
\begin{equation}
(\cK, g_j^C)\xrightarrow{GH}(\PP^1, d_{ML}).
\end{equation}
Let $\eta_j\in \mathring{\Omega}^1(\cK)$ be a sequence of coclosed $1$-forms
satisfying
$\|\eta_j\|_{C_{\mu}^0(\cK)}\leq 1$ and
\begin{align}
\|d^+_{g_j^C}\eta_j\|_{C_{\mu+1}^{0,\alpha}(\cK)} \to  0,
\end{align}
for $0<\mu<1$. Under the representation
\begin{equation}\eta_j=f_j^{(y)}dy+f_{j}^{(\bar{y})}d\bar y+ f_j^{(3)}\cdot e^3 +f_j^{(4)}\cdot e^4\end{equation} in $\mathcal{R}_j$, there are limiting functions $f_{\infty}^{(y)}$, $f_{\infty}^{(\bar{y})}$, $f_{\infty}^{(3)}$, $f_{\infty}^{(4)} \in C^0(\PP^1\setminus\mathcal{S})$ in the sense of Lemma \ref{l:convergence-lemma}.
Then $f_{\infty}^{(y)}=f_{\infty}^{(\bar{y})}=f_{\infty}^{(3)}=f_{\infty}^{(4)}=0$  on $\PP^1$.

\end{proposition}

\begin{proof}

By Proposition \ref{p:weighted-schauder}, $\|\eta_j\|_{C_{\mu}^{1,\alpha}(\cK)}\leq C$. By Lemma \ref{l:convergence-lemma}, for any point $q_{\infty}\in\PP^1\setminus\mathcal{S}$, there is some $s>0$ such that the
diagram \eqref{e:equivariant-convergence} holds,
where $B_{2s}(q_{\infty})\subset \PP^1\setminus \mathcal{S}$.
Moreover, there are functions $f_{\infty}^{(3)}$, $f_{\infty}^{(4)}$, $f_{\infty}^{(y)}\in C^0(\PP^1\setminus \mathcal{S})$ and a $1$-form $\eta_{\infty}\in \mathring{\Omega}^1(Y_{\infty})$ with a coordinate representation
\begin{equation}
\eta_{\infty}=f_{\infty}^{(y)}dy+f_{\infty}^{(\bar{y})}d\bar y+ f_{\infty}^{(3)}\cdot e^3 +f_{\infty}^{(4)} \cdot e^4,
\end{equation}
such that
\begin{equation}d^+_{\hat{g}_{\infty}}\eta_\infty=0, \ d^*_{\hat{g}_{\infty}}\eta_\infty = 0.
\end{equation}
It follows from Lemma \ref{l:D-xi-semi-flat}
that, on $Y_{\infty}$, the $\Gamma_{\infty}$-invariant functions
\begin{equation}\sqrt{\Ima(\bar{\tau}_1\tau_2)}\cdot \overline{F_{\infty}^{(x)}} \equiv \sqrt{\Ima(\bar{\tau}_1\tau_2)}(f_{\infty}^{(3)} - \sq f_{\infty}^{(4)})
\end{equation}
and $f_{\infty}^{(y)}$ are holomorphic in $y$ so that they descend to the quotient space $B_{2s}(q_{\infty})$.

By the above discussions, the function $|\eta_\infty|^2_{\hat{g}_{\infty}}$ is $\Gamma_{\infty}$-invariant and hence it becomes a function on the base $\PP^1\setminus\mathcal{S}$. By Bochner's formula,
\begin{equation}
\Delta_{\hat{g}_{\infty}}|\eta_\infty|^2_{\hat{g}_{\infty}} = 2 |\nabla_{\hat{g}_{\infty}}\eta_{\infty}|_{\hat{g}_{\infty}}^2 \geq 0.
\end{equation}
By  the $\Gamma_{\infty}$-invariance and the K\"ahler identity in terms of holomorphic coordinates, it turns out that on the punctured quotient space $\PP^1\setminus\mathcal{S}$
\begin{equation}\Delta_{d_{ML}}|\eta_\infty|^2_{\hat{g}_{\infty}}=2\partial_y\bp_y|\eta_\infty|^2_{\hat{g}_{\infty}} = \Delta_{\hat{g}_{\infty}}|\eta_\infty|^2_{\hat{g}_{\infty}}  \ge 0.
\end{equation}
We claim that the function $|\eta_{\infty}|_{\hat{g}_{\infty}}$ is vanishing globally on $\PP^1$.
For example,  near an $\I_{\nu}$-fiber, we can choose local coordinate $y$ such that $\tau_1=1$ and $\tau_2=-\frac{\nu\sq}{2\pi}\log y$. Under the growth assumption,  $|\eta_{\infty}|_{\hat{g}_{\infty}}=O(|y|^{-\mu})$  near $y=0$ for some $\mu\in(0,1)$.
On the other hand, $|\eta_\infty|_{\hat{g}_{\infty}}^2$
can be computed in terms of local coordinates,
\begin{equation}|\eta_\infty|^2_{\hat{g}_{\infty}} = (|f_{\infty}^{(3)}|^2 + |f_{\infty}^{(4)}|^2)
+\frac{1}{\Ima(\bar\tau_1\tau_2)}|f_{\infty}^{(y)}|^2. \end{equation}
So both $f_{\infty}^{(y)}$ and $\sqrt{\Ima(\bar{\tau}_1\tau_2)}\cdot  \overline{F_{\infty}^{(x)}}$ are holomorphic across $y=0$. It follows that both $f_{\infty}^{(y)}$ and  $\sqrt{\Ima(\bar{\tau}_1\tau_2)}\cdot \overline{F_{\infty}^{(x)}}$  are bounded around $y=0$. Thus, $|\eta_\infty|^2_{\hat{g}_{\infty}}\to 0$ as $y\to 0$.
Similarly, near $\I_{\nu}^*$, $\II$, $\III$, $\IV$, $\II^*$, $\III^*$, $\IV^*$ fibers, $|\eta_\infty|^2_{\hat{g}_{\infty}}\rightarrow 0$ as $y\rightarrow 0$ is also true. Thus by the maximum principle, $|\eta_\infty|^2_{\hat{g}_{\infty}}= 0$. This completes the proof.

\end{proof}

\section{Existence of collapsing hyperk\"ahler metrics}
\label{s:proof-of-main-theorem}

We will focus on the proof of Theorem~\ref{t:main-theorem} in this section. The main idea of the proof is to use an appropriate version of the implicit function theorem in order to perturb the approximate solutions to the genuine solutions.
Subsection~\ref{ss:hyperkaehler} outlines the framework of the perturbation.
In Subsection~\ref{ss:weighted-error}, we will prove error estimates in weighted H\"older spaces, which effectively measure how far the approximate solutions are from genuine ones.
We will prove the main uniform estimates for the linearized operator in Subsection~\ref{ss:uniform-injectivity}, employing the Liouville Theorems which were proved in Section~\ref{s:Liouville}.
Finally, in  Subsection~\ref{ss:proof-of-existence} we will complete the proof Theorem~\ref{t:main-theorem}.

\subsection{Framework of the perturbation analysis}
\label{ss:hyperkaehler}

Recall that in our context, we begin with an elliptic K3 surface $\cK$ with an associated elliptic fibration $\fF:\cK\to\PP^1$ such that there is a finite singular set $\cS \subset \PP^1$. In the gluing construction, we have fixed a holomorphic $(2,0)$-form $\Omega_{\cK}$ given by the complex structure of $\cK$. Then we define $\omega_2+\sqrt{-1}\omega_3\equiv\delta\Omega_{\cK}$, which satisfy
\begin{align}
\omega_2^2=\omega_3^2, \ \omega_2\wedge\omega_3=0.
\end{align}
Let $\dvol_0=\frac{1}{2}\omega_2^2=\frac{1}{2}\omega_3^2$.
From Remark~\ref{r:triple}, the family closed $2$-forms $\omega_\delta^C$ yields a family of approximately hyperk\"ahler metrics $g_\delta^C$ induced by the definite triple $\bm{\omega}_\delta^C \equiv (\omega_\delta^C,\omega_2,\omega_3)$. Recall that $\omega_\delta^C$ might not be a $(1,1)$-form because this property is destroyed in the case of finite monodromy. For $\delta \ll1$, the definite triple $\bm{\omega}_\delta^C$
is very close to being a hyperk\"ahler triple in the sense that
\begin{equation}
\|Q_{\bm{\omega}_\delta^C}-\Id\|_{C^0(\cK)} \ll1,\end{equation}
but a  more precise statement quantifying this in weighted H\"older spaces will be proved in Proposition~\ref{p:weighted-error} below.

The goal of the perturbation procedure will be to find a closed 2-form $\bm{\theta}=(\theta_1,0,0)$ such that $
\bm{\omega}_\delta^D \equiv\bm{\omega}_\delta^C+\bm{\theta}$ is an actual hyperk\"ahler triple on $\cK$,
which is equivalent to the system
\begin{align}
\frac{1}{2}(\omega_\delta^C + \theta_1)^2&=\frac{1}{2}\omega_2^2,\label{e:perturbed-hkt-eq-1}\\
\frac{1}{2} (\omega_\delta^C + \theta_1) \wedge \omega_2 &= \frac{1}{2}(\omega_\delta^C + \theta_1) \wedge \omega_3 =0. \label{e:perturbed-hkt-eq-2}\end{align}
Let us write $\theta_1=\theta_1^+ +\theta_1^-$, where $\theta_1^+,\theta_1^-$ are the
self-dual and anti-self-dual parts of $\theta_1$ with respect to $g_\delta^C$, respectively.
Then equations \eqref{e:perturbed-hkt-eq-1} and \eqref{e:perturbed-hkt-eq-2} can be written as
\begin{align}
\frac{1}{2}\Big(2\theta_1^+\wedge\omega_\delta^C + \theta_1^+\wedge \theta_1^+ \Big) & = \frac{1}{2}(\omega_2^2-(\omega_\delta^C)^2-\theta_1^-\wedge\theta_1^-),\label{e:hkt-1}
\\
 \frac{1}{2}\theta_1^+\wedge\omega_2& = - \frac{1}{2}\omega_\delta^C\wedge\omega_2,\label{e:hkt-2}
 \\
 \frac{1}{2} \theta_1^+\wedge\omega_3& = -   \frac{1}{2}\omega_\delta^C\wedge\omega_3.\label{e:hkt-3}
\end{align}
Let us denote
\begin{align}
\mathring{\Omega}^1(\cK)\equiv\{\eta\in\Omega^1(\cK)|d^*\eta = 0\},
\end{align}
and  let $\mathcal{H}_+^2(\cK)$ be the space of self-dual harmonic $2$-forms on $\cK$,
which is of dimension $3$, since $b_+(\cK)=3$.
For $\eta_1\in\mathring{\Omega}^1(\cK)$ and $\xi_1\in\mathcal{H}_+^2(\cK)$,
it follows from Hodge theory that a solution of following gauge-fixed system
\begin{align}
d^+ \eta_1 + \xi_1  = \mathfrak{H}_0\Big(\frac{1}{2}(\omega_2^2-(\omega_\delta^C)^2-d^- \eta_1\wedge d^- \eta_1, -\omega_\delta^C\wedge\omega_2, -\omega_\delta^C\wedge\omega_3)\Big),
\label{e:elliptic-system}
\end{align}
yields a solution $\theta_1 = d^+ \eta_1 + \xi_1$ of
the system \eqref{e:hkt-1}-\eqref{e:hkt-3}. Here the operator $\mathfrak{H}_0$ is defined as follows.
Let
\begin{equation}\mathfrak{G}_0: \Omega^2_+(\cK) \to  \Omega^4(\cK)\times \Omega^4(\cK)\times  \Omega^4(\cK)\end{equation}
be the map
\begin{align}\mathfrak{G}_0(\theta_1^+)  \equiv  \frac{1}{2}\Big(2\theta_1^+\wedge\omega_\delta^C + \theta_1^+\wedge \theta_1^+, \theta_1^+\wedge\omega_2, \theta_1^+\wedge\omega_3\Big). \end{align}
Note that, restricted to every point of $\cK$,  $\mathfrak{G}_0:\dR^3\to\dR^3$ is a local diffeomorphism at $0^3\in\dR^3$,
and the map $\mathfrak{H}_0$ is then defined as a pointwise local inverse near zero to $\mathfrak{G}_0$.

For a parameter $\mu\in(0,1)$ (the precise range of $\mu$ will be fixed later), we define the following Banach spaces,
\begin{align}
\mathfrak{A} & \equiv C_{\mu}^{1,\alpha}(\mathring{\Omega}^1(\cK))\oplus \mathcal{H}_+^2(\cK),
\\
\mathfrak{B} & \equiv C_{\mu+1}^{0,\alpha}(\Lambda^+ (\cK)),
\end{align}
where $\mathfrak{A}$ is equipped with the following norm: for $(\eta,\bar{\xi}^+)\in \mathfrak{A}$, \begin{equation}
\|(\eta,\bar{\xi}^+)\|_{\mathfrak{A}} \equiv
\|\eta\|_{C_{\mu}^{1,\alpha}(\cK)}
 +
  \| \bar{\xi}^+ \|_{C_{\mu+1}^{0,\alpha}(\cK)},
\end{equation}
and where the weighted H\"older space are defined in Definition~\ref{d:weighted-space},
with weight function is given by \eqref{d:weight}.

We define the operator $\mathscr{F}_{\delta}:\mathfrak{A}\to\mathfrak{B}$ by
\begin{equation}
\mathscr{F}_{\delta}(\eta,\bar{\xi}^+)\equiv d^+ \eta + \bar{\xi}^+ - \mathfrak{H}_0\Big(\frac{1}{2}(\omega_2^2-(\omega_\delta^C)^2-d^- \eta\wedge d^- \eta, -\omega_\delta^C\wedge\omega_2, -\omega_\delta^C\wedge\omega_3)\Big),
\end{equation}
so that a zero of $\mathscr{F}_{\delta}$ solves the system~\eqref{e:elliptic-system}.
The linearization of $\mathscr{F}_{\delta}$ at $(0,0)$ is the operator
\begin{equation}
\mathscr{L}_{\delta} \equiv d^+ \oplus \Id: \mathfrak{A}\longrightarrow \mathfrak{B}.
\end{equation}
The nonlinear part of $\mathscr{F}_{\delta}$ is given by
\begin{align}
\mathscr{N}_{\delta}(\eta,\bar{\xi}^+) \equiv & \ \mathfrak{H}_0\Big(\frac{1}{2}(\omega_2^2-(\omega_\delta^C)^2, -\omega_\delta^C\wedge\omega_2, -\omega_\delta^C\wedge\omega_3)\Big)
\nonumber\\
& - \mathfrak{H}_0\Big(\frac{1}{2}(\omega_2^2-(\omega_\delta^C)^2-d^- \eta\wedge d^- \eta, -\omega_\delta^C\wedge\omega_2, -\omega_\delta^C\wedge\omega_3)\Big).\end{align}

The main tool is the following standard implicit function theorem
(see for example~\cite{BM}).
\begin{lemma} \label{l:implicit-function}
Let $\mathscr{F}:\mathfrak{A} \to \mathfrak{B}$ be a map between two Banach spaces such that
\begin{equation}\mathscr{F}(\bx)-\mathscr{F}(\bo)=\mathscr{L}(\bx)+\mathscr{N}(\bx),\end{equation} where the operator $\mathscr{L}:\mathfrak{A}\to\mathfrak{B}$ is linear and $\mathscr{N}(\bo)=\bo$. Assume that
\begin{enumerate}
\item $\mathscr{L}$ is an isomorphism with $\| \mathscr{L}^{-1} \| \leq C_1$,

\item there are constants $r>0$ and $C_2>0$ with $r<\frac{1}{5C_1 C_2}$ such that
\begin{enumerate}\item  $\| \mathscr{N}(\bx) - \mathscr{N}(\by) \|_{\mathfrak{B}} \leq C_2\cdot ( \|\bx\|_{\mathfrak{A}} + \|\by\|_{\mathfrak{A}} ) \cdot  \| \bx - \by \|_{\mathfrak{A}} $ for all $x,y\in B_r(\bo)\subset {\mathfrak{A}}$,

\item $\| \mathscr{F}(\bo) \|_{\mathfrak{B}} \leq \frac{r}{2C_1}$,
\end{enumerate}
\end{enumerate}
then there exists a unique solution to $\mathscr{F}(\bx)=\bo$ in $\mathfrak{A}$   such that
\begin{equation}
\|\bx\|_{\mathfrak{A}} \leq 2C_1 \cdot \|\mathscr{F}(\bo)\|_{\mathfrak{B}}.
\end{equation}

\end{lemma}
In the following subsections, we will show that all of the assumptions in Lemma ~\ref{l:implicit-function}
are satisfied when $\delta$ is sufficiently small.

\subsection{Weighted error estimates}

\label{ss:weighted-error}

As noted above,  $\bm{\omega}_\delta^C = (\omega_\delta^C, \omega_2, \omega_3)$ is hyperk\"ahler
away from the regions  $\cS_{\ALG}$, and the damage zone regions in $\cS_{\I_{\nu}}$ and  $\cS_{\I_{\nu}^*}$.
The following proposition gives precise estimates for the error term with respect to the weighted H\"older spaces.
\begin{proposition}
[Weighted error estimates]\label{p:weighted-error} Let $\fF:\cK\to \PP^1$ be an elliptic K3 surface with a family of collapsing metrics $g_\delta^C$ induced from $\bm{\omega}_\delta^C$ such that
\begin{equation}
(\cK, g_\delta^C)\xrightarrow{GH} (\PP^1, d_{ML}).
\end{equation}
Let $\cS_{\ALG}$,  $\cS_{\I_{\nu}}$ and  $\cS_{\I_{\nu}^*}$ be the regions chosen in previous sections which surround the singular fibers of $\fF$.
Then the weighted error estimate is listed as follows:
\begin{enumerate}
\item
In the singular region $\cS_{\ALG}$ near a singular fiber $\fF^{-1}(p)$ with finite monodromy,
\begin{align}
\|Q_{\bm{\omega}_\delta^C}-\Id\|_{C_{\mu+1}^{0,\alpha}(\cO_{\delta^{\ell}}(p))}
& \leq
C\cdot \delta^{\ell(\frac{7}{5}+\mu)},
\\
\|Q_{\bm{\omega}_\delta^C}-\Id\|_{C_{\mu+1}^{0,\alpha}(\mathcal{A}_{\delta^{\ell},2\delta^{\ell}}(p))}
& \leq
C\cdot (\delta^{\ell(\frac{7}{5}+\mu)} +  \delta^{\ell(\mu-1)+2}),
\end{align}
where $\cO_{\delta^{\ell}}(p)\equiv\fF^{-1}(B_{2\delta^{\ell}}(p))$, $\mathcal{A}_{\delta^{\ell},2\delta^{\ell}}(p)\equiv\fF^{-1}(A_{\delta^{\ell},2\delta^{\ell}}(p))$, $\cS_{\ALG}=\cO_{\delta^{\ell}}(p)\cup \mathcal{A}_{\delta^{\ell},2\delta^{\ell}}(p)$.

\item Near a singular fiber of Type $\I_{\nu}$ for some $\nu\in\dZ_+$, the approximate triple in the singular region $\cS_{\I_{\nu}}$ yields to the estimate
\begin{equation}
\|Q_{\bm{\omega}_\delta^C}-\Id\|_{C_{\mu+1}^{0,\alpha}(\fF^{-1}(A_{\delta_0, 2\delta_0}(p)))}
\leq C_1e^{-C_2/\delta}.\label{e:exp-small-estimate}
\end{equation}

\item   Near a singular fiber $\fF^{-1}(p)$ of Type $\I_{\nu}^*$ for some $\nu\in\dZ_+$, the approximate triple in the singular region $\cS_{\I_{\nu}^*}$ yields to the following estimate,
\begin{align}
\|Q_{\bm{\omega}_\delta^C}-\Id\|_{C_{\mu+1}^{0,\alpha}(\fF^{-1}(A_{\delta_0, 2\delta_0}(p)))}
 & \leq  C_1 \cdot e^{-C_2/\delta},
 \\
 \|Q_{\bm{\omega}_\delta^C}-\Id\|_{C_{\mu+1}^{0,\alpha}(\mathcal{D}_{\EH})}
 &\leq
C \cdot \Big(\delta^{\mu+1} \cdot \ee_{\lambda}^{\mu+5} + (\ee_{\lambda}\cdot \delta)^{\mu+3}\Big),\end{align}
where
$\mathcal{D}_{EH}\equiv \{\bx\in\cK|  \ee_{\lambda}\delta\leq d_{g_\delta^C}(\bx, q_{\lambda}) \leq 2\ee_{\lambda} \delta,\ 1\leq\lambda\leq 4
\}$
denotes the transition region for gluing the Eguchi-Hanson metrics $\omega_{\EH}^{\flat}$.
\end{enumerate}

\end{proposition}

\begin{proof}

For Item (1), note that
the components of $Q_{\bm{\omega}_\delta^C} - \Id$
in $\cO_{\delta^{\ell}}(p)$ are \begin{equation}
  (\Psi_* \omega_\delta^{\cG})^2-\frac{1}{2}\delta^2\Omega_{\cK}\wedge\bar\Omega_{\cK} = \frac{1}{2}\delta^4 \Psi_*\Omega_{\cG}\wedge
\Psi_* \bar \Omega_{\cG} - \frac{1}{2}\delta^2\Omega_{\cK}\wedge\bar\Omega_{\cK}
  \end{equation}
  and \begin{equation}
 \Psi_* \omega_\delta^{\cG}\wedge\delta\Omega_{\cK} = \Psi_*\omega_\delta^{\cG}\wedge (\delta\Omega_{\cK}-\delta^2 \Psi_*\Omega_{\cG}).
  \end{equation} Thus, in order to estimate $Q_{\bm{\omega}_\delta^C} - \Id$ in $\cO_{\delta^{\ell}}(p)$, it suffices to compute $\delta\Omega_{\cK}-\delta^2 \Psi_*\Omega_{\cG}$ using the norm defined by the metric
$g_{\delta}^C$. Using the weighted norm, if follow immediately from Proposition \ref{p:complexerror} that
\begin{equation}
 \|  \delta\Omega_{\cK}-\delta^2 \Psi_* \Omega_{\cG} \|_{C_{\mu+1}^{0,\alpha}(\cO_{\delta^{\ell}}(p))}
\leq C \delta^{\ell(\frac{7}{5}+\mu)}.
 \end{equation}
Next, in addition to the above complex structure distortion,  the error term $|Q_{\bm{\omega}_\delta^C}-\Id|$ in the annulus transition region $\mathcal{A}_{\delta^{\ell},2\delta^{\ell}}(p)$ also arises from
the difference $| \omega_\delta^{\cG}- \omega_\delta^{\FF}|$. By Proposition \ref{p:exact-error-ALG},
\begin{align}
\|\omega_\delta^{\cG}- \omega_\delta^{\FF}\|_{C^0(\mathcal{A}_{\delta^{\ell},2\delta^{\ell}}(p))} &\leq C\delta^{2-2\ell},
\\
[\omega_\delta^{\cG}- \omega_\delta^{\FF}]_{C^{\alpha}(\mathcal{A}_{\delta^{\ell},2\delta^{\ell}}(p))} &\leq C\delta^{2-\ell(\alpha+2)}.
\end{align}
So it follows that
\begin{align}
\|Q_{\bm{\omega}_\delta^C}-\Id\|_{C_{\mu+1}^{0,\alpha}(\mathcal{A}_{\delta^{\ell},2\delta^{\ell}}(p))} \leq C \cdot (\delta^{\ell(\frac{7}{5}+\mu)} + \delta^{\ell(\mu-1)+2}).
\end{align}

In Item (2),
the error term $|Q_{\bm{\omega}_\delta^C}-\Id|$ around a singular fiber of Type $\I_{\nu}$ has an exponential decaying rate. The estimate \eqref{e:exp-small-estimate} in the special case $\nu=1$ was proved in \cite[Theorem~4.4]{GW}.
The computations for obtaining the exponential decaying
rate in the general $\I_{\nu}$ case are along the same lines. Indeed, the exponential decaying rate essentially
arises from the asymptotic behavior of the Green's function  in Lemma~\ref{l:Green's-function}.

Finally, we prove Item (3). The exponential error
estimate on $\fF^{-1}(A_{\delta_0, 2\delta_0}(p))$
is the same as the error estimate in Item (2), so we omit it.
It suffices to prove the error estimate near the Eguchi-Hanson bubbles. In this region, the approximate metric
is constructed by gluing $4$ copies of
Eguchi-Hanson metrics with
the quotient multi-Ooguri-Vafa metric near the orbifold points with the flat tangent cone $\dR^4/\dZ_2$.
To estimate the size of the error $|Q_{\bm{\omega}_\delta^C}-\Id|$ in the damage zone
\begin{equation}
\Big\{\bx\in\cK \Big|\ee_{\lambda}\delta\leq d_{g_\delta^C}(\bx, q_{\lambda}) \leq 2\ee_{\lambda} \delta,\ 1\leq\lambda\leq 4\Big\},\label{e:EH-damage-zone}
\end{equation}
 we need to analyze the asymptotic behavior of the rescaled Eugchi-Hanson metric $g_{\EH}^{\flat}$ and the local behavior of the quotient multi-Ooguri-Vafa metric $\check{g}_{\delta,\nu}^{\flat}$
around the singularity. By the approximation estimate \eqref{e:EH-approximation} for the Eguchi-Hanson metric, we have
\begin{equation}
 \Big|\nabla_{g_{\dR^4/\dZ_2}^{\flat}}^k(\omega_{\EH}^{\flat} - \omega_{\dR^4/\dZ_2}^{\flat}) \Big|_{g_{\dR^4/\dZ_2}^{\flat}} \leq C\cdot\frac{(\ee_{\lambda}^2\cdot\delta)^4}{(\ee_{\lambda}\cdot \delta)^{4+k}}=C\cdot\frac{\ee_{\lambda}^{4-k}}{ \delta^k}.
\end{equation}
 On the other hand, for the metric approximation around orbifold singularities, \eqref{e:orbifold-approximation} implies
 \begin{equation}
 \Big|\nabla_{g_{\dR^4/\dZ_2}^{\flat}}^k(\omega_{\delta,\nu}^{\flat} - \omega_{\dR^4/\dZ_2}^{\flat}) \Big|_{g_{\dR^4/\dZ_2}^{\flat}} \leq C\cdot (\ee_{\lambda}\cdot \delta)^{2-k}.
 \end{equation}
Therefore, the weighted estimate in the damage zone \eqref{e:EH-damage-zone} is given by
\begin{equation}
\|Q_{\bm{\omega}_\delta^C}-\Id\|_{C_{\mu+1}^{0,\alpha}}\leq C \cdot \Big(\delta^{\mu+1} \cdot \ee_{\lambda}^{\mu+5} + (\ee_{\lambda}\cdot \delta)^{\mu+3}\Big).
\end{equation}
This completes the proof.

\end{proof}

\subsection{The uniform injectivity estimates}
\label{ss:uniform-injectivity}

In this subsection, we will establish the uniform injectivity estimates for the linear operator $d^+_{g_\delta^C}$. The crucial part in proving such estimates is to use the Liouville theorems in Section \ref{s:Liouville} in the contradiction arguments.

To begin with, we give a more accurate upper bound for weight parameter $\mu\in(0,1)$.
In our following weighted analysis, the upper bound of the weight parameter $\mu\in(0,1)$ is determined by the Liouville type theorems in Section \ref{ss:Liouville-noncompact}.
By Proposition \ref{p:Liouville-sector}, for each punctured sector $\punsec(\beta)$ with $\beta\in(0,1)$, the growth parameter $\mu$ is chosen such that
\begin{align}
0<\mu<\iota_{\beta,\beta}=
\begin{cases}
1, & \beta\in(0,1/2],
\\
\frac{1}{\beta}-1, & \beta\in(1/2,1).
\end{cases}
\end{align}
Here $\beta\in(0,1)$ corresponds to
  the angle parameter in the ALG spaces associated to the singular fibers of type $\II^*$, $\III^*$, $\IV^*$, $\II$, $\III$, $\IV$ and $\I_0^*$.
Hence the possible range for $\beta$ is
 \begin{equation}
 \beta \in \{\frac{1}{6},\frac{1}{4},\frac{1}{3},\frac{5}{6},\frac{3}{4},\frac{2}{3},\frac{1}{2}\Big\},
 \end{equation}
 which gives $\iota_{\beta,\beta}\ge\frac{1}{5}$.
  Therefore, from now on, the weight parameter is will be chosen so that
$\mu \in (0,\frac{1}{5})$.

\begin{proposition}[Uniform injectivity estimate] \label{p:injectivity-of-D}
Given  $0<\delta\ll 1$,  $\alpha\in(0,1)$ and $\mu\in(0,\frac{1}{5})$, there exists
	 $C=C(\alpha)>0$ independent of $\delta>0$ such that for every $\eta\in\mathring{\Omega}^1(\cK)$,
	\begin{equation}
		\|\eta\|_{C_{\mu}^{1,\alpha}(\cK)} \leq C \|  d^+_{g_\delta^C} \eta \|_{C_{\mu+1}^{0,\alpha}(\cK)}.
	\end{equation}
\end{proposition}

\begin{proof}
	By Proposition \ref{p:weighted-schauder}, it suffices to show the uniform estimate
	\begin{equation}
			\|\eta\|_{C_{\mu}^{0}(\cK)} \leq C \|  d^+_{g_\delta^C} \eta \|_{C_{\mu+1}^{0,\alpha}(\cK)}.
	\end{equation}
We will prove it by contradiction and suppose that no such a uniform constant $C>0$ exists. That is, there are contradicting sequences:
\begin{enumerate}
\item a sequence of manifolds $(\cK, g_j^C)$ with a sequence of parameters $\delta_j\to0$ satisfying
\begin{equation}
(\cK, g_j^C, \bm{x}_j)\xrightarrow{GH}(X_{\infty}, d_{ML},\bm{x}_{\infty}).
\end{equation}

\item a sequence of $1$-forms $\eta_j\in\mathring{\Omega}^1(\mathcal{M}_j^4)$ satisfying
\begin{align}
&\|d^+_{g_j^C}\eta_j\|_{C_{\mu+1}^{0,\alpha}(\cK)} \to 0,
\\
&\|\eta_j\|_{C_{\mu}^0(\cK)} =1,
\\
&|\fs(\bm{x}_j)^{\mu}\cdot \eta_j(\bm{x}_j)| = 1,
\label{e:inj-contra-est}
\end{align}
as $j\to\infty$.
\end{enumerate}
Then, centering around the reference points $\bx_j$, we will rescale the metrics by
\begin{align}
\tilde{g}_j^C &= (\lambda_j)^2g_j^C
\\
\lambda_j &= \fs_j^{-1} \equiv \fs(\bx_j)^{-1}.
\end{align}
To guarantee rescaling invariance of \eqref{e:inj-contra-est} and the weighted Schauder estimate, we will simultaneously rescale the contradicting $1$-forms $\eta_j$ and
the functions $\fs$ by
\begin{align}
\begin{cases}
\tilde{\fs} = \fs_j^{-1}\cdot \fs, &
\\
\tilde{\eta}_j = \kappa_j \cdot \eta_j, & \kappa_j \equiv \fs_j^{\mu-1}.
\end{cases}
\end{align}

In terms of these rescalings, the contradiction assumptions become the following:
\begin{enumerate}
\item[($\tilde{1}$)] There is some constant $C_0>0$ independent of $j\in\dZ_+$ such that the rescaled metrics $\tilde{g}_j^C$ satisfies
\begin{align}
\frac{1}{C_0}\leq  \sup\limits_{B_1(\bx_j)} & |\Rm_{\tilde{g}_j^C}| \leq C_0,
\\
(\cK, \tilde{g}_{j}, \bm{x}_j)  \xrightarrow{GH} &(\M_{\infty}, \tilde{d}_{\infty},\bm{x}_{\infty}),
\end{align}
for some  complete metric space $(\M_{\infty}, \tilde{d}_{\infty},\bm{x}_{\infty})$.

\item[($\tilde{2}$)] The rescaled contradicting $1$-forms $\tilde{\eta}_j\in\mathring{\Omega}^1(\mathcal{M}_j^4)$ satisfy \begin{align}
&\|d^+_{\tilde{g}_j^C}\tilde{\eta}_j\|_{C_{\mu+1}^{0,\alpha}(\cK)} \to 0,
\\
&\|\tilde{\eta}_j\|_{C_{\mu}^0(\cK)} =1,
\\
&|\tilde{\fs}(\bm{x}_j)^{\mu}\cdot \tilde{\eta}_j(\bm{x}_j)|  =1.
\label{e:contradiction-1-form}
\end{align}
as $j\to\infty$.
\end{enumerate}
Under the rescaled weighted H\"older norm,
applying ($\tilde{2}$) and the weighted Schauder estimate in Proposition \ref{p:weighted-schauder}, we have
\begin{equation}
\|\tilde{\eta}_j\|_{C_{\mu}^{1,\alpha}(\cK)} \leq C.\label{e:bounded-C^{1,alpha}}
\end{equation}

Our contradiction arguments will be done in various regions of $\cK$ around the singular fibers of type $\I_{\nu},\I_{\nu}^*$ ($\nu\in\dZ_+$) and the fibers of finite monodromy. Recall those regions were denoted
by $\cS_{\I_{\nu}}$, $\cS_{\I_{\nu}^*}$ and $\cS_{\ALG}$ in Section \ref{s:metric-geometry}, respectively.

 \begin{flushleft}
{\bf Region $\cS_{\ALG}$ (singular fibers of finite monodromy):}
\end{flushleft}

We will further divide region into $\cS_{\ALG,1}$ and $\cS_{\ALG,2}$
according to the discussions in Section~\ref{ss:regularity-scale-ALG}.

\begin{flushleft}
{\bf Sub-region $\cS_{\ALG,1}$:}
\end{flushleft}
As in the proof of
Proposition \ref{p:regularity-scale-ALG},
the rescaled spaces $(\cK, \tilde{g}_j^C, \bx_j)$ for $\bx_j\in\cS_{\ALG,1}$ satisfy
the pointed convergence
\begin{equation}
(\cK, \tilde{g}_j^C, \bx_j) \xrightarrow{C^{\infty}} (\cG, g^{\cG}, \bx_{\infty}),
\end{equation}
where $(\cG, \tilde{g}_{\infty}, \bx_{\infty})$ is the complete hyperk\"ahler ALG space determined by the corresponding singular fiber with finite monodromy.

Combining the above smooth convergence of $\tilde{g}_j^C$ with uniform $C^{1,\alpha}$-estimate \eqref{e:bounded-C^{1,alpha}} for  $\tilde{\eta}_j$, then for every $\gamma\in(0,\alpha)$, $\tilde{\eta}_j$ converges to some limiting $1$-form $\tilde{\eta}_{\infty}\in\mathring{\Omega}^1(\cG)$  in the $C^{1,\gamma}$-topology which satisfies
\begin{align}
\begin{cases}
d^+_{\tilde{g}_{\infty}} \tilde{\eta}_{\infty} = 0 \\
|\tilde{\fs}_{\infty}(\bm{x}_{\infty})^{\mu}\cdot \tilde{\eta}_{\infty}(\bm{x}_{\infty}) | = 1 \\
\| \tilde{\eta}_{\infty} \|_{C_{\mu}^0(\cG)} =  1,\\
\end{cases}
\end{align}
and $\tilde{\fs}$ converges to the limiting function \begin{align}
\tilde{\fs}(\bm{x})  =
\begin{cases}
1, & \bm{x}\in B_{1}(\bx_{\infty})
\\
 d_{\tilde{g}_{\infty}}(\bm{x}, \bx_{\infty}), & \bm{x} \in \cG \setminus B_{2}(\bx_{\infty}).
\end{cases}
\end{align}
Notice that the above norm bound implies that for all $\bm{x}\in\cG\setminus B_{2}(\bx_{\infty})$,
\begin{equation}
|\tilde{\eta}_{\infty}(\bm{x})| \leq  (d_{\tilde{g}_{\infty}}(\bm{x},\bx_{\infty}))^{-\mu}.
\end{equation}
Since $\tilde{\eta}_{\infty}\in\Ker(d^+_{\tilde{g}_{\infty}})$, immediately $\tilde{\eta}_{\infty}$ is harmonic with respect to the hyperk\"ahler metric $\tilde{g}_{\infty}$ on the complete ALG space $\cG$.
Applying Lemma \ref{l:harmonic-1-form-Liouville}, we have $\tilde{\eta}_{\infty} = 0$.

\begin{flushleft}
{\bf Sub-region $\cS_{\ALG,2}$:}
\end{flushleft}
In this sub-region, the rescaled sequence $(\cK,\tilde{g}_j^C,\bx_j)$
is  collapsing to $T_{\infty}\cG$, the asymptotic cone of the corresponding ALG space $\cG$.  That is, as $j\to+\infty$,
\begin{equation}
(\cK,\tilde{g}_j^C,\bx_j)\xrightarrow{GH} (C(S_{2\pi\beta}^1), d_{C,2\pi\beta}, \bx_{\infty})
\end{equation}
where $d_{C,2\pi\beta}(\bx_{\infty},0^2)=1$ and $(C(S_{2\pi\beta}^1),d_{C,2\pi\beta})$ is a flat cone
for some angle parameter
\begin{equation}
 \beta \in \{\frac{1}{6},\frac{1}{4},\frac{5}{6},\frac{3}{4},\frac{2}{3},\frac{1}{3},\frac{1}{2}\Big\}.
 \end{equation}
Let $\Sec(\beta)\subset \dR^2$ be the open sector
obtained from the flat cone $C(S_{2\pi\beta}^1)$
removing rays $\theta=0$
and $\theta=2\pi\beta$.
So we can take a sequence of open subsets
$\mathcal{U}_j\subset \cK$ such that
\begin{equation}
(\mathcal{U}_j, \tilde{g}_j^C, \bx_j)\xrightarrow{GH} (\Sec(\beta),g_0, \bx_{\infty}).
\end{equation}
Moreover,
 $\mathcal{U}_j$ can be chosen such that it has a natural torus bundle structure
 \begin{equation}
 \dT^2\to \mathcal{U}_j \xrightarrow{\pi} \Sec(\beta) \end{equation}
with the standard holomorphic coordinate system $\{\mathbf{u},\mathbf{v}\}$,
where $\mathbf{u}$ and $\mathbf{v}$
are the holomorphic coordinates on the base and the torus fiber, respectively.
Restricted to $\mathcal{U}_j$, the contradicting sequence $\tilde{\eta}_j$, as real-valued $1$-forms, can be written in terms of the complex coordinates
\begin{equation}\tilde{\eta}_j = f_j d \mathbf{u} + \bar{f}_j d\bar{\mathbf{u}} + h_j d \mathbf{v} + \bar{h}_j d\bar{\mathbf{v}},\label{e:1-form-u-v}
\end{equation}
where $f_j$ and $h_j$
are complex-valued functions on $\mathcal{U}_j$. Taking any point $q_{\infty}\in \Sec(\beta)$ and letting $q_j\to q_{\infty}$ for $q_j\in \mathcal{U}_j$, then there is some $s>0$
depending on $q_{\infty}$
 such that the universal covering of $B_{2s}(q_j)$ is non-collapsing and the following equivariant-Gromov-Hausdorff convergence holds:
\begin{equation}
\xymatrix{
\Big(\widehat{B_{2s}(q_j)}, \hat{g}_j^C, \Gamma_j,\hat{q}_{j}\Big)\ar[rr]^{eqGH}\ar[d]_{\pr_j} &   & \Big(Y_{\infty}, \hat{g}_{\infty}, \Gamma_{\infty}, \hat{q}_{\infty}\Big)\ar [d]^{\pr_{\infty}}
\\
 \Big(B_{2s}(q_j), g_j^C, q_j \Big)\ar[rr]^{GH} && \Big(B_{2s}(q_{\infty}), d_{ML}, q_\infty \Big),
}\label{e:equivariant-convergence2}
\end{equation}
where $\Gamma_j\cong \dZ^2$ and $\Gamma_{\infty}\cong \dR^2$.
Remark that the convergence on the universal covering space $\widehat{B_{2s}(q_j)}$ can improved to be $C^{\infty}$.

Denote by
\begin{equation}\hat{\eta}_j= F_j d \mathbf{u} + \bar{F}_j d\bar{\mathbf{u}} + H_j d \mathbf{v} + \bar{H}_j d\bar{\mathbf{v}} \end{equation} the lifting of the $1$-forms $\tilde{\eta}_j$ on $\widehat{B_{2s}(q_j)}$, then combining the contradiction assumption \eqref{e:contradiction-1-form} and the above $C^{\infty}$-convergence of the local universal covers, we have that $\hat{\eta}_j\in\mathring{\Omega}^1(\widehat{B_{2s}(q_j)})$
converges to
$\hat{\eta}_{\infty} \in\mathring{\Omega}^1(Y_{\infty})$ so that
$d^+_{\hat{g}_{\infty}}\hat{\eta}_{\infty} = 0 \ \text{on}\ Y_{\infty}$.
Then by direct computations,
\begin{equation}
\Delta_{\hat{g}_{\infty}} \Rea(F_{\infty}) = \Delta_{\hat{g}_{\infty}} \Ima(F_{\infty}) =
\Delta_{\hat{g}_{\infty}} \Rea(H_{\infty}) =
\Delta_{\hat{g}_{\infty}} \Ima(H_{\infty}) =
0\ \text{on}\ Y_{\infty}.\label{e:harmonic-functions-limiting-cover}
\end{equation}
The definition of ALG space shows that the torus bundle restricted to $B_{2s}(q_j)\to B_{2s}(q_{\infty})$
is almost a metric product, which  implies
that any $\Gamma_{\infty}$-orbit is totally geodesic in $Y_{\infty}$.
Therefore, \eqref{e:harmonic-functions-limiting-cover} descends to $ B_{2s}(q_{\infty})\subset \Sec(\beta)$.

By Lemma \ref{l:convergence-lemma},
there are global limiting functions $f_{\infty}$
and $h_{\infty}$ on the open sector $\Sec(\beta)$ satisfying
\begin{align}
|\Rea(f_{\infty})(x)|  + |\Ima(f_{\infty})(x)| +|\Rea(h_{\infty})(x)|  + |\Ima(h_{\infty})(x)| \leq \frac{C}{r(x)^{\mu}},\label{e:components-weighted-C_0}
\end{align}
where $r$ is the distance to the origin. Moreover,
\begin{align}
f_{\infty}(r,\beta)&=e^{-\i \cdot 2\pi\beta}f_{\infty}(r,0)\\
h_{\infty}(r,\beta)&=e^{\i \cdot 2\pi\beta}h_{\infty}(r,0).
\end{align}
We have just shown that on $\Sec(\beta)$,
\begin{equation}
\Delta_{g_0}\Rea(f_{\infty})= \Delta_{g_0}\Ima(f_{\infty})=
\Delta_{g_0}\Rea(h_{\infty})=
\Delta_{g_0}\Ima(h_{\infty})= 0.
\end{equation}
Recall that we have chosen $\mu\in (0,\frac{1}{5})$, then Proposition \ref{p:Liouville-sector} implies that
in any case of $\beta$ in Table \ref{ALGtable},
we have
\begin{equation}
\Rea(f_{\infty})=\Ima(f_{\infty})= \Rea(h_{\infty})=\Ima(h_{\infty}) = 0.
\end{equation}
On other hand, this contradicts to the weighted control
\begin{equation}
\Big|\tilde{\fs}(\bx_j)^{\mu}\cdot\tilde{\eta}_j(\bx_j)\Big|=\|\tilde{\eta}_j\|_{C_{\mu}^0(\cK)}=1,
\end{equation}
which completes the proof in this region.

\begin{flushleft}
{\bf Region $\cS_{\I_{\nu}}$ (singular fibers of type $\I_{\nu}$, $\nu\in\dZ_+$):}
\end{flushleft}
By the regularity scale analysis in Section \ref{ss:regularity-scale-I_v}, the region $\cS_{\I_{\nu}}$
can be subdivided into
$\cS_{\I_{\nu},1}$, $\cS_{\I_{\nu},2}$
and $\cS_{\I_{\nu},3}$.

\begin{flushleft}
{\bf Sub-region $\cS_{\I_{\nu},1}$:}
\end{flushleft}
As $\bx_j\in \cS_{\I_{\nu},1}$, we have the following $C^{\infty}$-convergence for the rescaled metrics
\begin{equation}
(\cK,\tilde{g}_j^C,\bx_j) \xrightarrow{C^{\infty}} (\dC^2, g_{TN}, \bx_{\infty}),
\end{equation}
where  $(\dC^2, g_{TN}, \bx_{\infty})$ is
the Ricci-flat Taub-NUT space. The remainder of the contradiction arguments immediately follow from Lemma~\ref{l:harmonic-1-form-Liouville}, which is almost verbatim to the proof in Sub-region $\cS_{\ALG,1}$, so we omit the details.

\begin{flushleft}
{\bf Sub-region $\cS_{\I_{\nu},2}$:}
\end{flushleft}
In this sub-region, depending upon the distance to the singular fiber, there are three types of rescaled limits:
\begin{enumerate}
\item[(a)] the Ricci-flat Taub-NUT space $(\dC^2, g_{TN}, \bx_{\infty})$,
\item[(b)] the Euclidean space $\dR^3$,
\item[(c)] the flat product space $\dR^2\times S^1$.
\end{enumerate}
The argument in Case (a) is identical to $\cS_{\I_{\nu},1}$. In Case (b), the rescaled metrics yield to
\begin{equation}
(\cK, \tilde{g}_j^C, \bx_j)\xrightarrow{GH} (\dR^3, g_{\dR^3}, \bx_{\infty})
\end{equation}
with $d_{\dR^3}(\bx_{\infty}, 0^3) = 1$. Moreover, the convergence keeps curvatures uniformly bounded away from the origin $0^3\in\dR^3$. Our basic strategy is to reduce the convergence of the $1$-form $\tilde{\eta}_j$ to the convergence of the coefficient functions.
 Let $p_i, i=1,2,...,\nu$, be a fixed monopole in $\cS_{\I_{\nu}}$. For any fixed $\xi>1$, define
 \begin{equation}
 \mathcal{U}_j \equiv B_{\ell_j}^{g_j^C}(p_i)\setminus B_{\delta_j}^{g_j^C}(p_i)
 \end{equation}
where $\ell_j \equiv  \xi^{-1} \cdot d_{g_j^C}(\bx_j, p_i)$ and $\delta_j \equiv  \xi \cdot d_{g_j^C}(\bx_j, p_i)$, then $\mathcal{U}_j$ is naturally a circle bundle. Now with respect to the rescaled metrics
\begin{equation}
\tilde{g}_j^C =\lambda_j^2 g_j^C,\ \lambda_j = \frac{1}{d_{g_j^C}(\bx_j, p_i)},
\end{equation} the open subsets
$\mathcal{U}_j$ become large punctured balls
 $B_{\xi{-1}}^{\tilde{g}_j^C}(p_i)\setminus B_{\xi}^{\tilde{g}_j^C}(p_i)$
 and since $\xi$ is arbitrary,
\begin{equation}
 (\mathcal{U}_j,\tilde{g}_j^C,\bx_j) \xrightarrow{GH} (\dR^3\setminus\{0^3\}, g_{\dR^3}, \bx_{\infty}).
 \end{equation}
Therefore, restricted to $\mathcal{U}_j$, the contradicting $1$-forms $\tilde{\eta}_j$ can be written as
\begin{equation}
\tilde{\eta}_j = f_{j}^{1} \cdot \theta_j^{1} + f_{j}^{2} \cdot \theta_j^{2} + f_{j}^{3} \cdot \theta_j^{3} + f_{j}^{4} \cdot \theta_j^{4},
\end{equation}
where $\theta_j^{1}$,  $\theta_j^{2}$, $\theta_j^{3}$ and $\theta_j^{4}$ are the orthonormal basis defined as the rescaling of the pull back of $d u_1$, $d u_2$, $d u_3$ and the connection 1-form for $(u_1,u_2, u_3)\in\dR^3$.

Given the contradiction assumption \eqref{e:contradiction-1-form},
applying the similar arguments as Lemma~\ref{l:convergence-lemma},
we obtain limiting functions $f_{\infty}^{1}$,
 $f_{\infty}^{2}$,
 $f_{\infty}^{3}$,
 and $f_{\infty}^{4}$ on $\dR^3\setminus\{0^3\}$ which satisfy
 \begin{align}
 \Delta_{\dR^3} f_{\infty}^{1} (\bx) =   \Delta_{\dR^3} f_{\infty}^{2}  (\bx)=   \Delta_{\dR^3} f_{\infty}^{3}  (\bx)=   \Delta_{\dR^3} f_{\infty}^{4}  (\bx)  = 0,\ \forall \bm{x}\in\dR^3\setminus\{0^3\},
 \end{align}
 with the growth condition
 \begin{align}
(|f_{\infty}^{1}| + |f_{\infty}^{2}| + |f_{\infty}^{3}| + |f_{\infty}^{4}|)(\bm{x})
\leq C\Big(d_{g_0}(\bm{x},0^3) \Big)^{-\mu},\ \forall \bm{x}\in\dR^3\setminus\{0^3\}.
\end{align}
Since $\mu\in(0,\frac{1}{5})$,  Corollary~\ref{c:standard-Liouville-noncompact} implies
\begin{align}
f_{\infty}^{1} =f_{\infty}^{2} = f_{\infty}^{3}=f_{\infty}^{4} = 0 \ \text{on}\ \dR^3.
\end{align}
This contradicts to the property
\begin{equation}
\Big|\tilde{\fs}(\bx_j)^{\mu}\cdot\tilde{\eta}_j(\bx_j)\Big|=\|\tilde{\eta}_j\|_{C_{\mu}^0(\cK)}=1,
\end{equation}
which completes the proof of Case (b).

The rescaled limit in
Case (c) is a flat product space $\dR^2\times S^1$, so the proof in this region is almost identical to Case (b), details are omitted.

\begin{flushleft}
{\bf Sub-region $\cS_{\I_{\nu},3}$:}
\end{flushleft}
The large scale region $\cS_{\I_{\nu},3}$
can be divided into three cases of rescaled limits:
\begin{enumerate}
\item[(a)] the flat product $\dR^2\times S^1$,

\item[(b)] the Euclidean plane $\dR^2$,
\item[(c)] the compact space $(\PP^1, d_{ML}, \bx_{\infty})$, where $d_{ML}$ is the McLean metric.
\end{enumerate}
The proof are almost identical to the Sub-region $\cS_{\I_{\nu},2}$. Here we just point out the differences: In Case (b), the contradiction arises from the Liouville theorem
on the Euclidean plane, which is guaranteed
by Corollary~\ref{c:Liouville-plane}. The rescaled limit in Case (c) is compact, so we apply Proposition \ref{p:Liouville-Mclean} to obtain the desired contradiction.

\begin{flushleft}
{\bf Region $\cS_{\I_{\nu}^*}$ (singular fibers of type $\I_{\nu}^*$, $\nu\in\dZ_+$):}
\end{flushleft}
In this region, we have classified all the bubbles into the following $5$ types:
\begin{enumerate}
\item the Eguchi-Hanson space $(X^4_{EH}, g_{EH})$ and the Taub-NUT space $(\dC^2, g_{TN})$,
\item  $4$-dimensional flat orbifolds: $\dR^4/\dZ_2$ and $(\dR^3\times S^1)/\dZ_2$,
\item $3$-dimensional flat orbifolds: $\dR^3$, $\dR^3/\dZ_2$ and $(\dR^2\times S^1)/\dZ_2$
\item  $2$-dimensional flat orbifolds: $\dR^2/\dZ_2$,
\item  the compact space $(\PP^1, d_{ML}, \bx_{\infty})$, where $d_{ML}$ is the McLean metric.
\end{enumerate}

In Type (1), one can use Lemma~\ref{l:harmonic-1-form-Liouville} to obtain the contradiction. In Types (2)-(4), one can apply Corollary~\ref{c:standard-Liouville-noncompact} on the $\dZ_2$-covering space when necessary.
Finally, the contradiction in Type (5) follows from Proposition~\ref{p:Liouville-Mclean}.

This completes the proof of the proposition.
\end{proof}

\subsection{The proof of the existence theorem}

\label{ss:proof-of-existence}

In this subsection, we will complete the proof of the main existence theorem.
To start with, we restate Theorem~\ref{t:main-theorem} with more precise descriptions of choices of parameters, uniform estimates and bubbling behaviors.

\begin{theorem}\label{t:existence-hyperkaehler}
Let $\mu\in(0,\frac{1}{10})$, $\alpha\in(0,1)$, $\ell = \frac{11}{12}$ and $\frac{1}{\log(1/\delta)}\ll \ee_{\lambda} \ll  \frac{1}{\sqrt{\log(1/\delta)}}$ be parameters. Let $\fF:\cK\to \PP^1$
be any elliptic K3 surface with a fixed holomorphic $2$-form $\Omega$. Let  $g_\delta^C$ be the family of approximately hyperk\"ahler metrics with the error estimates in Proposition \ref{p:weighted-error} such that
\begin{equation}
(\cK,g_\delta^C)\xrightarrow{GH}(\PP^1,d_{ML}), \ \text{as}\ \delta\to0.
\end{equation}
Then for any $\delta\ll1$,
there exists a hyperk\"ahler metric $g_\delta^D$ induced by a hyperk\"ahler triple
$(\omega_\delta^D, \Rea(\delta\cdot\Omega), \Ima(\delta\cdot\Omega))$
such that the following properties hold:
\begin{enumerate}
\item Under the hyperk\"ahler metrics $g_\delta^D$,
$(\cK,g_\delta^D)$ are collapsing to $(\PP^1, d_{ML})$  with a finite singular set $\cS\subset \PP^1$
 such that curvatures of $g_\delta^D$ are uniformly bounded away from singular fibers, but are unbounded around singular fibers.

\item The hyperk\"ahler metrics $g_\delta^D$ satisfy the uniform weighted estimate
\begin{align}
\|g_\delta^D - g_\delta^C \|_{C_{0}^{0,\alpha}(\cK)} \ll \frac{1}{(\log(1/\delta))^{1/4}}, \label{e:error-deviation}\end{align}
where $C_{0}^{0,\alpha}$ norm means the weighted $C_{\mu'}^{k,\alpha}$ norm for $k=0$ and $\mu'=0$ as in Definition~\ref{d:weighted-space}.

\item
  If $\fF^{-1}(p)$ is singular with finite monodromy, then rescalings of $g_\delta^D$ converge to a complete hyperk\"ahler isotrivial ALG metric of asymptotic order at least $2$.

\item  When $\fF^{-1}(p)$ is singular of type $\I_{\nu}$ for some $\nu\in\dZ_+$, then rescalings of $g_\delta^D$ converge to $\nu$ copies of complete Taub-NUT metrics.

\item   When $\fF^{-1}(p)$ is singular of type $\I_{\nu}^*$ for some $\nu\in\dZ_+$, then rescalings of $g_\delta^D$
converge to $\nu$ copies of complete Taub-NUT metrics plus $4$ copies of Eguchi-Hanson metrics.

\end{enumerate}
\end{theorem}

\begin{remark}
We choose the parameters to ensure
\begin{align}
\max\Big\{\delta^{\ell(\frac{7}{5}+\mu)} +  \delta^{\ell(\mu-1)+2}, \delta^{\mu+1} \cdot \max_\lambda\ee_{\lambda}^{\mu+5} + (\max_\lambda\ee_{\lambda}\cdot \delta)^{\mu+3}\Big\} \ll \frac{(\delta\cdot \min_\lambda\ee_{\lambda}^2)^{\mu+1}}{(\log(1/\delta))^{1/4}}.\label{e:balancing}
\end{align}
We will see the application of this estimate in the proof of Theorem \ref{t:existence-hyperkaehler}. We remark that if there are no $\I_\nu^*$ fibers, the estimates may be improved to some polynomial rate in terms of $\delta$. However, for simplicity, we define $\max_\lambda\ee_{\lambda}=\min_\lambda\ee_{\lambda}=\frac{1}{(\log(1/\delta))^{3/4}}$ in that case so that our proof can be stated in a uniform way.
\end{remark}

\begin{remark}In the special case that all singular fibers are of Type $\I_{\nu}$ ($\nu\in\dZ_+$),
it is easy to see that the deviation  estimate \eqref{e:error-deviation} can be improved to
\begin{equation}
\|g_\delta^D - g_\delta^C \|_{C^k(\cK)} \leq C_k\cdot e^{-D_k/\delta}.
\end{equation}
See Item (2) of Proposition~\ref{p:weighted-error} and Lemma~\ref{l:implicit-function}. In this case, an alternative treatment of the higher order estimate can be found in \cite{JS}, which is based on a refined $C^2$-estimate for the K\"ahler potential compared with \cite[Lemma~5.3]{GW}.
\end{remark}

First, we will verify Property (1) in Lemma \ref{l:implicit-function} and we will prove that the linearized operator $\mathscr{L}_\delta$ is an isomorphism from $\mathfrak{A}$ to $\mathfrak{B}$.
\begin{proposition}
\label{p:injectivity-for-L} Let $(\cK,g_\delta^C)$ be a collapsing elliptic K3 surface with the family of approximately hyperk\"ahler metrics $g_\delta^C$, then there exists some constant $C>0$, independent of $\delta$, such that for every self-dual $2
$-form $\xi^+\in\mathfrak{B}
$, there exists a unique pair
$(\eta,\bar{\xi}^+)\in \mathfrak{A}$
such that \begin{equation}\mathscr{L}_\delta(\eta, \bar{\xi}^+) = \xi^+
\end{equation} and
\begin{equation}
\|\eta\|_{C_{\mu}^{1,\alpha}(\cK)}+\|\bar{\xi}^+\|_{C_{\mu+1}^{0,\alpha}(\cK)}\leq
 C  \|\xi^+\|_{C_{\mu+1}^{0,\alpha}(\cK)},\label{e:L-uniform-estimate}
\end{equation}
where  $\mu\in(0,\frac{1}{10})$ and $\alpha\in(0,1)$.\end{proposition}
\begin{proof}
First, the surjectivity of the linear operator $\mathscr{L}_\delta:\mathfrak{A}\to\mathfrak{B}$ immediately follows from the  standard Hodge theory. Indeed,
\begin{align}
\Omega^2_+(\cK) &=\mathcal{H}_+^2(\cK) \oplus d^+(\Omega^1(\cK))\\
 \Omega^1(\cK) &= d ( \Omega^0(\cK)) \oplus \mathring{\Omega}^1(\cK),
\end{align}
where $\mathring{\Omega}^1(\cK)$ denotes the space of divergence-free $1$-forms on $\cK$, therefore
\begin{equation}
\Omega^2_+(\cK) = \mathcal{H}_+^2(\cK) \oplus d^+(\mathring{\Omega}^1(\cK)).
\end{equation}
It follows that
\begin{equation}
\mathscr{L}_\delta = d^+\oplus \Id: \mathfrak{A} \longrightarrow \mathfrak{B}.
\end{equation}
is surjective.

Therefore the main part of the proof is to establish the uniform injectivity estimate \eqref{e:L-uniform-estimate}. By Proposition \ref{p:injectivity-of-D},
\begin{equation}
\|\eta\|_{C_{\mu}^{1,\alpha}(\cK)} \leq C \|d^+_{g_\delta^C}\eta\|_{C_{\mu+1}^{0,\alpha}(\cK)}=C\|\xi^+ - \bar{\xi}^+\|_{C_{\mu+1}^{0,\alpha}(\cK)}.
\end{equation}
So we only need to prove
\begin{equation}
\|\bar{\xi}^+\|_{C_{\mu+1}^{0,\alpha}(\cK)} \leq C\|\xi^+\|_{C_{\mu+1}^{0,\alpha}(\cK)}.
\end{equation}
Since the definite triple
\begin{equation}\bm{\omega}_\delta^C=(\omega_1,\omega_2,\omega_3)=(\omega_\delta^C,\Rea(\delta\cdot\Omega),\Ima(\delta\cdot\Omega))\end{equation} is self-dual harmonic and hence constitutes a basis of $\mathcal{H}^2_+(\cK)$ at every point, we can write
\begin{equation}
\bar{\xi}^+ = \lambda_1 \omega_1 + \lambda_2 \omega_2 + \lambda_3 \omega_3.
\label{e:proj}\end{equation}
It follows from the definition of the triple $\bm{\omega}_\delta^C$ that for every $1\leq p,q\leq 3$,
\begin{equation}
\omega_p\wedge\omega_q = Q_{pq}\dvol_{\bm{\omega}_\delta^C}. \end{equation}
By the error estimate in Proposition \ref{p:weighted-error}, the fact that the weight $\fs(\bx) \ge \delta \cdot \min_\lambda\ee_{\lambda}^2$ and the estimate (\ref{e:balancing}),
we have the estimate
\begin{equation}
\| Q_{pq}-\delta_{pq}\|_{C^0(\cK)} \ll \frac{1}{(\log(1/\delta))^{1/4}}.
\end{equation}
With respect to the K\"ahler metrics $\omega_\delta^C$, we have the volume estimate
\begin{equation}C^{-1}\delta^2\leq \Vol_{g_\delta}(\cK) \leq C \delta^2.\end{equation}
Therefore
\begin{align}
\mathscr{Q}_{pq} \equiv \int_{\cK} Q_{pq}\dvol_{\bm{\omega}_\delta^C}
\end{align}
has an inverse matrix whose norm is bounded by $C\delta^{-2}$.

On the other hand, we know that $\bar{\xi}^+$ is the $L^2(\cK)$ projection of $\xi^+$, so
\begin{equation}
\int_{\cK} \xi^+ \wedge \omega_q = \int_{\cK} \bar{\xi}^+ \wedge \omega_q = \sum_{p=1}^{3} \lambda_p \int_{\cK} Q_{pq}\dvol_{\bm{\omega}_\delta^C} = \sum_{p=1}^{3} \lambda_p \mathscr{Q}_{pq},
\end{equation}
for $q = 1,2,3$.  Therefore,
\begin{align}
\begin{split}
\Big(\sum_{q=1}^{3}|\lambda_q|^2\Big)^{1/2} &\leq C \delta^{-2} \cdot \Big(\sum_{q=1}^{3} (\int_{\cK} \xi^+ \wedge \omega_q)^2\Big)^{1/2} \\
&\leq C \delta^{-2} \cdot \big(\int_{\cK}\fs^{-\mu-1}\dvol_{\bm{\omega}_\delta^C}\big) \cdot \|\xi^+ \|_{C_{\mu+1}^{0,\alpha}(\cK)}\\
& \leq C\|\xi^+ \|_{C_{\mu+1}^{0,\alpha}(\cK)}
\end{split}
\end{align}
because $\|\omega_q\|_{C^0(\cK)}\le C$ and $\int_{\cK}\fs^{-\mu-1}\dvol_{\bm{\omega}_\delta^C}\le C\delta^2$.

Next,
\begin{align}
\begin{split}
\Vert \bar{\xi}^+ \Vert_{C_{\mu+1}^{0,\alpha}(\cK)}
&= \Vert \lambda_{1} \omega_1 +\lambda_{2} \omega_2 + \lambda_{3} \omega_3
\Vert_{C_{\mu+1}^{0,\alpha}(\cK)}\\
& \leq |\lambda_{1}|\cdot\Vert \omega_1\Vert_{C_{\mu+1}^{0,\alpha}(\cK)}
+ |\lambda_{2}|\cdot \Vert \omega_2\Vert_{C_{\mu+1}^{0,\alpha}(\cK)}
+ |\lambda_{3}| \cdot\Vert \omega_3\Vert_{C_{\mu+1}^{0,\alpha}(\cK)}
.
\end{split}
\end{align}
Since for every $1 \leq q \leq 3$,
\begin{align}
\Vert \omega_q\Vert_{C_{\mu+1}^{0,\alpha}(\cK)} \leq C,\end{align}
the above implies that
\begin{align}
\Vert \bar{\xi}^+\Vert_{C_{\mu+1}^{0,\alpha}(\cK)} \leq C \Vert \xi^+\Vert_{C_{\mu+1}^{0,\alpha}(\cK)},
\end{align}
and the proof of the proposition is done.

 \end{proof}

Next we prove a uniform weighted estimate
for the nonlinear term $\mathscr{N}_{\delta}$, which
 is given by the following elementary calculations.

\begin{proposition}
[Nonlinear Errors]\label{p:nonlinear-errors} Given the collapsing sequence $(\cK,g_\delta^C)$, then there exists some constant $C>0$, independent of $\delta$, such that for every
$v_1\equiv(\eta_1,\bar{\xi}_1^+) \in B_r(0)\subset\mathfrak{A}$ and $v_2\equiv(\eta_2,\bar{\xi}_2^+) \in B_r(0)\subset\mathfrak{A}$,
where $r=\frac{(\delta\cdot \min_\lambda\ee_{\lambda}^2)^{\mu+1}}{(\log(1/\delta))^{1/4}}$, we have
\begin{equation}
\| \mathscr{N}_{\delta}(v_1) - \mathscr{N}_{\delta}(v_2) \|_{\mathfrak{B}} \leq C \cdot (\delta\cdot \min_\lambda\ee_{\lambda}^2)^{-(\mu+1)}\cdot (\|v_1\|_{\mathfrak{A}} + \|v_2\|_{\mathfrak{A}} ) \cdot \|v_1 - v_2\|_{\mathfrak{A}}.\label{e:nonlinear-error}
\end{equation}

\end{proposition}

\begin{proof}
By definition, for any $v\equiv(\omega,\bar{\xi}^+)$,
\begin{align}
\mathscr{N}_\delta(v) \equiv &  \ \mathfrak{H}_0\Big(\frac{1}{2}(\omega_2^2-(\omega_\delta^C)^2, -\omega_\delta^C\wedge\omega_2, -\omega_\delta^C\wedge\omega_3)\Big)
 \nonumber\\
 & - \mathfrak{H}_0\Big(\frac{1}{2}(\omega_2^2-(\omega_\delta^C)^2-d^- \eta \wedge d^- \eta, -\omega_\delta^C\wedge\omega_2, -\omega_\delta^C\wedge\omega_3)\Big).
\end{align}
Since $\mathfrak{H}_0:\cO\subset \dR^3\to \Lambda_+(\cK)$ is smooth, so there is some universal constant $C>0$ such that
\begin{align}
\begin{split}
 |\mathscr{N}_\delta(v_1) - \mathscr{N}_\delta(v_2) |
&\leq   C |
 d^{-}\eta_1 *  d^{-}\eta_1   - d^{-}\eta_2 *  d^{-}\eta_2| \\
&\leq C ( |d^{-}\eta_1 | + |d^{-}\eta_2|) \cdot |d^{-}(\eta_1 - \eta_2)|.
\end{split}
\end{align}
Multiplying by the weight function $\fs(\bx)^{\mu +1}$,
\begin{align}
\fs(\bx)^{\mu +1}\cdot | \mathscr{N}_\delta(v_1) - \mathscr{N}_\delta(v_2)|
\leq  C\cdot \fs(\bx)^{\mu +1} \cdot ( |d^{-}\eta_1| + |d^{-}\eta_2|) \cdot |d^{-}(\eta_1 - \eta_2)|.
\end{align}
Since the weight function $\fs(\bx)^{\mu +1}$ has a minimum $(\delta\cdot \min_\lambda\ee_{\lambda}^2)^{\mu+1}$,
\begin{align}
&\fs(\bx)^{\mu +1}\cdot | \mathscr{N}_\delta(v_1) - \mathscr{N}_\delta(v_2)|
\nonumber \\
 \leq & C\cdot (\delta\cdot \ee_{\lambda}^2)^{-(\mu+1)} \Big(\fs(\bx)^{\mu +1}\cdot ( |d^-\eta_1| + |d^-\eta_2|)\Big)\cdot \Big(
 \fs(\bx)^{\mu +1}\cdot | d^- (\eta_1 - \eta_2)|\Big).
\end{align}
Taking sup norms,
\begin{align}
&\Vert \mathscr{N}_\delta(v_1) - \mathscr{N}_\delta(v_2) \Vert_{ C^0_{\mu+1}(\cK)}\nonumber\\
\leq& C\cdot (\delta\cdot \ee_{\lambda}^2)^{-(\mu+1)}\cdot \Big(  \Vert v_1 \Vert_{ C^1_{\mu}(\cK)}  +
\Vert v_2 \Vert_{ C^1_{\mu}(\cK)} \Big)\cdot \Big(
\Vert v_1 - v_2 \Vert_ { C^1_{\mu}(\cK)}\Big).
\end{align}
By similar computations,
we also have the estimate for the H\"older seminorm
\begin{align}
& \Big[ \mathscr{N}_\delta(v_1) - \mathscr{N}_\delta(v_2) \Big]_{ C^{0,\alpha}_{\mu +1}(\cK)}
\nonumber\\
\leq  & C \cdot (\delta\cdot \ee_{\lambda}^2)^{-(\mu+1)} \cdot \Big(  \Vert v_1 \Vert_{ C^{1,\alpha}_{\mu}(\cK)}  +
\Vert v_2 \Vert_{ C^{1,\alpha}_{\mu}(\cK)} \Big)\cdot \Big(
\Vert v_1 - v_2 \Vert_ { C^{1,\alpha}_{\mu}(\cK)}\Big).
\end{align}
So we obtain the effective estimate \eqref{e:nonlinear-error} for the nonlinear errors.
\end{proof}

With the above preparations, we are ready to complete the proof of Theorem \ref{t:existence-hyperkaehler}.
\begin{proof}
[Proof of Theorem \ref{t:existence-hyperkaehler}]

We start with the metrics $g_\delta^C$ induced by the approximately hyperk\"ahler triples \begin{equation}
\bm{\omega}_\delta^C=(\omega_\delta^C, \Rea(\delta\cdot\Omega), \Ima(\delta\cdot\Omega))
\end{equation}
on the K3 surface $\cK$ with a fixed homolomorphic $2$-form $\Omega$.
Then we will prove the existence of a genuine hyperk\"ahler triple \begin{equation}
\bm{\omega}_\delta^D = (\omega_\delta^D,\Rea(\delta\cdot\Omega), \Ima(\delta\cdot\Omega)).
\end{equation}
This will be accomplished by applying Lemma \ref{l:implicit-function} to perturb the approximate solutions, which requires us to combine all the uniform weighted estimates obtained in this section. Proposition \ref{p:injectivity-for-L} gives the isomorphism and uniform weighted estimate for the linearized operator $\mathscr{L}_\delta$.
Next, Item (2) of Lemma \ref{l:implicit-function} holds by (\ref{e:balancing}) for our choice of parameters.
Therefore, applying Lemma \ref{l:implicit-function}, the 2-form $\omega_\delta^C$ can be perturbed to $\omega_\delta^D$
such that the triple  $(\omega_\delta^D,\Rea(\delta\cdot\Omega), \Ima(\delta\cdot\Omega))$
is a hyperk\"ahler triple.
Moreover, the implicit function theorem
also gives the uniform
error estimate in the weighted H\"older space
\begin{align}
\|\omega_\delta^D - \omega_\delta^C \|_{C_{\mu+1}^{0,\alpha}(\cK)} \ll \frac{(\delta\cdot \min_\lambda\ee_{\lambda}^2)^{\mu+1}}{(\log(1/\delta))^{1/4}}. \end{align}

Next, we will analyze the Gromov-Hausdorff behaviors of $(\cK, g_\delta^D)$. Remark that the weight $\fs(\bx) \ge \delta \cdot \min_\lambda\ee_{\lambda}^2$. So the above error estimate immediately implies that \begin{align}
\|\omega_\delta^D - \omega_\delta^C \|_{C^{0,\alpha}_0(\cK)} \ll \frac{1}{(\log(1/\delta))^{1/4}}. \label{e:hoelder-estimate}
\end{align}

By the definition of $g_\delta^C$ and $g_\delta^D$, it is easy to see that
\begin{align}
\|g_\delta^D - g_\delta^C \|_{C^{0,\alpha}_0(\cK)} \ll \frac{1}{(\log(1/\delta))^{1/4}}. \label{e:hoelder-estimate-metric}
\end{align}
By the uniform estimate \eqref{e:hoelder-estimate-metric}, we conclude that, with respect to the hyperk\"ahler metrics $g_\delta^D$, all the bubbles around singular fibers coincide with the corresponding bubbles occurring in the collapsing of $g_\delta^C$. By Section \ref{s:metric-geometry}, it follows that the deepest bubbles around singular fibers are precisely given in the statement of the theorem. Then Theorem \ref{t:maximal-nilpotent-rank} for Einstein metrics implies that curvatures are uniformly bounded away from singular fibers. This completes the proof of the theorem.
\end{proof}

\section{Remarks on moduli}
\label{s:moduli-space}
In this section, we give a count of the parameters involved in our construction.
\subsection{ALG moduli}
In this subsection, we compute the dimension of the moduli space of isotrivial
ALG metrics which are ALG of at least order $2$ near a fixed such space $\cG$.
\begin{definition}[Isotrivial ALG moduli]
\label{d:IALGM} Given an isotrivial ALG manifold $\cG$ with holomorphic (2,0)-form $\Omega^{\cG}=\omega_2^{\cG}+ \sq \omega_3^{\cG}$ which is identified with $d\mathscr{U}\wedge d\mathscr{V}$ on the model space $\cC_{\beta,\tau}$ outside the central fiber (see Lemma \ref{l:hol2formg}), define $\mathcal{U}$ as the space of closed 2-forms $\omega_1$ on $\cG$ such that
  \begin{align}
  (\omega_1)^2=(\omega_2^{\cG})^2, \omega_1\wedge\omega_2^{\cG}=\omega_1\wedge\omega^{\cG}_3=0,
\end{align}
and
\begin{align}
|\nabla^k_{h^{\FF}}(\omega_1-\omega^{\FF})|_{h^{\FF}}=O(|\mathscr{U}|^{-k-2}),
  \end{align}
as $|\mathscr{U}| \to \infty$, for any $k \in \mathbb{N}$.
\end{definition}

We now fix $\omega^{\cG}\in \mathscr{U}$ and study the neighborhood of $\omega^{\cG}$ in $\mathscr{U}$. Let $g^{\cG}$ be the metric induced by the hyperK\"ahler triple $(\omega^{\cG},\omega_2^{\cG},\omega_3^{\cG})$.

The weight function on $\cG$ is
\begin{align}
\rho(\bx)
=
\begin{cases}
1, &   d(p,\bm{x}) \leq 1,
\\
d(p,\bx), &  d(p,\bm{x})\geq 2,\end{cases},
\end{align}
smoothly extended to $\cG$, and the weighted $W^{k,2}$ norm of a tensor $\eta$ is defined by
\begin{equation}
\|\eta\|_{L^2_{\mu}(\cG)}\equiv\int_{\cG}|\eta|^2\rho^{-2+2\mu}, \quad   \|\eta\|_{W^{k,2}_{\mu}(\cG)}\equiv\sum_{m=1}^{k}\|\nabla^m\eta\|_{L^2_{\mu+m}(\cG)}^2.
\end{equation}

\begin{proposition}
\label{indroots}
The indicial roots of $\Delta_{g^{\cG}}$ on $\dT^2$-invariant functions are $\lambda_j = \beta^{-1} j, j \in \ZZ$.  The indicial roots of $\Delta_{g^{\cG}}$ on $\dT^2$-invariant $(1,1)$-forms are given by
\begin{enumerate}
	\item   $\lambda_j = \beta^{-1} j, j \in \ZZ$,
	\item   $\lambda_j = \beta^{-1} j \pm 2, j \in \ZZ$.
\end{enumerate}
Furthermore, for $\mu$ not equal to an indicial root, then $\Delta_{g^{\cG}}:W^{k,2}_{\mu}\rightarrow W^{k-2,2}_{\mu+2}$ is Fredholm.
\end{proposition}
\begin{proof} For functions, this follows directly from Proposition~\ref{p:Liouville-sector} with $\sigma = 0$.	For $(1,1)$-forms, we write a harmonic $(1,1)$-form as
\begin{align}
\begin{split}
\xi=&f_1(|\mathscr{U}|)\phi_1(\arg \mathscr{U})d \mathscr{U}\wedge d\mathscr{\bar U}+f_2(|\mathscr{U}|)\phi_2(\arg \mathscr{U})d \mathscr{U}\wedge d\mathscr{\bar V} \\
& + f_3(|\mathscr{U}|)\phi_3(\arg \mathscr{U})d \mathscr{V}\wedge d\mathscr{\bar U}+f_4(|\mathscr{U}|) \phi_4(\arg \mathscr{U})d \mathscr{V}\wedge d\mathscr{\bar V},
\end{split}
\end{align}
and note that $\xi$ is well-defined on the model space if and only if
\begin{align}
\begin{split}
\phi_1(2\pi\beta)&=\phi_1(0),  \ \phi_2(2\pi\beta)=e^{-4\pi\beta\i}\phi_2(0),\\
\phi_3(2\pi\beta)&=e^{4\pi\beta\i}\phi_3(0), \  \phi_4(2\pi\beta)=\phi_4(0).
\end{split}
\end{align}
By Proposition~\ref{p:Liouville-sector}, we get the first class of indicial roots for $\phi_1, \phi_4$, and the second class
of indicial roots for $\phi_2, \phi_3$.

The Fredholm property is proved in \cite[Proposition~16]{HHM}. We also note that the estimate
\begin{align}
\Vert \phi \Vert_{W^{2,2}_{\mu}(\cG)} \leq C \Vert \Delta_{g^{\cG}} \phi \Vert_{L^2_{\mu + 2}(\cG)}
+ \Vert \phi \Vert_{L^2(B_R)},
\end{align}
for some $R > 0$, and $\mu$ non-indicial is proved in \cite[Theorem~4.11]{CCI}.
The Fredholm property then follows from this in a standard fashion, see \cite{Bartnik, MinerbeALF}.
\end{proof}

Define $L^2\mathcal{H}^2$ as the space of 2-forms $\xi\in L^2$ such that $d\xi=0$ and $d^*\xi=0$.
\begin{proposition}
\label{l2thm}
There exists a small neighborhood $U \subset \mathcal{U}$ around $\omega^{\cG}$ such that
there exists an isomorphism from $U$ onto a small ball in $L^2 \mathcal{H}^2$.
\end{proposition}
\begin{proof}
First, let $\mu > 0$ be sufficiently close to zero. By elliptic regularity of distributions and Proposition \ref{indroots},
the cokernel of the mapping
\begin{align}
\label{lapmap}
\Delta_{g^{\cG}}: W^{k,2}_{\mu} \rightarrow W^{k-2,2}_{\mu+2},
\end{align}
can be identified with harmonic functions in $W^{k,2}_{-\mu}$. So let $h \in  W^{k,2}_{-\mu}$, then
\begin{align}h = c_1 \log|\mathscr{U}| + c_2 + O(|\mathscr{U}|^{-\epsilon})\end{align} as $\mathscr{U} \to \infty$. By integration by parts, we see that $c_1 = 0$, and therefore $h = c_2$ is a constant. The mapping in \eqref{lapmap} therefore
has a $1$-dimensional cokernel spanned by the constants. To overcome this, let $\chi(|\mathscr{U}|)$ be a smooth function which is $1$ when $|\mathscr{U}|\ge 2R$ and is $0$ when $|\mathscr{U}|\le R$ for a large radius $R$. The mapping
\begin{equation}
\Delta_{g^{\cG}}: W^{k,2}_{\mu}\oplus \RR (\chi(|\mathscr{U}|)\log|\mathscr{U}|)\rightarrow W^{k-2,2}_{\mu+2},
\end{equation}
is then surjective since the element $\Delta_{g^{\cG}} ( \chi(|\mathscr{U}|)\log|\mathscr{U}|)$ pairs nontrivially with $1$ under the $L^2$ pairing.
We claim that this mapping is moreover an isomorphism.	
To see this,  suppose that $c\in\RR$ and $f\in W^{k,2}_{\mu}$ satisfies \begin{equation}
\Delta_{g^{\cG}}(f+c(\chi(|\mathscr{U}|)\log|\mathscr{U}|)=0.
\end{equation} By elliptic regularity, we can assume that $k$ is very large. Then
\begin{equation}
0=\int_{|\mathscr{U}|\le R} \Delta_{g^{\cG}}(f+c(\chi(|\mathscr{U}|)\log|\mathscr{U}|)=\int_{|\mathscr{U}|=R} \frac{\partial}{\partial n}(f+c(\chi(|\mathscr{U}|)\log|\mathscr{U}|),
\end{equation}
where $\frac{\partial}{\partial n}$ means the derivative in the normal direction. So $c=0$. By maximum principle, $f=0$.
	
The condition $\omega_1\wedge\omega_2^{\cG}=\omega_1\wedge\omega_3^{\cG}$ is the same as $\omega_1$ being $(1,1)$. For any $\omega^{\cG}\in\mathcal{U}$, let $\mathcal{V}$ as the linear space
\begin{equation}
\{\xi\in\Lambda^{1,1}(\cG), d\xi=0, \xi\wedge\omega^{\cG}=0, |\nabla^k\xi|=O((|\mathscr{U}|+1)^{-k-2}), k\in\mathbb{N}\}
\end{equation}
using the norm defined by the metric $g^{\cG}$ induced by $(\omega^{\cG},\omega_2^{\cG},\omega_3^{\cG})$. Then for any small enough $\xi\in \mathcal{V}$, by a standard application of the implicit function theorem (Lemma \ref{l:implicit-function}), there exists a unique $\varphi\in W^{k,2}_{\mu}\oplus \RR (\chi(|\mathscr{U}|)\log|\mathscr{U}|)$ such that \begin{equation}
(\omega^{\cG}+\xi+\sqrt{-1}\partial\bar\partial\varphi)^2=(\omega_2^{\cG})^2.
\end{equation} By elliptic regularity, $\omega^{\cG}+\xi+\sqrt{-1}\partial\bar\partial\varphi\in\mathcal{U}$. This provides a local diffeomorphism from a small neighborhood of 0 in $\mathcal{V}$ to a small neighborhood of $\omega^{\cG}$ in $\mathcal{U}$.

Next, we will show that $L^2\mathcal{H}^2=\mathcal{V}$. For any element $\xi\in\mathcal{V}$, we know that $\xi$ is anti-self-dual. So \begin{equation}
d_{g^{\cG}}^*\xi=-*_{g^{\cG}}\circ d\circ *_{g^{\cG}}\xi=*_{g^{\cG}}\circ d\xi=0.
\end{equation}
Therefore $\mathcal{V}\subset L^2\mathcal{H}^2$. On the other hands, for any element $\xi \in L^2\mathcal{H}^2$, $\Delta_{g^{\cG}}\xi=0$. So by K\"ahler identities (for example, see the proof of \cite[Theorem~5.1]{CCIII}), the coefficient of the self-dual part of $\xi$ is also harmonic. It must vanish by maximal principle. In other words, $\xi$ is anti-self-dual. Since $L^2=L^2_{-1}$, we just need to analyze the indicial roots between~$-2$ and~$-1$. By \cite[Proposition~16]{HHM} (see also \cite[Theorem~4.6]{CCII} for a similar result), we have existence of harmonic expansions. For $\beta < 1/2$, there is no indicial root in the above range, so we are done. For $\beta > 1/2$, from Proposition \ref{indroots}, we have the harmonic leading terms
\begin{align}
\label{lt1}
\mathscr{U}^{-1/\beta}d \mathscr{U}\wedge d\mathscr{\bar U}, \ \mathscr{U}^{-1/\beta}d \mathscr{V}\wedge d\mathscr{\bar V},
\end{align}
and their conjugates. However, we require that $\xi$ is not only harmonic, but moreover closed and coclosed because we can do the integration by parts. It is clear that the first term in \eqref{lt1} is not coclosed, and the second term in \eqref{lt1} is not closed. Consequently, there is no non-trivial linear combination of these $4$ leading terms with is both closed and co-closed. So the leading term of $\xi$ can not be of order $-1/\beta$. In the cases $\beta = 2/3$ or $\beta = 3/4$, the next indicial root is $-2$, so we are done in these cases. In the case $\beta = 5/6$, the next indicial root is $-8/5$.  The corresponding leading term is
\begin{align}
(\bar{\mathscr{U}})^{-8/5} d \mathscr{U} \wedge d \bar{\mathscr{V}},
\end{align}
or its conjugate. Any non-trivial linear combination of these terms is not closed, so the leading
term cannot be of order $-8/5$. The next indicial root is $-2$, so we are done in this case as well.

We have shown that in all cases, the leading term of $\xi$ must be of order at least $-2$. By elliptic regularity, $\xi\in\mathcal{V}$, and this finishes the proof.

\end{proof}

\begin{theorem}The isotrivial ALG moduli space has dimension $\dim(\mathcal{U}) = b_2(\cG) -1$.
\end{theorem}
\begin{proof}
Recall that from \cite{HHM}, there is an isomorphism
\begin{align}
L^2 \mathcal{H}^2 \cong  Im (  H^2_{cpt}(\cG)  \rightarrow  H^2(\cG) ),
\end{align}
so by Proposition~\ref{l2thm}, we just need to compute the dimension of the right hand side,
which is a topological invariant. Note that  $\cG$ deformation retracts onto a large ball $B \subset \cG$, and set
$S = \partial( B_p(R))$. Consider $B$ as a compact manifold with boundary $S$.
We claim that
\begin{align}
\label{Scoho}
H^1(S; \RR) \cong H^2(S; \RR) \cong \RR.
\end{align}
To see this, recall the action is $(\mathscr{U},\mathscr{V}) \mapsto (e^{2 \pi \i \beta}\mathscr{U}, e^{-2 \pi \i \beta }\mathscr{V})$,
so on $S^1 \times T^2$, the action is a rotation on the first factor.
Then it is easy to see that the pull-back of $d \theta$ from $S^1$ is the only invariant harmonic $1$-form under this action. Since the action is free, we can identify the cohomology of the quotient with the invariant cohomology.
	
The long exact sequence in relative cohomology, with real coefficients gives
\begin{align}
0 \rightarrow H^1(S) \rightarrow H^2(B, S) \rightarrow  H^2(B) \rightarrow H^2(S) \rightarrow 0.
\end{align}
This is because by Poincar\'e-Lefschetz duality for manifolds with boundary,
\begin{equation}
H^3(B,S) \cong H_1(B) \cong H_1(\cG) = 0
\end{equation}
and $H^1(B) \cong H^1(\cG) = 0$ since $\cG$ deformation retracts onto $B$. Also by duality we have \begin{equation}
H^2(B,S) \cong H_2(B) \cong H_2(\cG) \cong H^2_{cpt}(\cG).
\end{equation}
So the above exact sequence can be written as
\begin{align}
0 \rightarrow \RR \rightarrow H^2_{cpt}(\cG)  \rightarrow  H^2(\cG) \rightarrow \RR  \rightarrow 0.
\end{align}
Consequently, we have
\begin{align}
\dim \{   Im (  H^2_{cpt}(\cG)  \rightarrow  H^2(\cG) ) \} = b_2(\cG) -1.
\end{align}
\end{proof}

\subsection{Parameter count}
Our family of Calabi-Yau metrics depends upon the following parameters. For each
ALG space $\mathcal{G}_i$,  we have an open set $\mathcal{U}_i$ of dimension  $b_2( \mathcal{G}_i) -1$ as discussed in the previous subsection. Note that these parameters correspond to the area of the holomorphic curves in the finite singular fiber of each ALG space, of which there are $b_2( \mathcal{G}_i)$, but we subtract $1$ since the area of the singular fiber is fixed to be $\delta^2$ (after scaling).
For each $I_{\nu}$ fiber, we can also parametrize the multi-Ooguri-Vafa
metric by $\mathcal{V}_i$ which is an open set in $\RR^{\nu_i-1}$, which corresponds to
the areas of the holomorphic curves, minus $1$ constraint. We can vary these parameters by moving the monopole
points along the $S^1$ direction.
For the $I_{\nu'}^*$ fibers, we have an open set
$\mathcal{W}_i$ in $\RR^{\nu'_i + 4}$ given by the areas of the holomorphic curves
of multiplicity two (of which there are $\nu'_i +1$), subtracting $1$ constraint, together with the areas of the $(-2)$-curves in each Eguchi-Hanson metric (of which there are $4$).

Recall that $k_1$ denote the number of fibers with finite monodromy, $k_2$ denote the number of
$I_{\nu}$ fibers, and $k_3$ denote the number of $I_{\nu'}^*$ fibers.
The K\"ahler cone $\mathcal{H}(\cK) \subset H^{1,1}_{\RR}(\cK) \cong \RR^{20}$ is a convex cone.
By taking the K\"ahler class $[\omega_\delta^D]$
of the Calabi-Yau metric obtained in Theorem~\ref{t:existence-hyperkaehler}, we have a mapping
\begin{align}
\Phi: I \times  \mathscr{B} \times \Big( \prod_{i = 1}^{k_1} \mathcal{U}_i\Big) \times \Big( \prod_{i=1}^{k_2} \mathcal{V}_i\Big)
\times \Big(  \prod_{i=1}^{k_3} \mathcal{W}_i \Big) \rightarrow \mathcal{H}(\cK),
\end{align}
where $\delta \in I = (0, \delta_0)$, for $\delta_0$ sufficiently small, and
the space $\mathscr{B}$ is defined in \eqref{Bdef} above.
Note that, $\dim( \mathcal{U}_i) = b_2( \mathcal{G}_i) - 1$, $\dim ( \mathcal{V}_i) = \nu_i - 1$
and $\dim (\mathcal{W}_i) = \nu'_i + 4$, so
\begin{align}
\sum_{i=1}^{k_1} \dim( \mathcal{U}_i) + \sum_{i=1}^{k_2} \dim ( \mathcal{V}_i)
+ \sum_{i=1}^{k_3} \dim (\mathcal{W}_i)
= 24 - 2k_1 -k_2 - 2 k_3.
\end{align}
Combining this with \eqref{dimB} from above, we conclude that the domain of $\Phi$ is $20$-dimensional, which is the same dimension as $\mathcal{H}(\cK)$.

\begin{remark} We expect that  $[\omega_\delta^D] \to [\omega_E]$ as $\delta \to 0$,
where $[\omega_E]$ denotes the Poincar\'e dual of a fiber, and the image of $\Phi$ is an open set in $\mathcal{H}(\cK)$. This is intuitively clear from our construction, but a detailed analysis of this is very lengthy, so we do not include this here.
\end{remark}

\subsection{More bubble limits}
\label{ss:more-bubbles}
We can find other possible bubble limits which occur by slightly changing the gluing data.
First, near an $\I_{\nu}$-fiber ($\nu\in\dZ_+$), recall that the singularity model
comes from the multi-Ooguri-Vafa
metrics. In this case, we can change the locations of the monopoles so that they cluster together at points, so that one can also see multi-Taub-NUT ALF-$A_k$ metrics instead of having $\nu$ copies of Taub-NUT bubbles. It is also possible to obtain nontrivial bubble tree structure. For instance, one can make the monopole points cluster together at different rates, which can give an ALF orbifold with an orbifold point and the deepest bubble is given by a multi-Taub-NUT metric. As mentioned above in Remark~\ref{d2remark}, near a singular fiber of Type $\I_{\nu}^*$ ($\nu\in\dZ_+$), one can also change the scales of the Eguchi-Hanson bubbles so that
the ALF-$D_2$ type bubbles appear, which is identified with the
resolution of $(\dR^3\times S^1)/\dZ_2$. See \cite{BM} for more details about the Kummer construction of the ALF-$D_2$ space. Moreover, if the monopole points approach the ALF-$D_2$ space, it is also possible to get ALF-$D_k$ spaces for larger $k$. See \cite{CCII} and the references therein to see more details about ALF-$D_k$ hyperk\"ahler 4-manifolds.

\bibliographystyle{amsalpha}

\bibliography{CVZ_references}

\providecommand{\bysame}{\leavevmode\hbox to3em{\hrulefill}\thinspace}
\providecommand{\MR}{\relax\ifhmode\unskip\space\fi MR }
% \MRhref is called by the amsart/book/proc definition of \MR.
\providecommand{\MRhref}[2]{%
  \href{http://www.ams.org/mathscinet-getitem?mr=#1}{#2}
}
\providecommand{\href}[2]{#2}
\begin{thebibliography}{GSVY90}

\bibitem[AKL89]{AKL}
Michael~T. Anderson, Peter~B. Kronheimer, and Claude LeBrun, \emph{Complete
  {R}icci-flat {K}\"ahler manifolds of infinite topological type}, Comm. Math.
  Phys. \textbf{125} (1989), no.~4, 637--642.

\bibitem[Bar86]{Bartnik}
Robert Bartnik, \emph{The mass of an asymptotically flat manifold}, Comm. Pure
  Appl. Math. \textbf{39} (1986), no.~5, 661--693.

\bibitem[BB04]{BB}
Olivier Biquard and Philip Boalch, \emph{Wild non-abelian {H}odge theory on
  curves}, Compos. Math. \textbf{140} (2004), no.~1, 179--204.

\bibitem[BM11]{BM}
Olivier Biquard and Vincent Minerbe, \emph{A {K}ummer construction for
  gravitational instantons}, Comm. Math. Phys. \textbf{308} (2011), no.~3,
  773--794.

\bibitem[BR75]{BurnsRapoport}
Dan Burns, Jr. and Michael Rapoport, \emph{On the {T}orelli problem for
  k\"{a}hlerian {$K-3$} surfaces}, Ann. Sci. \'{E}cole Norm. Sup. (4)
  \textbf{8} (1975), no.~2, 235--273.

\bibitem[CC97]{ChC1}
Jeff Cheeger and Tobias~H. Colding, \emph{On the structure of spaces with
  {R}icci curvature bounded below. {I}}, J. Differential Geom. \textbf{46}
  (1997), no.~3, 406--480.

\bibitem[CC15a]{CCI}
Gao {Chen} and Xiuxiong {Chen}, \emph{Gravitational instantons with faster than
  quadratic curvature decay ({I})}, arXiv.org:1505.01790, 2015.

\bibitem[CC15b]{CCII}
\bysame, \emph{Gravitational instantons with faster than quadratic curvature
  decay ({II})}, arXiv.org:1508.07908 (to appear in Journal f\"ur die reine und
  angewandte Mathematik), 2015.

\bibitem[CC16]{CCIII}
\bysame, \emph{Gravitational instantons with faster than quadratic curvature
  decay ({III})}, arXiv.org:1603.08465, 2016.

\bibitem[CK02]{CherkisKapustinALG}
Sergey~A. Cherkis and Anton Kapustin, \emph{Hyper-{K}\"{a}hler metrics from
  periodic monopoles}, Phys. Rev. D (3) \textbf{65} (2002), no.~8, 084015, 10.

\bibitem[CT06]{CheegerTian}
Jeff Cheeger and Gang Tian, \emph{Curvature and injectivity radius estimates
  for {E}instein 4-manifolds}, J. Amer. Math. Soc. \textbf{19} (2006), no.~2,
  487--525.

\bibitem[Don06]{Donaldson}
Simon~K. Donaldson, \emph{Two-forms on four-manifolds and elliptic equations},
  Inspired by {S}. {S}. {C}hern, Nankai Tracts Math., vol.~11, World Sci.
  Publ., Hackensack, NJ, 2006, pp.~153--172.

\bibitem[Don12]{Donaldson-kummer}
\bysame, \emph{Calabi-{Y}au metrics on {K}ummer surfaces as a model gluing
  problem}, Advances in geometric analysis, Adv. Lect. Math. (ALM), vol.~21,
  Int. Press, Somerville, MA, 2012, pp.~109--118.

\bibitem[EH79]{EguchiHanson}
Tohru Eguchi and Andrew~J. Hanson, \emph{Self-dual solutions to {E}uclidean
  gravity}, Ann. Physics \textbf{120} (1979), no.~1, 82--106.

\bibitem[FLS17]{FineLotaySinger}
Joel Fine, Jason~D. Lotay, and Michael Singer, \emph{The space of hyperk\"ahler
  metrics on a 4-manifold with boundary}, Forum Math. Sigma \textbf{5} (2017),
  e6, 50.

\bibitem[FM94]{FriedmanMorgan}
Robert Friedman and John~W. Morgan, \emph{Smooth four-manifolds and complex
  surfaces}, Ergebnisse der Mathematik und ihrer Grenzgebiete (3), vol.~27,
  Springer-Verlag, Berlin, 1994.

\bibitem[Fos19]{Foscolo}
Lorenzo Foscolo, \emph{A{LF} gravitational instantons and collapsing
  {R}icci-flat metrics on the {$K3$} surface}, J. Differential Geom.
  \textbf{112} (2019), no.~1, 79--120.

\bibitem[FY92]{FY}
Kenji Fukaya and Takao Yamaguchi, \emph{The fundamental groups of almost
  non-negatively curved manifolds}, Ann. of Math. (2) \textbf{136} (1992),
  no.~2, 253--333.

\bibitem[Gro99]{Gross1999}
Mark Gross, \emph{Special {L}agrangian fibrations. {II}. {G}eometry. {A} survey
  of techniques in the study of special {L}agrangian fibrations}, Surveys in
  differential geometry: differential geometry inspired by string theory, Surv.
  Differ. Geom., vol.~5, Int. Press, Boston, MA, 1999, pp.~341--403.

\bibitem[GSVY90]{GSVY}
Brian~R. Greene, Alfred Shapere, Cumrun Vafa, and Shing-Tung Yau, \emph{Stringy
  cosmic strings and noncompact {C}alabi-{Y}au manifolds}, Nuclear Phys. B
  \textbf{337} (1990), no.~1, 1--36.

\bibitem[GT83]{GilbargTrudinger}
David Gilbarg and Neil~S. Trudinger, \emph{Elliptic partial differential
  equations of second order}, second ed., Grundlehren der Mathematischen
  Wissenschaften, vol. 224, Springer-Verlag, Berlin, 1983.

\bibitem[GTZ16]{GTZ}
Mark Gross, Valentino Tosatti, and Yuguang Zhang, \emph{Gromov-{H}ausdorff
  collapsing of {C}alabi-{Y}au manifolds}, Comm. Anal. Geom. \textbf{24}
  (2016), no.~1, 93--113.

\bibitem[GW00]{GW}
Mark Gross and P.~M.~H. Wilson, \emph{Large complex structure limits of {$K3$}
  surfaces}, J. Differential Geom. \textbf{55} (2000), no.~3, 475--546.

\bibitem[Haw77]{Hawking}
S.~W. Hawking, \emph{Gravitational instantons}, Phys. Lett. A \textbf{60}
  (1977), no.~2, 81--83.

\bibitem[Hei12]{Hein}
Hans-Joachim Hein, \emph{Gravitational instantons from rational elliptic
  surfaces}, J. Amer. Math. Soc. \textbf{25} (2012), no.~2, 355--393.

\bibitem[HHM04]{HHM}
Tam\'{a}s Hausel, Eugenie Hunsicker, and Rafe Mazzeo, \emph{Hodge cohomology of
  gravitational instantons}, Duke Math. J. \textbf{122} (2004), no.~3,
  485--548.

\bibitem[Hit87]{Hitchin}
N.~J. Hitchin, \emph{The self-duality equations on a {R}iemann surface}, Proc.
  London Math. Soc. (3) \textbf{55} (1987), no.~1, 59--126. \MR{887284}

\bibitem[Hit97]{Hitchin97}
Nigel~J. Hitchin, \emph{The moduli space of special {L}agrangian submanifolds},
  Ann. Scuola Norm. Sup. Pisa Cl. Sci. (4) \textbf{25} (1997), no.~3-4,
  503--515 (1998).

\bibitem[HSVZ18]{HSVZ}
Hans-Joachim Hein, Song Sun, Jeff Viaclovsky, and Ruobing Zhang,
  \emph{Nilpotent structures and collapsing {R}icci-flat metrics on {K}3
  surfaces}, arXiv.org:1807.09367, 2018.

\bibitem[Huy16]{Huybrechts}
Daniel Huybrechts, \emph{Lectures on {K}3 surfaces}, Cambridge Studies in
  Advanced Mathematics, vol. 158, Cambridge University Press, Cambridge, 2016.

\bibitem[JS19]{JS}
Wangjian {Jian} and Yalong {Shi}, \emph{Global higher order estimates for
  collapsing {C}alabi-{Y}au metrics on elliptic {K}3 surfaces},
  arXiv.org:1909.05521, 2019.

\bibitem[Kod63]{Kodaira1963}
K.~Kodaira, \emph{On compact analytic surfaces. {II}, {III}}, Ann. of Math. (2)
  77 (1963), 563--626; ibid. \textbf{78} (1963), 1--40.

\bibitem[Kro89]{Kronheimer}
P.~B. Kronheimer, \emph{A {T}orelli-type theorem for gravitational instantons},
  J. Differential Geom. \textbf{29} (1989), no.~3, 685--697.

\bibitem[KW11]{KW}
Vitali Kapovitch and Burkhard Wilking, \emph{Structure of fundamental groups of
  manifolds with {R}icci curvature bounded below}, arXiv.org:1105.5955, 2011.

\bibitem[LS94]{LeBrun-Singer}
Claude LeBrun and Michael Singer, \emph{A {K}ummer-type construction of
  self-dual {$4$}-manifolds}, Math. Ann. \textbf{300} (1994), no.~1, 165--180.

\bibitem[McC01]{McCleary}
John McCleary, \emph{A user's guide to spectral sequences}, second ed.,
  Cambridge Studies in Advanced Mathematics, vol.~58, Cambridge University
  Press, Cambridge, 2001.

\bibitem[McL98]{McLean}
Robert~C. McLean, \emph{Deformations of calibrated submanifolds}, Comm. Anal.
  Geom. \textbf{6} (1998), no.~4, 705--747.

\bibitem[Min09]{MinerbeALF}
Vincent Minerbe, \emph{A mass for {ALF} manifolds}, Comm. Math. Phys.
  \textbf{289} (2009), no.~3, 925--955.

\bibitem[Min11]{Minerbe}
\bysame, \emph{Rigidity for multi-{T}aub-{NUT} metrics}, J. Reine Angew. Math.
  \textbf{656} (2011), 47--58.

\bibitem[NT11]{NT1}
Aaron Naber and Gang Tian, \emph{Geometric structures of collapsing
  {R}iemannian manifolds {I}}, Surveys in geometric analysis and relativity,
  Adv. Lect. Math. (ALM), vol.~20, Int. Press, Somerville, MA, 2011,
  pp.~439--466.

\bibitem[NZ16]{NaberZhang}
Aaron Naber and Ruobing Zhang, \emph{Topology and {$\varepsilon$}-regularity
  theorems on collapsed manifolds with {R}icci curvature bounds}, Geom. Topol.
  \textbf{20} (2016), no.~5, 2575--2664.

\bibitem[OO18]{OdakaOshima}
Yuji Odaka and Yoshiki Oshima, \emph{Collapsing {K3} surfaces, tropical
  geometry and moduli compactifications of {S}atake, {M}organ-{S}halen type},
  arXiv.org:1810.07685, 2018.

\bibitem[Pet16]{Petersen}
Peter Petersen, \emph{Riemannian geometry}, third ed., Graduate Texts in
  Mathematics, vol. 171, Springer, Cham, 2016.

\bibitem[PSS71]{ShapiroShafarevitch}
I.~I. Piateskii-Shapiro and I.~R. Shafarevitch, \emph{Torelli's theorem for
  algebraic surfaces of type {${\rm K}3$}}, Izv. Akad. Nauk SSSR Ser. Mat.
  \textbf{35} (1971), 530--572.

\bibitem[PWY99]{PWY}
Peter Petersen, Guofang Wei, and Rugang Ye, \emph{Controlled geometry via
  smoothing}, Comment. Math. Helv. \textbf{74} (1999), no.~3, 345--363.

\bibitem[Shi72]{Shioda1972}
Tetsuji Shioda, \emph{On elliptic modular surfaces}, J. Math. Soc. Japan
  \textbf{24} (1972), 20--59.

\bibitem[Siu83]{Siu}
Yum-Tong Siu, \emph{Every {$K3$} surface is {K}\"{a}hler}, Invent. Math.
  \textbf{73} (1983), no.~1, 139--150.

\bibitem[SZ19]{SZ}
Song Sun and Ruobing Zhang, \emph{Complex structure degenerations and
  collapsing of {C}alabi-{Y}au metrics}, arXiv.org:1906.03368, 2019.

\bibitem[TY90]{TianYau}
Gang Tian and Shing-Tung Yau, \emph{Complete {K}\"ahler manifolds with zero
  {R}icci curvature. {I}}, J. Amer. Math. Soc. \textbf{3} (1990), no.~3,
  579--609.

\bibitem[Yau78]{Yau}
Shing-Tung Yau, \emph{On the {R}icci curvature of a compact {K}\"ahler manifold
  and the complex {M}onge-{A}mp\`ere equation. {I}}, Comm. Pure Appl. Math.
  \textbf{31} (1978), no.~3, 339--411.

\end{thebibliography}

 \end{document}